\documentclass[10pt,reqno]{amsart}
\usepackage{amsmath,amsfonts,amssymb,amscd,amsthm,amsbsy,bbm, epsf,calc}
\usepackage{color}
\usepackage{datetime}

\usepackage{enumerate}
\usepackage{graphicx}
\usepackage{color}

\usepackage{tipa}
\usepackage{relsize}
\usepackage{wasysym}
\usepackage{marvosym}
\usepackage{upgreek}
\usepackage{mathrsfs}
\usepackage{stmaryrd}

\usepackage{amsfonts}
\usepackage{hyperref}
\hypersetup{
	colorlinks = true, 
	linkcolor = red,
	citecolor =  blue
}

\newcommand{\RR}{\mathbb R}
\newcommand{\NN}{\mathbb N}

\newcommand{\ZZ}{\mathbb Z}

\newcommand{\QQ}{\mathbb Q}

\newcommand{\e}{\varepsilon}

\DeclareMathOperator{\argmin}{arg\,min}

\DeclareMathOperator{\tr}{tr}
\DeclareMathOperator{\dist}{dist}
\DeclareMathOperator{\USC}{USC}
\DeclareMathOperator{\LSC}{LSC}

\DeclareMathOperator{\UC}{UC}
\DeclareMathOperator{\DIV}{div}

\def\e{\varepsilon}

\theoremstyle{plain}

\theoremstyle{theorem}
\newtheorem{theorem}{Theorem}[section]
\newtheorem{corollary}[theorem]{Corollary}
\newtheorem{lemma}[theorem]{Lemma}
\newtheorem{proposition}[theorem]{Proposition}
\newtheorem{remark}[theorem]{Remark}

\theoremstyle{definition}
\newtheorem{definition}{Definition}[section]

\numberwithin{equation}{section}

\renewcommand{\leq}{\leqslant}
\renewcommand{\geq}{\geqslant}
\renewcommand{\setminus}{\smallsetminus}

\title[Head and tail speeds of mean curvature flow with forcing]{Head and tail speeds of mean curvature flow with forcing}
\author[Hongwei Gao]{Hongwei Gao}
\address{
Department of Mathematics \\ 
University of California, Los Angeles   \\ 
CA, 90095\\
USA}
\email{hwgao@math.ucla.edu}

\author[Inwon C. Kim]{Inwon Kim}
\address{
	Department of Mathematics \\ 
	University of California, Los Angeles   \\ 
	CA, 90095\\
	USA}
\email{ikim@math.ucla.edu}

\begin{document}

\begin{abstract}
	In this paper, we investigate the large time behavior 
	of interfaces moving with motion law $V = -\kappa + g(x)$, where $g$ is positive, Lipschitz and  $\ZZ^n$-periodic. It turns out that the behavior of the interface can be characterized by its head and tail speed, which depends continuously on its overall direction of propagation $\nu$. If head speed equals tail speed at a given direction $\nu$, the interface has a unique large-scale speed in that direction. In general the interface develops linearly growing ``long fingers" in the direction where the equality breaks down. We discuss these results in both general setting and in laminar setting, where further results are obtained due to regularity properties of the flow. 
\end{abstract}

\maketitle

\tableofcontents
\section{Introduction}

We consider the evolution of domains $(\Omega_{\varepsilon}(t))_{t>0}$ in $\RR^{n}$, where $\Gamma_{\varepsilon}(t) := \partial\Omega_{\varepsilon}(t)$ moves with the (outward) normal velocity
\begin{eqnarray}\label{the normal velocity of moving front in macro scale}
V = -\varepsilon\kappa + g(x/\e) &\text{on}& \Gamma_{\varepsilon}(t).
\end{eqnarray}
Here  $\kappa$ denotes the mean curvature of $\Gamma_{\varepsilon}(t)$, with positive sign when $\Omega_\e(t)$ is convex, $g$ is a $\ZZ^{n}$-periodic function in $\RR^{n}$. Note that $\Gamma_{\varepsilon}(t)$ is a zoomed-out version of $\Gamma_{1}(t)$ with scaling $(x,t) \rightarrow (\varepsilon x, \varepsilon t)$. The oscillation in the forcing term $g$ will be reflected in the oscillatory behavior of $\Gamma_{\varepsilon}$. 

\medskip

We are interested in the asymptotic behavior of $\Gamma_{\varepsilon}$ as $\varepsilon \rightarrow 0$, or equivalently, the large-scale behavior of $\Gamma_{1}$. $\Gamma_{\varepsilon}$ may go through topological changes and other singularities as it intersects with the oscillatory forcing. Thus the evolution \eqref{the normal velocity of moving front in macro scale} must be understood in a weak sense, while for our purpose the weak notation should still be able to describe the pointwise behavior of the solution. To this end we work with viscosity solution $u^\varepsilon$ of the corresponding level set equation (here $\widehat{p} = \frac{p}{|p|}$ for $p \in \RR^{n}\diagdown\left\lbrace 0 \right\rbrace $) with $\Gamma_\e = \left\lbrace u^{\varepsilon} = 0 \right\rbrace$,
\begin{equation}\label{the scaled forced mean curvature flow}
u_{t}^{\varepsilon} =  \mathscr{F}(\e D^{2}u^{\e}, Du^{\e}, x/\e) := \e\tr \left\lbrace  D^{2}u^\e\left( I - \widehat{Du^\e}\otimes \widehat{Du^\e}\right) \right\rbrace + g\left(x/\e\right) |Du^\e| \quad\text{ in }\RR^{n} \times (0,\infty),
\end{equation}
which is a degenerate viscous Hamilton-Jacobi equation. We say \textit{homogenization occurs} when $u^\e$ converges to a homogenization profile as $\e \to 0$. If not we say \textit{homogenization fails}. The study of \eqref{the scaled forced mean curvature flow} as $\varepsilon \rightarrow 0$ has attacted much attention in the past decade, for instance see  \cite{Lions and Souganidis},  \cite{Dirr Karali and Yip},  \cite{Cardaliaguet Lions and Souganidis}, \cite{Cesaroni and Novaga}, \cite{Caffarelli and Monneau ARMA}, \cite{Armstrong and Cardaliaguet} and the references therein.

\medskip
We investigate the case of positive, Lipschitz continuous $g$, where (\ref{the normal velocity of moving front in macro scale}) is most well understood. i.e., there exist $m_0, M_0, L_0 > 0$, such that 
\begin{equation}\label{the hypothesis}
\begin{cases}
g(x): \RR^{n} \rightarrow [m_0, M_0] \text{ is a } \ZZ^{n} \text{-periodic Lipschitz continuous function}\\ 
\text{ and } |g(x) - g(y)| \leq L_0 |x - y|, \hspace{2mm} (x,y) \in \RR^{n} \times \RR^{n}.
\end{cases} \tag{H}
\end{equation}
In this setting  Lions and Souganidis \cite{Lions and Souganidis} showed homogenization results with the condition $\frac{|Dg|}{g^{2}} < \frac{1}{n-1}$. This condition amounts to ensuring the existence of Lipschitz continuous solution $v$ of the corresponding cell problem (see below for further discussion of the cell problem), to ensure the existence of plane-like solutions of the form $u^\e \sim u^0 + \e v + o(\e)$, where $u^0$ is a linear profile. Here the regularity of $v$ is central to obtain Lipschitz continuity of the homogenized front velocity. For two space dimensions Caffarelli and Monneau \cite{Caffarelli and Monneau ARMA} shows that homogenization always occurs, with continuous homogenized velocity. The main step here is to show the exitence of a bounded solution of the cell problem, by a geometric argument that is particular to two dimensions. In general homogenization may fail when the oscillation of $g$ grows large, as we will discuss in the paper (sect. \ref{discussion of construction of travelling waves}). In three or higher dimensions,  \cite{Caffarelli and Monneau ARMA} gave an example in the laminar setting $g(x) = g(x^{\prime})$, $x = (x^{\prime}, x_{n})$, where the oscillation of $\Gamma_\e$ grows linearly as $\e \to 0$.

\medskip

Even when homogenization fails and $\Gamma_\e$ does not approach an asymptotic profile, it is still reasonable to expect its {\it head} and {\it tail}  speeds to homogenize, as the front propagates through the periodic media. Our goal is to describe this behavior of $\Gamma_\e$ as $\e\to 0$ in general setting.  As stated below, these speeds $\bar{s}$ and $\underline{s}$ only depend on the asymptotic direction of propagation $\nu$.

\begin{theorem}[Proposition \ref{the ordering relation in irrational directions}, \ref{the continuity of head and tail speeds}, \ref{the ordering relation in all directions}, Theorem \ref{the macro and micro scale description of head and tail speeds}]\label{main theorem 1}
	Let $\mathbb{S}^{n-1}$ denote the set of unit vectors in $\RR^{n}$. Then there exist two functions $\overline{s}, \underline{s} : \mathbb{S}^{n-1} \rightarrow \left[m_0, M_0\right]$ with the following properties:
	\begin{enumerate}
		\item [(a)] $\overline{s}$ and $\underline{s}$ are continuous in $\mathbb{S}^{n-1}$ and $\overline{s} \geq \underline{s}$. In particular if $\overline{s} > \underline{s}$ at a direction $\nu_{0}$, then the same holds for $\nu$ sufficiently close to $\nu_{0}$.
		\item [(b)] Let $\nu \in \mathbb{S}^{n-1}$ and $u^{\varepsilon}$ solve (\ref{the scaled forced mean curvature flow}) with its initial data $u^{\varepsilon}(x,0) = - (x - x_{0}) \cdot \nu$ for some $x_{0} \in \RR^{n}$. Then in micro-scale (for $\varepsilon = 1$)
		\begin{eqnarray*}
		\overline{s}(\nu) = \lim_{t \rightarrow \infty} \frac{\sup\left\lbrace x\cdot \nu : u^{1}(x,t) = 0\right\rbrace }{t}, && \underline{s}(\nu) = \lim_{t \rightarrow \infty} \frac{\inf\left\lbrace x\cdot \nu : u^{1}(x,t) = 0\right\rbrace }{t},
		\end{eqnarray*}
       and in macro-scale
       \begin{eqnarray*}
       \limsup_{\varepsilon \rightarrow 0}u^{\varepsilon}(x,t) = - (x - x_{0})\cdot \nu + \overline{s}(\nu) t, && \liminf_{\varepsilon \rightarrow 0}u^{\varepsilon}(x,t) = - (x - x_{0})\cdot \nu + \underline{s}(\nu) t.
       \end{eqnarray*}
	\end{enumerate}
    In particular, when $\overline{s}(\nu) > \underline{s}(\nu)$, the set $\{ u^{\varepsilon}(\cdot,t) = 0\} $ oscillates by unit size as $\varepsilon \rightarrow 0$ and thus homogenization fails.
\end{theorem}

Above results state that the \textit{head speed} $\overline{s}$ and the \textit{tail speed} $\underline{s}$ provide a comprehensive description of the asymptotic behavior in the limit $\varepsilon \rightarrow 0$ for the motion law (\ref{the normal velocity of moving front in macro scale}) in all scenarios. Let us mention that if, in addition, we have local regularity properties in micro-scale $\e = 1$, our approach would yield the existence of localized ``pulsating travelling waves" with speeds $\overline{s}$ and $\underline{s}$. This is indeed the case in the laminar setting discussed below.

\medskip
	
For solutions with general initial data, it is more difficult to pinpoint the precise location of the heads and tails of the front in the asymptotic limit $\e\to 0$. However the following holds, which provides in particular the optimal upper and lower bound for the propagation of solutions with general geometry.
\begin{theorem}[Proposition \ref{the sub equation satisfied by the limit}, \ref{the super equation satisfied by the limit}, \ref{homogenization occurs when the head speed and tail speed coincide}, Corollary \ref{the head and tail of front with general initial profile}]\label{main theorem 2}
	Let $u^\e$ solve \eqref{the scaled forced mean curvature flow} with initial data $u_0$ that are uniformly continuous in $\RR^n$. Then 
	$u^{\star}:=\lim\limits_{\e\to 0}\sup^{*} u^\e$ is a viscosity subsolution of $u_t = \overline{s}(-\widehat{Du}) |Du|$. Similarly $u_{\star}:=\lim\limits_{\e\to 0}\inf_{*} u^\e$ is a viscosity supersolution of $u_t = \underline{s}(- \widehat{Du}) |Du|$. 
		
\medskip
		
In particular
		
\begin{itemize}
	\item Let us consider a  collection of points and directions $\mathcal{A}=\{(x_i, \nu_i)\}\subset \RR^n\times \mathbb{S}^{n-1}$, and define the associated convex sets $\mathcal{E} (t):= \inf\limits_{(x_i,\nu_i) \in \mathcal{A}} \{(x-x_i)\cdot\nu_i \leq \overline{s}(\nu)  t \}$. If initially $\{u_0 = 0\} \subset \mathcal{E}(0)$, then $\{u^{\star}(\cdot,t)=0\} \subset \mathcal{E}(t)$.			
	\item If $s=\bar{s}= \underline{s}$, then $u^\e$ uniformly converges to $u$, the unique viscosity solution of $u_t = s(-\widehat{Du})|Du|$ with initial data $u_0$.
\end{itemize}
\end{theorem}

\medskip

Stronger statements are available  in the {\it Laminar setting,} when $g(x)= g(x')$ for $x=(x', x_n)$. For $\e=1$, if we start from a Lipschitz and periodic graph $\Gamma_0=\{(x,x_n): x_n = U^0(x)\}$ that is bounded, then we can show that $\Gamma_1$ stays as a graph and moreover remains as $C^{1,\alpha}$ hypersurface in space, locally uniformly for all large times (see Proposition \ref{the regularity of front in laminar case}). This is sufficient regularity for $\Gamma_1$ to yield the following results.
\begin{theorem}[Theorem \ref{the existence of a travelling wave subsolution and its properties}, \ref{the existence of a travelling wave supersolution and its properties}]\label{main theorem 3}
	Let $g(x)=g(x')$ for $x=(x',x_n)$. Suppose that $\bar{s}(e_n) > \underline{s}(e_n)$. Then there are disjoint, open, non-empty sets $E_1, E_2$ in $\RR^{n-1}$ and functions $U_1: E_1 \to (-\infty, 0]$, $U_2: E_2 \to [0, \infty)$ such that the following is true:
	\begin{itemize}
		\item[(a)] The sets $E_i\times (-\infty, \infty)$ are stationary solutions of \eqref{the normal velocity of moving front in macro scale}.
		\item[(b)]  $U_1\to -\infty$ as $x \to \partial{E_1}$ and $U_2 \to +\infty$ as $x\to \partial{E_2}$.
		\item[(c)] The surfaces $\Gamma_i := \{x_n = U_i(x') + s_i t\}, i = 1,2,$  satisfy \eqref{the normal velocity of moving front in macro scale} with $\e=1$, away from the ``obstacle" $\{x_n = s_i t\}$. (here $s_{1} = \overline{s}(e_{n})$ and $s_{2} = \underline{s}(e_{n})$)
		\item[(d)] $\Gamma_1$ and $\Gamma_2$ are respectively a subsolution and a supersolution of \eqref{the normal velocity of moving front in macro scale} with $\e=1$.
	\end{itemize}
\end{theorem}

For $\nu \neq e_n$ a parallel argument should lead to the existence of pulsating traveling waves away from the obstacles, but we do not pursue this.

\medskip

Our results accompanies that of  Cesaroni and Novaga \cite{Cesaroni and Novaga}, where variational methods were adopted to yield the existence of the maximal traveling wave in the above laminar setting. While our approach allows to describe travelling waves both at maximal and minimal speed, we only recover partial travelling waves away from their highest and lowest positions, as described in $(c)$. In fact in the scenario where there exists multiple localized travelling waves at the same asymptotic speed, let's say $\overline{s}(e_n)$,  our method appears to capture the most external profile of these waves. 

%


In laminar setting, when the oscillation of $g$, $M_0-m_0$, is smaller than a dimensional constant, \cite{Cesaroni and Novaga} shows the existence of global traveling wave solution with a unique speed $\bar{s}(e_n)=\underline{s}(e_n)$, which provides the large-time behavior of graph solutions in the direction of $e_n$. When the oscillation of $g$ is allowed to be large, it is not hard to generate examples of $\bar{s}(e_n) >\underline{s}(e_n)$ following that of \cite{Caffarelli and Monneau ARMA}. We briefly discuss this in section \ref{discussion of construction of travelling waves}.

\medskip

{\bf Main challenges and new ingredients}

\medskip

The central difficulty in obtaining these results is the lack of regularity of the solutions, which comes naturally with the general scenario. In aforementioned literature regarding homogenization of \eqref{the normal velocity of moving front in macro scale}, one starts with an Ansatz $u^{\epsilon}(x,t) = u^0(x,t) + \epsilon v(\frac{x}{\epsilon}) + o(\epsilon)$, where $v$ solves a cell problem given by the limit profile $u^0$, which is, for \eqref{the scaled forced mean curvature flow}, a linear profile $x\cdot\nu - s t$.  The idea is then to look for  $s = s(\nu)$ for which there exists a $\ZZ^n$-periodic solution $v$  of the cell problem
\begin{equation*}
\mathscr{F}(D^2v, \nu + Dv, y) = s \hbox{ in } [0,1]^n, \hbox{ where } \mathscr{F} \hbox{ is as given in } \eqref{the scaled forced mean curvature flow}.
\end{equation*}
The existence of such $v$ is central in establishing homogenization results.

\medskip

In our setting this approach fails to apply for two reasons. First, in our general settings, there may be no global limit profile for $u^{\epsilon}$,  let alone an asymptotic planar profile. Indeed our goal is to look for profiles of limit supremum and limit infimum of $u^{\e}$, as stated above. To study these partial limits, we will introduce ``obstacle cell problems", which amounts to looking for the maximal subsolution and minimal supersolution of a ``cell problem". Second, our ``cell problem" is not the standard cell problem in the sense that the corresponding solutions are not periodic if $\nu$ is irrational. This necessitates formulation of the problem in a bounded domain instead, generating an ``approximate" sub- and super-cell problem (Definition \ref{the definitions of obstacle solutions}).

\medskip

The obstacle approach was first introduced by Caffarelli, Souganidis and Wang \cite{Caffarelli Souganidis and Wang CPAM} for random homogenization of uniformly elliptic PDEs, and later adopted by Kim \cite{Kim ARMA,Kim CPDE} and Po$\check{\text{z}}\acute{\text{a}}$r \cite{Pozar 2015} for free boundary problems. In both of these results the common feature is that there are no standard cell problems one can expect to solve, either due to the non-periodic environment or non-periodic evolution of the free boundaries. This corresponds to our second difficulty described above. However in all of the aforementioned results homogenization is expected to hold: indeed the obstacle solutions in these settings turn out to be asymptotically regular.  Our contribution in this paper is thus introducing a ``cell problem" type approach for a problem where homogenization is not expected to occur in general, or more precisely when large-scale regularity is missing for the $\e$-solutions.

\medskip

Roughly speaking the obstacle solutions  solve \eqref{the normal velocity of moving front in macro scale}  with the constraint for the solutions to be below or above the planar obstacle $x\cdot\nu - st$. For instance $\overline{s}(\nu)$ is then obtained as the largest speed for which the solutions put below the obstacle stay close to it, which is what is expected for the head speed of an oscillatory interface.  We observe that, when $\nu$ is {\it irrational} i.e. if $\nu \notin \RR\ZZ^{n}$, this approach has the advantage of introducing a fine-scale dynamic recurrence property to the problem (Proposition \ref{a lattice point that is close to a hyperplane}), which  compensates for the lack of regularity properties to study its large-scale behavior. A more precise form of this observation is formulated in the {\it local comparison} (Proposition \ref{LCP in the unit scale}), which is an important new ingredient in our analysis. This theorem, of independent interest, localizes obstacle solutions of the curvature flow \eqref{the normal velocity of moving front in macro scale} which are only continuous. Such localization procedure is central in showing qualitative properties of the head and tail speeds, such as linear detachment, continuity and fingering (see e.g. Propositions \ref{the detachment lemma supersolution}, \ref{the ordering relation in irrational directions}, \ref{the continuity of head and tail speeds}, \ref{the repeated patten regarding the head and tail of the oscillation of the real solution in a thin strip}). 

\medskip

Our framework is rather general, and we expect that it could be used to study other geometric flows where homogenization does not always hold. In particular we plan to pursue the case when $g$ changes sign, where there is an added feature of {\it a trapping zone}, where $u^{\epsilon}$ converges to its initial data as $\e \to 0$. See \cite{Cardaliaguet Lions and Souganidis} for illuminating discussions of this phenomena. Technically speaking there are added challenges. For instance when $g$ is positive, $u^{\epsilon}$ with affine initial data turns out to be monotone increasing in time. This adds additional stability in the evolution which is useful in our analysis. Still at the heuristic level our approach should apply to this case. In particular we believe that Theorem \ref{main theorem 4} should still apply to the general, sign-changing $g$. 

\medskip

{\bf Outline of the paper}

\medskip

We start with formulation of obstacle solutions in Section \ref{obstacle problems}, with their properties. In particular the recurrence property mentioned above is given as the Birkhoff property in Section \ref{basic properties of obstacle solutions}. In Section \ref{the inf convolution} we introduce a local perturbation of solutions that was inspired from its usage in free boundary problems (see \cite{Athanasopoulos Caffarelli and Salsa} and \cite{Choi Jerison and Kim}). Section \ref{the section of LCP} proves local comparison principle in terms of the obstacle semi-solutions with irrational directions. To show this, we use the discrepancy results in Section \ref{the subsection of discrepancy} to show that the Birkhoff property leads to a fine-scale recurrence property for irrational directions. Then we prove the local comparison principle (Proposition \ref{LCP in the unit scale} in Section \ref{the subsection of LCP}), using this property as well as the local perturbation introduced in Section \ref{the inf convolution}. Similar results are available in \cite{Kim ARMA, Kim CPDE, Pozar 2015}, however in our problem neither large scale regularity nor perturbation parameters exist.  Both of these facts lead to significant challenges in the proof.  In Section \ref{the section of head and tail speeds} we define $\bar{s}$ and $\underline{s}$ based on the detachment of solutions from the obstacles (Definition \ref{the head speed in an irrational direction} - \ref{the tail speed in an irrational direction} in Section \ref{the subsection of irrational directions}), and use approximation by irrational directions to show continuity of these functions at all directions, based both local comparison (Proposition \ref{the semicontinuity of the head and tail speed}) and a blow-up argument using global solutions (Proposition \ref{the continuity of head and tail speeds}). Section \ref{the section of homogenization} and \ref{the section of nonhomogenization} contains the proof our main results, Theorem \ref{main theorem 1} and Theorem \ref{main theorem 2}.  Lastly Section \ref{the section of laminar forcing term} discusses the Laminar case, where Theorem \ref{main theorem 3} is proved. We finish with Section \ref{discussion of construction of travelling waves} where some scenarios are discussed under which homogenization fails.

\section{Obstacle problems}\label{obstacle problems}
In this section, we introduce the obstacle problem associated to the forced mean curvature flow (\ref{the scaled forced mean curvature flow}) with $\varepsilon = 1$. In later sections, it allows us to analyze the homogenization in each direction independently. The role an obstacle problem plays here is similar to that of the usual cell problem in homogenization problems. Therefore, the obstacle problem here can be regarded as a variant version of the cell problem.

\subsection{Setup}

Let us denote by $\mathscr{F}$ the operator regarding space derivatives in the equation (\ref{the scaled forced mean curvature flow}) with $\varepsilon = 1$:
\begin{equation}\label{the differential operator regarding space derivatives}
\mathscr{F}\left( D^{2}u, Du, x\right) := \tr \left\lbrace  D^{2}u\left( I - \widehat{Du}\otimes \widehat{Du}\right) \right\rbrace + g\left(x\right) |Du|.
\end{equation}

\begin{definition}[c.f. \cite{Caffarelli and Monneau ARMA}]
	Let $\mathcal{S}^{n}$ be the set of all $n \times n$ symmetric matrices and denote $\mathscr{D}_{0} := \mathcal{S}^{n} \times \left( \RR^{n}\diagdown \left\lbrace 0\right\rbrace \right) \times \RR^{n}$. We define for all $\left( X, p, x\right) \in \mathcal{S}^{n} \times \RR^{n} \times \RR^{n}$:
	\begin{eqnarray*}
	\mathscr{F}^{*}(X,p,x) &:=& \limsup_{\eta \rightarrow 0}\left\lbrace \mathscr{F}\left( Y, q, y\right) \big| \left( Y, q, y\right) \in \mathscr{D}_{0}, \hspace{2mm} |X - Y|, |p - q|, |x - y| \leq \eta \right\rbrace, \\
	\mathscr{F}_{*}(X,p,x) &:=& \liminf_{\eta \rightarrow 0}\left\lbrace \mathscr{F}\left( Y, q, y\right) \big| \left( Y, q, y\right) \in \mathscr{D}_{0}, \hspace{2mm} |X - Y|, |p - q|, |x - y| \leq \eta \right\rbrace. 
	\end{eqnarray*}
    In particular, we have $\mathscr{F}^{*}(X,p,x) = \mathscr{F}_{*}(X,p,x) = \mathscr{F}(X,p,x)$ for $(X, p ,x) \in \mathscr{D}_{0}$.
\end{definition}

\begin{definition}[c.f. \cite{Crandall Ishii Lions BAMS, Caffarelli and Monneau ARMA}]\label{the definition of viscosity subsolution}
	Let $\Omega \subseteq \RR^{n} \times (-\infty,\infty)$ and $u(x,t) \in \USC(\Omega)$, the space of upper semicontinuous functions over $\Omega$. Then $u(x,t)$ is called a \textit{viscosity subsolution} in $\Omega$, which is denoted as follows
	\begin{eqnarray}\label{the viscosity subsolution}
	u_{t} \leq \mathscr{F}\left( D^{2}u, Du, x\right), && (x,t) \in \Omega,
	\end{eqnarray}
    if for any $(x_{0},t_{0}) \in \Omega$, $r > 0$ and $\phi(x,t) \in C^{2,1}\left( B_{r}(x_{0},t_{0})\right)$, such that
    \begin{eqnarray*}
    u(x,t) \leq \phi(x,t) \hspace{2mm} \text{in} \hspace{2mm} B_{r}(x_{0},t_{0}) &\text{and}& u(x_{0},t_{0}) = \phi(x_{0},t_{0}),
    \end{eqnarray*}
   then
   \begin{equation*}
   \phi_{t}(x_{0},t_{0}) \leq \mathscr{F}^{*}\left( D^{2}\phi(x_{0},t_{0}), D\phi(x_{0},t_{0}), x_{0}\right).
   \end{equation*}
\end{definition}

\begin{definition}\label{pseudo viscosity subsolution}
	Let $\Omega \subseteq \RR^{n} \times (-\infty, \infty)$ and $u(x,t) \in \USC(\Omega)$. Then $u(x,t)$ is called a \textit{pseudo viscosity subsolution} in $\Omega$, if for any $(x_{0},t_{0}) \in \Omega$, $r > 0$ and $\phi(x,t) \in C^{2,1}\left( B_{r}(x_{0},t_{0})\right)$, such that
	\begin{eqnarray*}
		u(x,t) \leq \phi(x,t) \hspace{2mm} \text{on} \hspace{2mm} B_{r}(x_{0},t_{0}), \hspace{2mm} u(x_{0},t_{0}) = \phi(x_{0},t_{0}) &\text{and}& |D\phi(x_{0},t_{0})| > 0,
	\end{eqnarray*}
	then
	\begin{equation*}
	\phi_{t}(x_{0},t_{0}) \leq \mathscr{F}\left( D^{2}\phi(x_{0},t_{0}), D\phi(x_{0},t_{0}), x_{0}\right). 
	\end{equation*}
\end{definition}

\begin{definition}[c.f. \cite{Crandall Ishii Lions BAMS, Caffarelli and Monneau ARMA}]
	Let $\Omega \subseteq \RR^{n} \times (-\infty, \infty)$ and $v(x,t) \in \LSC(\Omega)$, the space of lower semicontinuous functions. Then $u(x,t)$ is called a \textit{viscosity supersolution} in $\Omega$, which is denoted as follows
	\begin{eqnarray}\label{the viscosity supersolution}
		v_{t} \geq \mathscr{F}\left( D^{2}v, Dv, x\right), && (x,t) \in \Omega,
	\end{eqnarray}
	if for any $(x_{0},t_{0}) \in \Omega$, $r > 0$ and $\psi(x,t) \in C^{2,1}\left( B_{r}(x_{0},t_{0})\right)$, such that
	\begin{eqnarray*}
		v(x,t) \geq \psi(x,t) \hspace{2mm} \text{in} \hspace{2mm} B_{r}(x_{0},t_{0}) &\text{and}& v(x_{0},t_{0}) = \psi(x_{0},t_{0}),
	\end{eqnarray*}
	then
	\begin{equation*}
	\psi_{t}(x_{0},t_{0}) \geq \mathscr{F}_{*}\left( D^{2}\psi(x_{0},t_{0}), D\psi(x_{0},t_{0}), x_{0}\right). 
	\end{equation*}
\end{definition}

\begin{definition}\label{pseudo viscosity supersolution}
	Let $\Omega \subseteq \RR^{n} \times (-\infty, \infty)$ and $v(x,t) \in \LSC(\Omega)$. Then $u(x,t)$ is called a \textit{pseudo viscosity supersolution} in $\Omega$, if for any $(x_{0},t_{0}) \in \Omega$, $r > 0$ and $\psi(x,t) \in C^{2,1}\left( B_{r}(x_{0},t_{0})\right)$, such that
	\begin{eqnarray*}
		v(x,t) \geq \psi(x,t) \hspace{2mm} \text{in} \hspace{2mm} B_{r}(x_{0},t_{0}), \hspace{2mm} v(x_{0},t_{0}) = \psi(x_{0},t_{0}) &\text{and}& |D\psi(x_{0},t_{0})| > 0, 
	\end{eqnarray*}
	then
	\begin{equation*}
	\psi_{t}(x_{0},t_{0}) \geq \mathscr{F}\left( D^{2}\psi(x_{0},t_{0}), D\psi(x_{0},t_{0}), x_{0}\right). 
	\end{equation*}
\end{definition}

\begin{definition}[c.f. \cite{Crandall Ishii Lions BAMS, Caffarelli and Monneau ARMA}]
	Let $\Omega \subseteq \RR^{n} \times (-\infty, \infty)$ and $u(x,t): \Omega \rightarrow \RR$. Then $u(x,t)$ is called a \textit{viscosity solution} if $u^{*}(x,t)$ is a viscosity subsolution and $u_{*}(x,t)$ is a viscosity supersolution,
	where
	\begin{eqnarray*}
	u^{*}(x,t) := \limsup_{(y,\tau) \rightarrow (x,t)} u(y,\tau) &\text{and}& u_{*}(x,t) := \liminf_{(y,\tau) \rightarrow (x,t)} u(y,\tau).
	\end{eqnarray*}
	It is well-known that for any $\varepsilon > 0$ and $u_{0}(x) \in \UC(\RR^{n})$, the equation (\ref{the scaled forced mean curvature flow}) has a unique continuous viscosity solution.
\end{definition}

\begin{proposition}[Comparison principle, see \cite{Caffarelli and Monneau ARMA}]\label{the usual comparison principle}
	Let us consider $\Omega = \hat{\Omega} \times (0,T)$ with $T > 0$, where $\hat{\Omega} \subseteq \RR^{n}$. Assume that either $\hat{\Omega} = \RR^{n}$ or $\hat{\Omega}$ is a bounded open subset of $\RR^{n}$, assume that $u(x,t)$ is a viscosity subsolution of (\ref{the viscosity subsolution}) and $v(x,t)$ is a viscosity supersolution of (\ref{the viscosity supersolution}) such that
	\begin{equation*}
	\begin{cases}
	\limsup\limits_{\delta \rightarrow 0}\left\lbrace u(x,0) - v(y,0) \big||x - y| \leq \delta \right\rbrace \leq 0, & \text{ if } \hspace{2mm} \hat{\Omega} = \RR^{n}\\
	u \leq v \hspace{2mm} \text{ on } \hspace{2mm} \partial_{p}\left( \hat{\Omega} \times (0, T)\right),  & \text{ if } \hspace{2mm} \hat{\Omega} \text{ is bounded}, 
	\end{cases}
	\end{equation*}
    then
    \begin{eqnarray*}
    u(x,t) \leq v(x,t), && (x,t) \in \Omega.
    \end{eqnarray*}
\end{proposition}

\begin{definition}
	Let us denote some frequently used sets throughout the paper. 
	\begin{equation}
	\mathbb{D} := \mathbb{S}^{n-1} \times (0,\infty) \times \RR \tag{i}
	\end{equation}
	\begin{equation}
	\mathbb{E} := \left\lbrace (\nu,q,s) \in \mathbb{S}^{n-1} \times \left( \RR^{n}\diagdown \left\lbrace 0\right\rbrace \right) \times \left[m_0, M_0\right] \Big| \nu = -\frac{q}{|q|}\right\rbrace  \tag{ii}
	\end{equation}
	\begin{equation}
	\mathbb{F} := \left\lbrace \left( r, \varphi\right) \big|r(t): [0, \infty) \rightarrow (0, \infty), \hspace{2mm} \varphi(x): \RR^{n} \rightarrow (0, \infty) \right\rbrace \tag{iii}
	\end{equation}
	\begin{equation}
	\mathbb{A} := \left\lbrace (\nu, R, \mathscr{R}, q, s) \in \mathbb{S}^{n-1} \times (0,\infty) \times \RR \times \left( \RR^{n}\diagdown \left\lbrace 0\right\rbrace \right) \times \left[m_0, M_0 \right]\Big| \nu = -\frac{q}{|q|} \right\rbrace  \tag{iv}
	\end{equation}
\end{definition}

\begin{definition}\label{the definition of the domain}
	Fix any $d := (\nu, R, \mathscr{R}) \in \mathbb{D}$, denote by $\mathrm{C}_{\mathbb{d}}(t)$ the $\nu$ directional cylinder with initial radius $R$ and expanding/shrinking speed $\mathscr{R}$ at time $t$. i.e.,
	\begin{equation*}
	\mathrm{C}_{d}(t) := \left\lbrace x \in \RR^{n} \big| |x - (x\cdot\nu)\nu| < R + \mathscr{R}t\right\rbrace,
	\end{equation*}
    where $R + \mathscr{R}t > 0$. Let us also denote the whole space-time domain by that
    \begin{equation}\label{the space time domain}
    \mathrm{C}_{d} := \left\lbrace (x, t) \in \RR^{n} \times [0, \infty) \big| x \in \mathrm{C}_{d}(t), \hspace{2mm} R + \mathscr{R}t > 0\right\rbrace. 
    \end{equation}
    In particular, let $(x,r,\nu) \in \RR^{n} \times (0,\infty) \times \mathbb{S}^{n-1}$ and denote a static region as follows,
    \begin{equation}\label{thin cylinder pointing to a specific direction}
    \Upomega(x,r;\nu) := \left\lbrace y \in \RR^{n} \big| \left| (y - x) - \left( (y - x)\cdot \nu\right)\nu \right| \leq r \right\rbrace.
    \end{equation}
\end{definition}
\begin{definition}
	Fix any $e := (\nu, q, s) \in \mathbb{E}$, we denote by $\mathrm{O}_{e}(x,t)$ the obstacle function with slope $q$ and speed $s$ in the $\nu$ direction. To be more precise,
	\begin{eqnarray}\label{the definition of the obstacle function}
	\mathrm{O}_{e}(x,t) := x \cdot q + st |q|, &\text{for}& x\in \RR^{n} \hspace{2mm}\text{and}\hspace{2mm} t \geq 0.
	\end{eqnarray}
\end{definition}

\begin{remark}
	Let $e = (\nu, q, s) \in \mathbb{E}$, then the zero level set of $\mathrm{O}_{e}(x,t)$ is a hyperplane moving with speed $s$ in the normal direction $\nu$.
\end{remark}

\begin{definition}\label{the definitions of obstacle solutions}
	Fix any $a := (\nu, R, \mathscr{R}, q, s) \in \mathbb{A}$, then set $d := (\nu, R, \mathscr{R}) \in \mathbb{D}$ and $e := (\nu, q, s) \in \mathbb{E}$, let us denote by $\overline{\mathscr{S}}_{a}$ (resp. $\underline{\mathscr{S}}_{a}$) the set of all subsolutions (resp. supersolutions) in $\mathrm{C}_{d}$ that is bounded from above (resp. below) by $\mathrm{O}_{e}(x,t)$. i.e.,
	\begin{eqnarray*}
	\overline{\mathscr{S}}_{a} &:=& \left\lbrace u(x,t) \in \USC(\mathrm{C}_{d})\big | \hspace{2mm} u_{t} \leq \mathscr{F}\left(D^{2}u, Du, x \right), \hspace{2mm} u(x,t) \leq \mathrm{O}_{e}(x,t) \right\rbrace, \\
	\underline{\mathscr{S}}_{a} &:=& \left\lbrace u(x,t) \in \LSC(\mathrm{C}_{d}) \big | \hspace{2mm} u_{t} \geq \mathscr{F}\left(D^{2}u, Du, x \right), \hspace{2mm} u(x,t) \geq \mathrm{O}_{e}(x,t) \right\rbrace.
	\end{eqnarray*}
	Let us also denote the obstacle subsolution/supersolution as follows.
	\begin{eqnarray*}
	\overline{\mathrm{U}}_{a}(x,t) := \left( \sup \left\lbrace u(x,t) \big| u\in \overline{\mathscr{S}}_{a}\right\rbrace\right)^{*} &\text{and}& \underline{\mathrm{U}}_{a}(x,t) := \left( \inf \left\lbrace u(x,t) \big| u\in \underline{\mathscr{S}}_{a}\right\rbrace \right)_{*}
	\end{eqnarray*}
\end{definition}

\subsection{Properties}\label{basic properties of obstacle solutions}

\begin{lemma}
	Fix any $a \in \mathbb{A}$, then
	\begin{eqnarray*}
		\overline{\mathrm{U}}_{a}(x,t) \in \overline{\mathscr{S}}_{a} &\text{and}&  \underline{\mathrm{U}}_{a}(x,t) \in \underline{\mathscr{S}}_{a}.
	\end{eqnarray*}
\end{lemma}

\begin{proof}
	If follows from the definition of viscosity sub/super-solution (c.f. \cite{Crandall Ishii Lions BAMS}).
\end{proof}

\subsubsection{Coincidence on the boundary}

The following Lemma shows that the obstacle subsolution coincides with the obstacle if the domain is not shrinking.

\begin{lemma}\label{the obstacle subsolution coincides with the obstacle on the boundary}
	Fix $a := (\nu, R, \mathscr{R}, q, s) \in \mathbb{A}$ with $\mathscr{R} \geq 0$, then set $d := (\nu, R, \mathscr{R}) \in \mathbb{D}$ and $e := (\nu, q, s) \in \mathbb{E}$, then
	\begin{eqnarray*}
		\overline{\mathrm{U}}_{a}(x,t) = \mathrm{O}_{e}(x,t), && (x, t) \in \left\lbrace (y, \tau)\big |y \in \partial \mathrm{C}_{d}(\tau), \hspace{2mm} y\cdot\nu = s\tau\right\rbrace.
	\end{eqnarray*}
\end{lemma}

\begin{proof}
	Let us denote the set of admissible normal directions:
	\begin{eqnarray*}
	\overline{M}_{a} := \left\lbrace \mu \in \mathbb{S}^{n-1} \big| \mu \cdot \nu = \sigma\right\rbrace, &\text{where}& \sigma := \frac{m_0s + \mathscr{R}\sqrt{\mathscr{R}^{2} + s^{2} - m_0^{2}}}{\mathscr{R}^{2} + s^{2}} > 0.
	\end{eqnarray*}
    Then for any $\mu \in \overline{M}_{a}$, let us define the moving hyperplane
    \begin{equation*}
    \overline{V}_{\mu}(x,t) := - \frac{|q|}{\sigma}\left( x\cdot \mu - m_0t + R\sqrt{1 - \sigma^{2}} \right),
    \end{equation*}
    and a specific subsolution $\overline{V}_{a}(x,t)$ in $\mathrm{C}_{d}$ as below.
    \begin{equation*}
    \overline{V}_{a}(x,t) := \sup_{\mu \in \overline{M}_{a}} \overline{V}_{\mu}(x,t) \in \overline{\mathscr{S}}_{a}.
    \end{equation*}
   Based on the above construction, we have that
	\begin{eqnarray*}
		\overline{V}_{a}(x,t) = \mathrm{O}_{e}(x,t), && (x, t) \in \left\lbrace (y, \tau)\big |y \in \partial \mathrm{C}_{d}(\tau), \hspace{2mm} y\cdot\nu = s\tau\right\rbrace.
	\end{eqnarray*}
    The result follows from the ordering relation $\overline{V}_{a}(x,t) \leq \overline{\mathrm{U}}_{a}(x,t) \leq \mathrm{O}_{e}(x,t)$.   
\end{proof}


In a similar manner, the next Lemma says that if the domain's expanding speed is large enough, the obstacle supersolution matches the obstacle on the boundary.

\begin{lemma}\label{the obstacle supersolution coincides with the obstacle on the boundary}
	Fix $a := (\nu, R, \mathscr{R}, q, s) \in \mathbb{A}$ with $\mathscr{R} \geq \sqrt{M_0^{2} - s^{2}}$, then set $d := (\nu, R, \mathscr{R}) \in \mathbb{D}$ and $e := (\nu, q, s) \in \mathbb{E}$, then
	\begin{eqnarray*}
		\underline{\mathrm{U}}_{a}(x,t) = \mathrm{O}_{e}(x,t), && (x, t) \in \left\lbrace (y, \tau)\big |y \in \partial \mathrm{C}_{d}(\tau), \hspace{2mm} y\cdot\nu = s\tau\right\rbrace. 
	\end{eqnarray*}
\end{lemma}

\begin{proof}
	Let us denote the set of admissible normal directions
	\begin{eqnarray*}
	\underline{M}_{a} := \left\lbrace \mu \in \mathbb{S}^{n-1} \big | \nu \cdot \mu = \sigma\right\rbrace, &\text{where}& \sigma := \frac{s}{\sqrt{\mathscr{R}^{2} + s^{2}}} > 0.
	\end{eqnarray*}
	Then for any $\mu \in \underline{M}_{a}$, let us define the moving hyperplane
	\begin{equation*}
	\underline{V}_{\mu}(x,t) := \left\lbrace x \cdot \mu - \left[ \frac{R\mathscr{R}\sigma}{s} + \sqrt{\mathscr{R}^{2} + s^{2}}t\right] \right\rbrace \cdot \left( -|q|\sigma\right) \in \underline{\mathscr{S}}_{a}, 
	\end{equation*}
	and a specific supersolution $\underline{V}_{a}(x,t)$ in $\mathrm{C}_{d}$ as below:
	\begin{equation*}
	\underline{V}_{a}(x,t) := \inf_{\mu \in \underline{M}_{a}} \underline{V}_{\mu}(x,t) \in \underline{\mathscr{S}}_{a}.
	\end{equation*}
	Based on the above construction, we have that
	\begin{eqnarray*}
		\underline{V}_{a}(x,t) = \mathrm{O}_{e}(x,t), && (x, t) \in \left\lbrace (y, \tau)\big |y \in \partial \mathrm{C}_{d}(\tau), \hspace{2mm} y\cdot\nu = s\tau\right\rbrace, 
	\end{eqnarray*}
	the result follows from the ordering relation $ \mathrm{O}_{e}(x,t)\leq \underline{\mathrm{U}}_{a}(x,t) \leq \underline{V}_{a}(x,t)$.   
\end{proof}

\subsubsection{The Birkhoff properties}

The Birkhoff property describes the monotonicity of a specific obstacle sub/super-solution with respect to time, under certain interger vector shift. The monotonicity depends on two aspects: (i) subsolution or supersolution; (ii) expanding domain or shrinking domain. Let us discuss each of them respectively.\\
In the expanding domain, the obstacle sub/super-solution tends to keep away from the obstacle as time evolves. Therefore, the obstacle subsolution (resp. supersolution) shows a decreasing (resp. an increasing) pattern.

\begin{proposition}\label{the birkhoff property for subsolution in an expanding domain}
	Fix $a := (\nu, R, \mathscr{R}, q, s) \in \mathbb{A}$ with $\mathscr{R} \geq 0$, then set $d := (\nu, R, \mathscr{R}) \in \mathbb{D}$ and $e := (\nu, q, s) \in \mathbb{E}$. Let $\Delta t > 0$ and $\Delta z \in \ZZ^{n}$, such that
	\begin{eqnarray*}
	0 < s\Delta t \leq \Delta z \cdot \nu &\text{and}& \mathscr{R}\Delta t \geq \left| \Delta z - \left( \Delta z \cdot \nu\right) \nu\right|, 
	\end{eqnarray*}
    then
    \begin{eqnarray*}
    \overline{\mathrm{U}}_{a}(x + \Delta z, t + \Delta t) \leq \overline{\mathrm{U}}_{a}(x, t), && (x,t) \in \mathrm{C}_{d}.
    \end{eqnarray*}
\end{proposition}

\begin{proof}
	By the choice of $\Delta t$ and $\Delta z$, $(x,t) \in \mathrm{C}_{d}$ indicates $(x + \Delta z, t + \Delta t) \in \mathrm{C}_{d}$. Moreover, $\overline{\mathrm{U}}_{a}(x + \Delta z, t + \Delta t) \leq \mathrm{O}_{e}(x,t)$, for any $(x,t) \in \mathrm{C}_{d}$. Because $\overline{\mathrm{U}}_{a}(x,t) \in \overline{\mathscr{S}}_{a}$ and $\Delta z \in \ZZ^{n}$, $\overline{\mathrm{U}}_{a}(\cdot + \Delta z, \cdot + \Delta t)\big|_{\mathrm{C}_{d}} \in \overline{\mathscr{S}}_{a}$. Hence the maximality of $\overline{\mathrm{U}}_{a}(x,t)$ from the Definition \ref{the definitions of obstacle solutions} implies that $\overline{\mathrm{U}}_{a}(x + \Delta z, t + \Delta t) \leq \overline{\mathrm{U}}_{a}(x, t)$, for any $(x,t) \in \mathrm{C}_{d}$.
\end{proof}

\begin{proposition}\label{the birkhoff property for supersolution in an expanding domain}
	Fix $a := (\nu, R, \mathscr{R}, q, s) \in \mathbb{A}$ with $\mathscr{R} \geq 0$, then set $d := (\nu, R, \mathscr{R}) \in \mathbb{D}$ and $e := (\nu, q, s) \in \mathbb{E}$. Let $\Delta t > 0$ and $\Delta z \in \ZZ^{n}$, such that
	\begin{eqnarray*}
		s\Delta t \geq \Delta z \cdot \nu > 0 &\text{and}& \mathscr{R}\Delta t \geq \left| \Delta z - \left( \Delta z \cdot \nu\right) \nu\right| ,
	\end{eqnarray*}
	then
	\begin{eqnarray*}
		 \underline{\mathrm{U}}_{a}(x, t) \leq \underline{\mathrm{U}}_{a}(x + \Delta z, t + \Delta t), && (x,t) \in \mathrm{C}_{d}.
	\end{eqnarray*}
\end{proposition}

\begin{proof}
	By the choice of $\Delta t$ and $\Delta z$, if we have $(x, t) \in \mathrm{C}_{d}$, so does $(x + \Delta z, t + \Delta t)$. In addition, $\mathrm{O}_{e}(x,t) \leq \underline{\mathrm{U}}_{a}(x + \Delta z, t + \Delta t)$, for any $(x,t) \in \mathrm{C}_{d}$. Since $\underline{\mathrm{U}}_{a}(x,t) \in \underline{\mathscr{S}}_{a}$ and $\Delta z \in \ZZ^{n}$, $\underline{\mathrm{U}}_{a}(\cdot + \Delta z, \cdot + \Delta t) \in \underline{\mathscr{S}}_{a}$. The minimality of $\underline{\mathrm{U}}_{a}(\cdot,\cdot)$ from the Definition \ref{the definitions of obstacle solutions} implies that $\underline{\mathrm{U}}_{a}(x,t) \leq \underline{\mathrm{U}}_{a}(x + \Delta z, t + \Delta t)$, for any $(x,t) \in \mathrm{C}_{d}$.
\end{proof}

Next, let us investigate the case of static domains, i.e., $\mathscr{R} = 0$. In the following two propositions, we shall compare the sub/super-solutions in two different static domains. It turns out that the larger the domain is, the further the sub/super-solutions stay away from the associated obstacles.

\begin{proposition}\label{the birkhoff property for subsolutions in two static domains}
	Fix $a_{i} := (\nu, R_{i}, 0, q, s) \in \mathbb{A}$, where $i = 1, 2$ and $0 < R_{1} < R_{2} < \infty$, then set $d_{i} := (\nu, R_{i}, 0) \in \mathbb{D}$, $i = 1, 2$ and $e := (\nu, q, s) \in \mathbb{E}$. Let $\Delta t \geq 0$ and $\Delta z \in \ZZ^{n}$, such that
	\begin{eqnarray*}
	0 < s\Delta t \leq \Delta z \cdot \nu &\text{and}& R_{2} - R_{1} \geq |\Delta z - (\Delta z \cdot \nu) \nu|,
 	\end{eqnarray*}
    then
    \begin{eqnarray*}
    \overline{\mathrm{U}}_{a_{2}}(x + \Delta z, t + \Delta t) \leq \overline{\mathrm{U}}_{a_{1}}(x,t), && (x,t) \in \mathrm{C}_{d_{1}}.
    \end{eqnarray*}
\end{proposition}

\begin{proof}
	By the choice of $R_{1}$, $R_{2}$, $\Delta t$ and $\Delta z$, $(x,t) \in \mathrm{C}_{d_{1}}$ indicates $(x + \Delta z, t + \Delta t) \in \mathrm{C}_{d_{2}}$. Moreover, $\overline{\mathrm{U}}_{a_{2}}(x + \Delta z, t + \Delta t) \leq \mathrm{O}_{e}(x,t)$, for any $(x,t) \in \mathrm{C}_{d_{1}}$. Because $\overline{\mathrm{U}}_{a_{2}}(\cdot,\cdot) \in \overline{\mathscr{S}}_{a_{2}}$ and $\Delta z \in \ZZ^{n}$, $\overline{\mathrm{U}}_{a_{2}}(\cdot + \Delta z, \cdot + \Delta t)\big|_{\mathrm{C}_{d_{1}}} \in \overline{\mathscr{S}}_{a_{1}}$. Hence the maximality of $\overline{\mathrm{U}}_{a_{1}}(x,t)$ the Definition \ref{the definitions of obstacle solutions} implies that $\overline{\mathrm{U}}_{a_{2}}(x + \Delta z, t + \Delta t) \leq \overline{\mathrm{U}}_{a_{1}}(x, t)$, for any $(x,t) \in \mathrm{C}_{d_{1}}$.
\end{proof}

\begin{proposition}\label{the birkhoff property for supersolutions in two static domains}
	Fix $a_{i} := (\nu, R_{i}, 0, q, s) \in \mathbb{A}$, where $i = 1, 2$ and $0 < R_{1} < R_{2} < \infty$, then set $d_{i} := (\nu, R_{i}, 0) \in \mathbb{D}$, $i = 1, 2$ and $e := (\nu, q, s) \in \mathbb{E}$. Let $\Delta t \geq 0$ and $\Delta z \in \ZZ^{n}$, such that 
	\begin{eqnarray*}
	s \Delta t \geq \Delta z \cdot \nu > 0 &\text{and}& R_{2} - R_{1} \geq |\Delta z - (\Delta z \cdot \nu)\nu|,
	\end{eqnarray*}
    then
    \begin{eqnarray*}
    \underline{\mathrm{U}}_{a_{1}}(x,t) \leq \underline{\mathrm{U}}_{a_{2}}(x + \Delta z, t + \Delta t), && (x,t) \in \mathrm{C}_{d_{1}}.
    \end{eqnarray*}
\end{proposition}

\begin{proof}
	By the choice of $R_{1}$, $R_{2}$, $\Delta t$ and $\Delta z$, if we have $(x, t) \in \mathrm{C}_{d_{1}}$, then $(x + \Delta z, t + \Delta t) \in \mathrm{C}_{d_{2}}$. In addition, $\mathrm{O}_{e}(x,t) \leq \underline{\mathrm{U}}_{a_{2}}(x + \Delta z, t + \Delta t)$, for any $(x,t) \in \mathrm{C}_{d_{1}}$. Since $\underline{\mathrm{U}}_{a_{2}}(x,t) \in \underline{\mathscr{S}}_{a_{2}}$ and $\Delta z \in \ZZ^{n}$, $\underline{\mathrm{U}}_{a_{2}}(\cdot + \Delta z, \cdot + \Delta t) \in \underline{\mathscr{S}}_{a_{2}}$. The minimality of $\underline{\mathrm{U}}_{a_{1}}(\cdot,\cdot)$ from Definition \ref{the definitions of obstacle solutions} implies that $\underline{\mathrm{U}}_{a_{1}}(x,t) \leq \underline{\mathrm{U}}_{a_{2}}(x + \Delta z, t + \Delta t)$, for any $(x,t) \in \mathrm{C}_{d_{1}}$.
\end{proof}

Finally, in the case of shrinking domains, we have the monotonicity with an opposite direction. i.e., as time passes by, the obstacle sub/super-solutions tend to stay closer to the associated obstacle.

\begin{proposition}\label{the birkhoff property for subsolution in a shrinking domain}
	Fix $a := (\nu, R, \mathscr{R}, q, s) \in \mathbb{A}$ with $\mathscr{R} < 0$, then set $d := (\nu, R, \mathscr{R}) \in \mathbb{D}$ and $e := (\nu, q, s) \in \mathbb{E}$. Let $\Delta t > 0$ and $\Delta z \in \ZZ^{n}$, such that
	\begin{eqnarray*}
		m_0\Delta t \geq \Delta z \cdot \nu > 0 &\text{and}& \left( -\mathscr{R}\right) \Delta t \geq \left| \Delta z - \left( \Delta z \cdot \nu\right) \nu\right|, 
	\end{eqnarray*}
	then
	\begin{eqnarray*}
		\overline{\mathrm{U}}_{a}(x - \Delta z, t) \leq \overline{\mathrm{U}}_{a}(x, t + \Delta t), && x \in \mathrm{C}_{d}(t + \Delta t).
	\end{eqnarray*}
\end{proposition}

\begin{proof}
	Since $m_0 \leq s \leq M_0$, the function $-|q|\left( x \cdot \nu - m_0t\right)$ is a subsolution in $\mathrm{C}_{d}$. The choice of $\Delta t$ and $\Delta z$ indicates that $\overline{\mathrm{U}}_{a}(x - \Delta z, 0) \leq \overline{\mathrm{U}}_{a}(x, \Delta t)$, for any $x \in \mathrm{C}_{d}(\Delta t)$. It also implies that $\overline{\mathrm{U}}_{a}(x - \Delta z, t) \leq \mathrm{O}_{e}(x, t + \Delta t)$, for any $x \in \mathrm{C}_{d}(t + \Delta t)$. Because $\Delta z \in \ZZ^{n}$, $\overline{\mathrm{U}}_{a}(x - \Delta z, t)$ is a subsolution bounded from above by $\mathrm{O}_{e}(x, t + \Delta t)$, in $\hat{\mathrm{C}}_{d} := \left\lbrace (x,t) \in \RR^{n} \times (0,\infty) \big| x \in \mathrm{C}_{d}(t + \Delta t) \right\rbrace $, so does $\max\left\lbrace \overline{\mathrm{U}}_{a}(x - \Delta z, t), \overline{\mathrm{U}}_{a}(x, t + \Delta t) \right\rbrace$. By the maximality of $\overline{\mathrm{U}}_{a}$ from Definition \ref{the definitions of obstacle solutions}, we conclude that
	\begin{eqnarray*}
	\max\left\lbrace \overline{\mathrm{U}}_{a}(x - \Delta z, t), \overline{\mathrm{U}}_{a}(x, t + \Delta t) \right\rbrace \leq \overline{\mathrm{U}}_{a}(x, t + \Delta t), && (x,t) \in \hat{\mathrm{C}}_{d}.
	\end{eqnarray*}
    Equivalently,
    \begin{eqnarray*}
    \overline{\mathrm{U}}_{a}(x - \Delta z, t) \leq \overline{\mathrm{U}}_{a}(x, t + \Delta t), && x \in \mathrm{C}_{d}(t + \Delta t).
    \end{eqnarray*}

\end{proof}

\begin{proposition}\label{the birkhoff property for supersolution in a shrinking domain}
	Fix $a := (\nu, R, \mathscr{R}, q, s) \in \mathbb{A}$ with $\mathscr{R} < 0$, then set $d := (\nu, R, \mathscr{R}) \in \mathbb{D}$ and $e := (\nu, q, s) \in \mathbb{E}$. Let $\Delta t > 0$ and $\Delta z \in \ZZ^{n}$, such that
	\begin{eqnarray*}
		\Delta z \cdot \nu \geq M\Delta t \geq 0 &\text{and}& \left( -\mathscr{R}\right) \Delta t \geq \left| \Delta z - \left( \Delta z \cdot \nu\right) \nu\right| ,
	\end{eqnarray*}
	then
	\begin{eqnarray*}
		\underline{\mathrm{U}}_{a}(x - \Delta z, t) \geq \underline{\mathrm{U}}_{a}(x, t + \Delta t), && x \in \mathrm{C}_{d}(t + \Delta t).
	\end{eqnarray*}
\end{proposition}

\begin{proof}
	Since $m_0 \leq s \leq M_0$, the function $-|q|\left( x \cdot \nu - M_0t\right)$ is a supersolution in $\mathrm{C}_{d}$. The choice of $\Delta t$ and $\Delta z$ indicates that $\underline{\mathrm{U}}_{a}(x - \Delta z, 0) \geq \underline{\mathrm{U}}_{a}(x, \Delta t)$, for any $x \in \mathrm{C}_{d}(\Delta t)$. It also implies that $\underline{\mathrm{U}}_{a}(x - \Delta z, t) \geq \mathrm{O}_{e}(x, t + \Delta t)$, for any $x \in \mathrm{C}_{d}(t + \Delta t)$. Because $\Delta z \in \ZZ^{n}$, $\underline{\mathrm{U}}_{a}(x - \Delta z, t)$ is a supersolution bounded from below by $\mathrm{O}_{e}(x, t + \Delta t)$, in $\hat{\mathrm{C}}_{d} := \left\lbrace (x,t) \in \RR^{n} \times (0,\infty) \big| x \in \mathrm{C}_{d}(t + \Delta t) \right\rbrace $, so does $\min\left\lbrace \underline{\mathrm{U}}_{a}(x - \Delta z, t), \underline{\mathrm{U}}_{a}(x, t + \Delta t) \right\rbrace$. By the minimality of $\underline{\mathrm{U}}_{a}$ from Definition \ref{the definitions of obstacle solutions}, we conclude that
	\begin{eqnarray*}
		\min\left\lbrace \underline{\mathrm{U}}_{a}(x - \Delta z, t), \underline{\mathrm{U}}_{a}(x, t + \Delta t) \right\rbrace \geq \underline{\mathrm{U}}_{a}(x, t + \Delta t), && (x,t) \in \hat{\mathrm{C}}_{d}.
	\end{eqnarray*}
	Equivalently,
	\begin{eqnarray*}
		\underline{\mathrm{U}}_{a}(x - \Delta z, t) \geq \underline{\mathrm{U}}_{a}(x, t + \Delta t), && x \in \mathrm{C}_{d}(t + \Delta t).
	\end{eqnarray*}	
	
\end{proof}

\begin{remark}\label{a remark on the birkhoff property in laminar case}
	In the previous Propositions \ref{the birkhoff property for subsolution in an expanding domain}, \ref{the birkhoff property for supersolution in an expanding domain}, \ref{the birkhoff property for subsolutions in two static domains}, \ref{the birkhoff property for supersolutions in two static domains}, \ref{the birkhoff property for subsolution in a shrinking domain}, \ref{the birkhoff property for supersolution in a shrinking domain}, the space shift $\Delta z \in \ZZ^{n}$ is only due the to periodicity of $g(x)$. In the laminar case, i.e., $g(x) = g(x^{\prime})$ with $x = (x^{\prime}, x_{n})$, it suffices to have $\Delta z = (\Delta z^{\prime}, \Delta z_{n})$ with $\Delta z^{\prime} \in \ZZ^{n-1}$ and $\Delta z_{n} \in \RR$.
\end{remark}

\section{Inf-convolution}\label{the inf convolution}

\subsection{Concepts and properties}

\begin{definition}\label{the definition of the inf convolution}
	Let $h := (r(t), \varphi(x)) \in \mathbb{F}$ and $u(x,t): U \rightarrow \RR$ be a function defined in a space time domain $U \subseteq \RR^{n} \times (0, \infty)$. The $h$ inf-convolution of $u(x,t)$, denoted by $u^{h_{-}}(x,t)$, is defined as follows.
	\begin{eqnarray*}
	u^{h_{-}}(x,t) := \inf_{y \in B_{r(t) \varphi(x)}(x)}u(y,t), &\text{where}& \left( B_{r(t)\varphi(x)}(x), t\right)  \subseteq U.
	\end{eqnarray*}
    Here $B_{r(t)\varphi(x)}(x) := \left\lbrace y \in \RR^{n} \big| |y - x| \leq r(t)\varphi(x)\right\rbrace $.
\end{definition}

\begin{definition}\label{the strict ordering relation between level sets of two functions}
	Let $f(x), g(x): \Omega \rightarrow \RR$, where $\Omega$ is a subset of $\RR^{n}$, for any $\mu \in \RR$, let us denote the sublevel set and superlevel set of the function $f$ in $\Omega$ as follows.
	\begin{eqnarray*}
	L^{-}_{\mu}\left( f; \Omega\right) := \left\lbrace x \in \Omega \big| f(x) \leq \mu \right\rbrace, && L^{+}_{\mu}\left( f; \Omega\right) := \left\lbrace x \in \Omega \big| f(x) \geq \mu \right\rbrace.
	\end{eqnarray*}
	For later convenience, we also denote the sub/super level-set based ordering relation as below:
	\begin{eqnarray*}
		f \prec_{(\Omega,\mu)} g, &\text{if}& (i)\hspace{1mm} f < g, (ii) \hspace{1mm} L_{\mu}^{+}(f;\Omega) \cap L_{\mu}^{-}(g;\Omega) = \emptyset, (iii)\hspace{1mm} \begin{cases}
			\inf\left\lbrace x \cdot \nu \big| x \in L_{\mu}^{+}(f;\Omega) \right\rbrace = - \infty\\
			\sup\left\lbrace x \cdot \nu \big| x \in L_{\mu}^{-}(g;\Omega) \right\rbrace = + \infty
		\end{cases} 
	\end{eqnarray*}
\end{definition}

The next proposition shows that any sublevel (resp. superlevel) set of the $(r(t),\varphi(x))$ inf-convolution of a lower semicontinuous function has a $r(t)\varphi(x)$ interior (resp. exterior) ball condition. That is to say, the inf-convolution makes the sublevel (resp. superlevel) set more regular from one direction.
\begin{proposition}\label{the exterior ball condition of the inf convolution}
	Fix $\mu \in \RR$, $h := (r(t), \sigma) \in \mathbb{F}$, where $\sigma > 0$ is a constant and $u(x,t) \in \LSC(U)$, where $U \subseteq \RR^{n} \times (0,\infty)$ is a space time domain. Let $u^{h_{-}}(x,t)$ be the $h$ inf-convolution through Definition \ref{the definition of the inf convolution}. Assume $(x,t)$ satisfies the following (i)-(iii),
	\begin{enumerate}
		\item [(i)] $x \in \partial \left\lbrace w \in \RR^{n} \big | (w,t) \in U, \hspace{2mm} u^{h_{-}}(w,t) \leq \mu\right\rbrace$;
		\item [(ii)] $\left( B_{r(t)\sigma}(x), t \right)  \subseteq U $;
		\item [(iii)] $u^{h_{-}}(x,t) = \mu$.
	\end{enumerate}
	Then there exists $\left( y, t\right)  \in U$, such that
	\begin{eqnarray*}
	|y - x| = r(t)\sigma &\text{and}& u^{h_{-}}(z,t) \leq \mu, \hspace{2mm}\text{if}\hspace{2mm} z \in B_{r(t)\sigma}(y).
	\end{eqnarray*}
\end{proposition}

\begin{proof}
	By the choice of $x$, we get that $\inf\limits_{y \in B_{r(t)\sigma}(x)}u(y,t) = \mu$. Since $u \in \LSC (U)$, there exists $y \in B_{r(t)\sigma}(x)$, such that $u(y,t) = \mu$. Suppose $|y - x| < r(t)\sigma$, then $u^{h_{-}}(\cdot,t) \leq \mu$ in a neighborhood of $x$, which contradicts to the choice of $x$. Hence $|y - x| = r(t)\sigma$. Moreover, $u(\cdot,t) > \mu$ in the interior of $B_{r(t)\sigma}(x)$. Next, let $z \in  B_{r(t)\sigma}(x)$, then $y \in B_{r(t)\sigma}(z)$, therefore, $u^{h_{-}}(z,t) \leq u(y,t) = \mu$.
\end{proof}

Based on the construction of a super barrier as follows, we show that the superlevel set of an obstacle subsolution propagates with a finite speed.

\begin{proposition}\label{finite speed propagation of subsolution}
	Fix $(\delta, \mu, t_{0}, C) \in (0,\infty) \times \RR \times (0, \infty) \times (1, \infty)$, $a := (\nu, R, \mathscr{R}, q, s) \in \mathbb{A}$ and then set $d := (\nu, R, \mathscr{R}) \in \mathbb{D}$ and $h := (\delta, 1) \in \mathbb{F}$. Assume that
	\begin{enumerate}
		\item [(i)] $u(x,t) \in \overline{\mathscr{S}}_{a}$, $v(x,t) \in \underline{\mathscr{S}}_{a}$;
		\item [(ii)] Let $\Omega$ be a domain such that $\left( \overline{\Omega} + B_{\delta}(0), t\right) \subseteq \mathrm{C}_{d}$, $t_{0} - \Delta t \leq t \leq t_{0}$;
		\item [(iii)] $v^{h_{-}}(x,t) : \overline{\Omega} \times [t_{0} - \Delta t, t_{0}] \rightarrow \RR$ is defined through Definition \ref{the definition of the inf convolution};
		\item [(iv)] $0 < \Delta t < \frac{(C - 1)\delta^{2}}{(n - 1)C^{2} + C(C - 1)M_0\delta}$ and $u(x, t_{0} - \Delta t) \prec_{(\Omega, \mu)} v^{h_{-}}(x,t_{0})$.
	\end{enumerate}
    Then
    \begin{eqnarray*}
    u(x, t_{0}) \prec_{(\Omega, \mu)} v^{\hat{h}_{-}}(x, t_{0}), &\text{where}& \hat{h} := \left(\left(1 - \frac{1}{C}\right)\delta, 1\right) \in \mathbb{F}.
    \end{eqnarray*}

\end{proposition}

\begin{proof}
	For any $y \in L_{\mu}^{-}\left( v(\cdot, t_{0}); \overline{\Omega}\right)$, let us define the superbarrier
	\begin{equation*}
	G_{y}(x,t) := \begin{cases}
	\mu, & x \in B_{r(t)}(y)\\
	+ \infty, & x \in \overline{\Omega}\diagdown B_{r(t)}(y)
	\end{cases},
	\end{equation*}
    where
    \begin{equation*}
    r(t) := \delta - \left( t - t_{0} + \Delta t\right)\cdot \left( \frac{(n - 1)C}{(C - 1)\delta} + M_0\right).
    \end{equation*}
    Here $r(t)$ is chosen such that $r(t_{0} - \Delta t) = \delta$ and $r(t_{0}) = \left( 1 - \frac{1}{C}\right)\delta$. In addition, $r^{\prime}(t) = - \left( \frac{(n-1)C}{(C - 1)\delta} + M_0 \right)$ guarantees that $G_{y}(x,t)$ is a supersolution. i.e.,
    \begin{eqnarray*}
    \partial_{t}G_{y} \geq \mathscr{F}\left( D^{2}G_{y}, DG_{y}, x\right), && (x,t) \in \overline{\Omega} \times \left( t_{0} - \Delta t, t_{0}\right).  
    \end{eqnarray*}
    On the other hand, let us consider the function
    \begin{equation*}
    H_{u}(x,t) := \begin{cases}
    \mu, & u(x, t) \geq \mu\\
    -\infty, & u(x,t) < \mu
    \end{cases}.
    \end{equation*}
    Since the operator $\mathscr{F}(\cdot, \cdot, \cdot)$ is geometric (c.f. \cite{Caffarelli and Monneau ARMA}), it is clear that $H_{u}(x,t)$ is a subsolution. i.e.,
    \begin{eqnarray*}
    \partial_{t}H_{u} \leq \mathscr{F}\left( D^{2}H_{u}, DH_{u}, x\right), && (x,t) \in \overline{\Omega} \times (t_{0} - \Delta t, t_{0}).
    \end{eqnarray*}
    Then Proposition \ref{the exterior ball condition of the inf convolution} and an application of the usual comparison principle (c.f. Proposition \ref{the usual comparison principle}), restricted to certain bounded domain if necessary, shows that 
    \begin{eqnarray*}
    H_{u}(x,t) < G_{y}(x,t), && (x,t) \in \overline{\Omega} \times \left[ t_{0} - \Delta t, t_{0}\right] .
    \end{eqnarray*}
In particular, it is true that
\begin{eqnarray*}
H_{u}(x, t_{0}) < G_{y}(x, t_{0}), && x \in \overline{\Omega}.
\end{eqnarray*}
Hence
\begin{eqnarray*}
L_{\mu}^{+}\left( u(\cdot, t_{0}); \overline{\Omega}\right) \cap L_{\mu}^{-}\left( G_{y}(\cdot, t_{0}); \overline{\Omega}\right) = \emptyset, && y \in L_{\mu}^{-}\left( v(\cdot, t_{0}); \overline{\Omega}\right) .
\end{eqnarray*}
Notice that
\begin{equation*}
\bigcup_{y \in L_{\mu}^{-}\left( v(\cdot,t_{0}); \overline{\Omega}\right)} L_{\mu}^{-}\left( G_{y}(\cdot, t_{0}); \overline{\Omega}\right) = L_{\mu}^{-}\left( v(\cdot, t_{0}); \overline{\Omega}\right) + B_{\left( 1 - \frac{1}{C}\right)\delta} = L_{\mu}^{-}\left( v^{\hat{h}_{-}}(\cdot, t_{0}); \overline{\Omega}\right). 
\end{equation*}
Based on Definition \ref{the strict ordering relation between level sets of two functions}, the conclusion follows immediately.
\end{proof}


\subsection{Evolution law}

The coming proposition shows that if we choose $h \in \mathbb{F}$ in an appropriate way, the $h$ inf-convolution of a (pseudo) supersolution is still a (pseudo) supersolution. This plays an important role, in later sections, in proving the local comparison principle.
\begin{proposition}\label{the evolution law of inf convolution}
	Let $h := (r(t), \varphi(x)) \in \mathbb{F}$ and assume the followng (i)-(iv).
	\begin{enumerate}
		\item [(i)] $\left| r(t)D\varphi(x)\right| < 1$;
		\item [(ii)] $r(t)$ and $\varphi(x)$ satisfy the differential inequality as follows.
		\begin{flalign}\label{the assumption in the inf convolution}
		\begin{aligned}
		r^{\prime}(t) + \left( \frac{\left( n + 1\right) \lVert D^{2}\varphi \rVert_{\infty}}{\varphi(x)} + \frac{M_0|D\varphi(x)|}{\varphi(x)}
		+ L_0\right)r(t) + \frac{|D\varphi(x)|^{2}r(t)}{\left( 1 - r(t)|D\varphi(x)|\right)^{2}\varphi^{2}(x)} &\leq& 0
		\end{aligned} 
		\end{flalign} 
		where $M_0$, $L_0$ are from (\ref{the hypothesis});
		\item [(iii)] Let $T > 0$, $\Omega \subseteq \RR^{n}$ be an open set and $u(x,t): \Omega \times (0,T) \rightarrow \RR$ be a pseudo viscosity supersolution (c.f. Definition \ref{pseudo viscosity supersolution});
		\item [(iv)] Denote the space domain $\Omega^{h_{-}}$ as follows:
		\begin{equation*}
		\Omega^{h_{-}} := \left\lbrace x\in\Omega \big| \dist(x, \Omega^{c}) > \sup_{(x,t) \in \Omega \times (0,T)}r(t)\varphi(x)\right\rbrace.
		\end{equation*} 
	\end{enumerate}
	  Then $u^{h_{-}}(x,t): \Omega^{h_{-}} \times (0,T) \rightarrow \RR$ is also a pseudo viscosity supersolution.
\end{proposition}

\begin{figure}[h]
	\centering
	\includegraphics[width=0.7\linewidth]{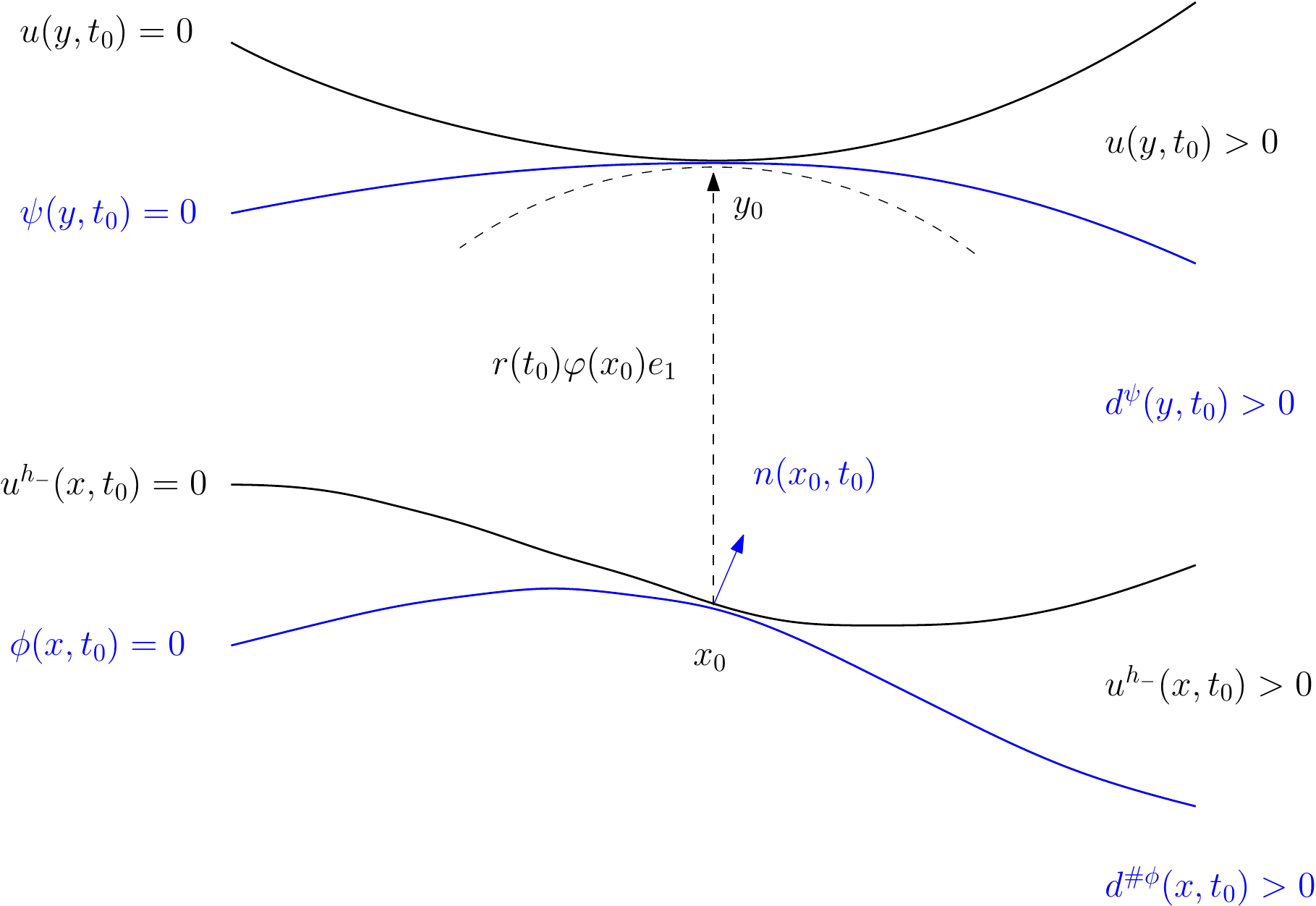}
	\caption{The inf-convolution}
	\label{fig:1}
\end{figure}

\begin{proof}
    Fix $(x_{0},t_{0}) \in \Omega \times (0,T)$, $\phi \in C^{2,1}\left( \Omega \times (0, T)\right) $ with (a) and (b) as below.
    \begin{enumerate}
    	\item [(a)] $|D\phi(x_{0},t_{0})| > 0$;
    	\item [(b)] $u^{h_{-}}(x_{0},t_{0}) - \phi(x_{0},t_{0}) \leq u^{h_{-}}(x,t) - \phi(x,t)$, $(x,t) \in \Omega \times (0,T)$.
    \end{enumerate} 
    Let us assume for simplicity that $u^{h_{-}}(x_{0},t_{0}) = \phi(x_{0},t_{0}) = \mu$, without loss of generality, we can take $\mu = 0$. The case of the general $\mu$ level set can be argued similarly. By the above (a), we have that $x_{0} \in \partial \left\lbrace w \in \RR^{n} \big| w \in \Omega, \hspace{2mm}\phi(w,t_{0}) > 0\right\rbrace $. Similar to the proof of Proposition \ref{the exterior ball condition of the inf convolution}, there exists $y_{0} \in \Omega$, such that 
    \begin{eqnarray*}
    |y_{0} - x_{0}| = r(t_{0})\varphi(x_{0}) &\text{and}& u(y_{0},t_{0}) = 0.
    \end{eqnarray*}
    Denote the orthonormal basis of $\RR^{n}$ by $e_{1}, e_{2}, \cdots, e_{n}$, such that $y_{0} - x_{0} = r(t_{0})\varphi(x_{0})e_{1}$.\\
   \underline{Step 1.} Let us define $\psi(y,t)$, which touches $u(y,t)$ from below at $(y_{0},t_{0})$, by the relation:
   \begin{equation*}
   \phi(x,t) = \psi(x + r(t)\varphi(x)e_{1},t).
   \end{equation*}
   Since $\left| r(t)D\varphi(x)\right| < 1$, we have $\text{det}J \neq 0$, where $J$ is the Jacobian matrix associated to the map $x \mapsto x + r(t)\varphi(x)e_{1}$. i.e.,
   \begin{equation*}
   J = \begin{pmatrix}
   1 + r(t)\varphi_{x_{1}}(x) & r(t)\varphi_{x_{2}}(x) & \cdots & r(t)\varphi_{x_{n}}(x)\\
   0 & 1 & \cdots & 0 \\
   \vdots & \vdots & \ddots &\vdots \\
   0 & 0 & \cdots & 1
   \end{pmatrix}.
   \end{equation*}
   By the inverse function theorem, $\psi(y,t)$ exists in a neighborhood of $(y_{0},t_{0})$\\
   \underline{Step 2.} By the choice of $\psi(y,t)$, let us set the (smooth) zero level set as follows
   \begin{equation*}
   \Gamma^{\psi}(t) := \partial\left\lbrace y \in \Omega \big| \psi(y,t) > 0 \right\rbrace, 
   \end{equation*}
   and we define the signed distance function $d^{\psi}(y,t)$ as below (c.f. \cite{Evans Soner and Souganidis CPAM}).
   \begin{equation*}
   d^{\psi}(y,t) := \begin{cases}
   \dist\left( y, \Gamma^{\psi}(t)\right), & y \in \mathrm{I}_{t}^{\psi} := \left\lbrace z \in \Omega \big| \psi(z,t) > 0\right\rbrace \\
   0, & y \in \Gamma^{\psi}(t) \\
   -\dist\left( y, \Gamma^{\psi}(t)\right), & y \in \mathrm{O}_{t}^{\psi} := \left\lbrace z \in \Omega \big| \psi(z,t) < 0\right\rbrace  
   \end{cases}.
   \end{equation*}
   Then $d^{\psi}$ is well defined and smooth in a neighborhood of $(y_{0},t_{0})$. In general, although $d^{\psi}$ may not touch $u(y,t)$ at $(y_{0},t_{0})$ from below, the zero level set of $d^{\psi}$ indeed touches that of $u(y,t)$ from below. Since $|D\phi(x_{0},t_{0})| > 0$, then $|D\psi(y_{0},t_{0})| > 0$, the viscosity inequality satisfied by $\psi$ is equivalent to that
   \begin{equation*}
   \frac{\psi_{t}}{|D\psi|}(y_{0},t_{0}) \geq - \kappa(y_{0}) + g\left(y_{0} \right),
   \end{equation*}
   where $\kappa(y_{0})$ denotes the mean curvature of the hypersurface $\left\lbrace y\in\Omega\big| u(y,t_{0}) = 0 \right\rbrace$ at $y_{0}$. Since $|D\psi(y_{0},t_{0})| > 0$, this is equivalent to the same inequality with $\psi$ replaced by $d^{\psi}$, i.e.,
  \begin{equation*}
  d^{\psi}_{t}(y_{0},t_{0}) \geq \mathscr{F}\left( D^{2}d^{\psi}(y_{0},t_{0}), Dd^{\psi}(y_{0},t_{0}), y_{0}\right),
  \end{equation*}
   which can be rewritten as
   \begin{equation}\label{the super equation satisfied by the signed distance function}
   d_{t}^{\psi}(y_{0},t_{0}) - \left[ \left( \Delta - \Delta_{\infty}\right)d^{\psi}(y_{0},t_{0}) - g(y_{0})|Dd^{\psi}(y_{0},t_{0})| \right] \geq 0,
   \end{equation}
   where $\Delta$ is the usual Laplace operator and $\Delta_{\infty}$ is defined as follows
   \begin{equation*}
   \Delta_{\infty}d^{\psi} := \frac{1}{|Dd^{\psi}|^{2}}\sum_{1 \leq i,j \leq n} \frac{\partial d^{\psi}}{\partial y_{i}}\cdot \frac{\partial^{2}d^{\psi}}{\partial y_{i}\partial y_{j}}\cdot \frac{\partial d^{\psi}}{\partial y_{j}},
   \end{equation*}
   which is called the (normalized) infinity Laplace operator (c.f. \cite{Crandall Evans and Gariepy, Peres Schramm Sheffield and Wilson}, etc.). It denotes the second derivative of $d^{\psi}$ in the direction of $\frac{Dd^{\psi}}{|Dd^{\psi}|}$.\\
   \underline{Step 3.} Let us introduce the function $d^{\#\phi}(x,t)$ as follows, which is well defined and smooth in a neighborhood of $(x_{0},t_{0})$:
   \begin{equation}\label{the pull back of the signed distance function}
   d^{\#\phi}(x,t) := d^{\psi}\left( x + r(t)\varphi(x)e_{1},t\right). 
   \end{equation}
   It is clear that the zero level set of $d^{\#\phi}$ coincides with that of $\phi$, at least in a neighborhood of $(x_{0},t_{0})$. However, it is not necessary the case that $d^{\#\phi}$ matches the signed distance function $d^{\phi}$, associated to the test function $\phi$. Our aim in the following steps is: by equation (\ref{the super equation satisfied by the signed distance function}) and relation (\ref{the pull back of the signed distance function}), we derive the evolution law of the zero level set of $d^{\#\phi}$ (i.e., the zero level set of $\phi$) at $(x_{0},t_{0})$. From which, we get the viscosity inequality satisfied by the test function $\phi$.\\
  \underline{Step 4.} Since $d^{\psi}(y,t)$ (resp. $d^{\#\phi}(x,t)$) is smooth around $(y_{0},t_{0})$ (resp. $(x_{0},t_{0})$), we are allowed to differentiate it in the classical sense. In fact, we have the following properties (c.f. \cite{Evans Soner and Souganidis CPAM}) regarding the derivatives of $d^{\psi}$.
  \begin{enumerate}
  	\item [(1)] $|Dd^{\psi}(y,t_{0})| = 1$ in a neighborhood of $y_{0}$;
  	\item [(2)] $\frac{\partial d^{\psi}}{\partial y_{1}}(y_{0},t_{0}) = -1$ and $\frac{\partial d^{\psi}}{\partial y_{k}}(y_{0},t_{0}) = 0$, $k \neq 1$;
  	\item [(3)] $\frac{\partial^{2} d^{\psi}}{\partial y_{1}^{2}}(y_{0},t_{0}) = 0$, i.e., $\Delta_{\infty}d^{\psi}(y_{0},t_{0}) = 0$.
  \end{enumerate}
  The above (1) implies that $\left( \frac{\partial}{\partial y_{k}}\left| Dd^{\psi}\right|^{2}\right)(y_{0},t_{0}) = 0$, which is equivalent to that
  \begin{eqnarray*}
  \sum_{i = 1}^{n}\left( \frac{\partial^{2}d^{\psi}}{\partial y_{i}\partial y_{k}}\cdot \frac{\partial d^{\psi}}{\partial y_{i}}\right) (y_{0},t_{0}) = 0, && k = 1, \cdots, n.
  \end{eqnarray*}
  Then by (2), we have for any $k$ that $\frac{\partial^{2}d^{\psi}}{\partial y_{1}\partial y_{k}}(y_{0},t_{0}) = 0$. A direct differentiation in (\ref{the pull back of the signed distance function}) implies that
  \begin{eqnarray*}
  \frac{\partial d^{\#\phi}}{\partial x_{k}}(x_{0},t_{0}) = \frac{\partial d^{\psi}}{\partial y_{k}}(y_{0},t_{0}) + \frac{\partial d^{\psi}}{\partial y_{1}}(y_{0},t_{0})\cdot r(t_{0})\frac{\partial \varphi}{\partial x_{k}}(x_{0}), && 1 \leq k \leq n,
  \end{eqnarray*}
  and then
  \begin{eqnarray*}
  \frac{\partial^{2}d^{\#\phi}}{\partial x_{k}^{2}}(x_{0},t_{0}) &=& \frac{\partial^{2}d^{\psi}}{\partial y_{k}^{2}}(y_{0},t_{0}) + 2\frac{\partial^{2}d^{\psi}}{\partial y_{1}\partial y_{k}}(y_{0},t_{0})\cdot r(t_{0})\frac{\partial \varphi}{\partial x_{k}}(x_{0})\\
  &+& \frac{\partial^{2}d^{\psi}}{\partial y_{1}^{2}}(y_{0},t_{0})\cdot r^{2}(t_{0})\left( \frac{\partial \varphi}{\partial x_{k}}(x_{0})\right)^{2} + \frac{\partial d^{\psi}}{\partial y_{1}}\cdot r(t_{0})\frac{\partial^{2}\varphi}{\partial x_{k}^{2}}(x_{0}). 
  \end{eqnarray*}
  Hence,
  \begin{equation*}
  \Delta d^{\#\phi}(x_{0},t_{0}) = \Delta d^{\psi}(y_{0},t_{0}) - r(t_{0})\Delta\varphi(x_{0}).
  \end{equation*}
  Next, let us explore the evolution law of $d^{\#\phi}(x,t)$ through the following relation.
  \begin{eqnarray*}
  d_{t}^{\#\phi}(x_{0},t_{0}) &=& d_{t}^{\psi}(y_{0},t_{0}) + \frac{\partial d^{\psi}}{\partial y_{1}}(y_{0},t_{0})\cdot \varphi(x_{0})r^{\prime}(t_{0})\\
  &\geq& \Delta d^{\psi}(y_{0},t_{0}) + g(y_{0})|Dd^{\psi}(y_{0},t_{0})| - \varphi(x_{0})r^{\prime}(t_{0})\\
  &=& \left(\Delta d^{\#\phi}(x_{0},t_{0}) + r(t_{0})\Delta\varphi(x_{0})\right) + g(y_{0})|Dd^{\psi}(y_{0},t_{0})| - \varphi(x_{0})r^{\prime}(t_{0}) \\
  &=& \left( \Delta - \Delta_{\infty}\right)d^{\#\phi}(x_{0},t_{0}) + g(x_{0}) |Dd^{\#\phi}(x_{0},t_{0})| \\
  &+& \left( \Delta_{\infty}d^{\#\phi}(x_{0},t_{0}) + r(t_{0})\Delta\phi(x_{0})\right) - \varphi(x_{0})r^{\prime}(t_{0})\\
  &+& g(y_{0})|Dd^{\psi}(y_{0},t_{0})| - g(x_{0})|Dd^{\#\phi}(x_{0},t_{0})| \\
  &=& \left( -\kappa^{\#\phi}(x_{0}) + g(x_{0})\right)|Dd^{\#\phi}(x_{0},t_{0})| + \underbrace{\left( - \varphi(x_{0})r^{\prime}(t_{0})\right)}_{\varepsilon_{1}}\\
  &+& \underbrace{\left( \Delta_{\infty}d^{\#\phi}(x_{0},t_{0}) + r(t_{0})\Delta\varphi(x_{0})\right)}_{\varepsilon_{2}} + \underbrace{g(y_{0}) - g(x_{0})|Dd^{\#\phi}(x_{0},t_{0})|}_{\varepsilon_{3}},
  \end{eqnarray*}
  where $\kappa^{\#\phi}(x_{0})$ is the mean curvature of the set $\left\lbrace x\in\Omega \big| d^{\#\phi}(x,t_{0}) = 0\right\rbrace $ at $x_{0}$.\\
  \underline{Step 5.} Keep the choice of $e_{1}$, let us select $e_{2}$, $\cdots$, $e_{n}$, so that $D\varphi(x_{0}) = \alpha e_{1} + \beta e_{2}$ for two numbers $\alpha$ and $\beta$. Recall the formula of $\Delta_{\infty}d^{\#\phi}(x_{0},t_{0})$ derived in (\ref{the fomula of the infinity laplacian term}) and $\Gamma^{\psi}(t_{0})$ has an interior $r(t_{0})\varphi(x_{0})$ ball condition at $y_{0}$, which implies that
  \begin{equation*}
  \frac{\partial^{2}d}{\partial y_{2}^{2}}(y_{0},t_{0}) \geq - \frac{1}{r(t_{0})\varphi(x_{0})}.
  \end{equation*}
Then, we can estimate the error term $\varepsilon_{2}$ as follows,
  \begin{eqnarray*}
  \varepsilon_{2} &:=& \Delta_{\infty}d^{\#\phi}(x_{0},t_{0}) + r(t_{0})\Delta\varphi(x_{0})\\
  &\geq& \frac{(r(t_{0})\beta)^{2}}{(1 + r(t_{0})\alpha)^{2} + (r(t_{0})\beta)^{2}}\cdot\frac{-1}{r(t_{0})\varphi(x_{0})}  + r(t_{0})\Delta\varphi(x_{0})\\
  &-& \frac{(1 + r(t_{0})\alpha)^{2}r(t_{0})}{(1 + r(t_{0})\alpha)^{2} + \left( r(t_{0})\beta\right)^{2}} \frac{\partial^{2}\varphi}{\partial x_{1}^{2}}(x_{0}) - \frac{2(1 + r(t_{0})\alpha)\left( r(t_{0})\beta\right)r(t_{0})}{(1 + r(t_{0})\alpha)^{2} + \left( r(t_{0})\beta\right)^{2}} \frac{\partial^{2}\varphi}{\partial x_{1}\partial x_{2}}(x_{0})\\
  &-& \frac{\left( r(t_{0})\beta\right)^{2}r(t_{0})}{(1 + r(t_{0})\alpha)^{2} + \left( r(t_{0})\beta\right)^{2} } \frac{\partial^{2}\varphi}{\partial x_{2}^{2}}(x_{0})\\
  &\geq& - \frac{r(t_{0})|D\varphi(x_{0})|^{2}}{\left( 1 - r(t_{0})|D\varphi(x_{0})|\right)^{2}\varphi(x_{0})} - (n + 1)r(t_{0})\lVert D^{2}\varphi\rVert_{\infty}.
  \end{eqnarray*}
And then the term $\varepsilon_{3}$ as below:
\begin{eqnarray*}
\varepsilon_{3} &:=& g(y_{0}) - g(x_{0})|Dd^{\#\phi}(x_{0},t_{0})|\\
&=& \left( g(y_{0}) - g(x_{0})\right) + g(x_{0})\left( 1 - |Dd^{\#\phi}(x_{0},t_{0})|\right) \\
&\geq& - L_0 r(t_{0})\varphi(x_{0}) - M_0 r(t_{0})|D\varphi(x_{0})|,
\end{eqnarray*}
where $L_0$ and $M_0$ are from the hypothesis (\ref{the hypothesis}). The assumption (\ref{the assumption in the inf convolution}) implies that $\varepsilon_{1} + \varepsilon_{2} + \varepsilon_{3} \geq 0$. Therefore,
\begin{equation*}
d_{t}^{\#\phi}(x_{0},t_{0}) \geq \left( -\kappa^{\#\phi}(x_{0}) + g(x_{0})\right)|Dd^{\#\phi}(x_{0},t_{0})|,
\end{equation*}
which is equivalent to (since $|D\phi(x_{0},t_{0})| > 0$) the inequality as follows, as desired:
\begin{equation*}
\phi_{t}(x_{0},t_{0}) \geq \mathscr{F}\left( D^{2}\phi(x_{0},t_{0}), D\phi(x_{0},t_{0}), x_{0}\right).
\end{equation*}
\end{proof}

\section{Local comparison principle}\label{the section of LCP}

\subsection{Discrepancy and the lattice points}\label{the subsection of discrepancy}

In this part, based on the language of discrepancy, we investigate the existence of certain lattice points that are arbitrarily close to a given hyperplane with irrational normal direction $\nu \in \mathbb{S}^{n-1}\diagdown \RR\ZZ^{n}$. Let us first recall some definitions and Lemmas. 

\begin{definition}[\cite{Kuipers and Niederreiter,Choi Kim and Lee Analysis and PDEs,Feldman JMPA}]
	A bounded sequence $\left\lbrace x_{\ell}\right\rbrace_{\ell \geq 1} \subseteq [a,b]$ is said to be equidistributed on the interval $[a,b]$ if for any $[c,d] \subseteq [a,b]$, we have that
	\begin{equation*}
	\lim_{\ell\rightarrow \infty} \frac{\left| \left\lbrace x_{i} \right\rbrace_{i = 1}^{\ell} \cap [c,d]\right| }{\ell} = \frac{d - c}{b - a}.
	\end{equation*}
\end{definition}

\begin{definition}
	A sequence $\left\lbrace x_{\ell}\right\rbrace_{\ell \geq 1}$ is said to be equidistributed modulo $1$ if $\left\lbrace x_{\ell} - \left[x_{\ell}\right] \right\rbrace_{\ell \geq 1} $ is equidistributed in the interval $[0,1]$.
\end{definition}

\begin{lemma}[Weyl's equidistribution theorem]
	Let $x \in \RR \diagdown \QQ$, then $\left\lbrace \ell x \right\rbrace_{\ell \geq 1}$ is equidistributed modulo $1$.
\end{lemma}

\begin{definition}[\cite{Kuipers and Niederreiter,Choi Kim and Lee Analysis and PDEs,Feldman JMPA}]
	Let $\left\lbrace x_{n}\right\rbrace_{n \geq 1}$ be a sequence in $\RR$. For a subset $E \subseteq [0,1]$, let $A(E;N)$ denote the number of points $\left\lbrace x_{\ell}\right\rbrace_{\ell = 1}^{N}$ that lie in $E$. 
	\begin{enumerate}
		\item [(a)] The sequence $\left\lbrace x_{\ell}\right\rbrace_{\ell \geq 1}$ is said to be uniformly distributed modulo $1$ in $\RR$ if
		\begin{equation*}
		\lim_{N \rightarrow \infty}\frac{A(E;N)}{N} = \mathcal{L}(E),
		\end{equation*}
		for all $E = [a, b) \subseteq [0,1]$. Here $\mathcal{L}$ denotes the Lebesgue measure in $\RR$.
		\item [(b)] For $x \in [0, 1]$, we define the discrepancy
		\begin{equation*}
		D_{N}(x) := \sup_{E = [a, b) \subseteq [0,1]} \left|\frac{A(E;N)}{N} - \mathcal{L}(E) \right|, 
		\end{equation*}
		where $A(E;N)$ is defined with the sequence $\left\lbrace \ell x\right\rbrace_{\ell \geq 1}$ modulo $1$.
		\item [(c)] For $x \in [0,1]$, we define the modified discrepancy 
		\begin{equation*}
		D_{N}^{*}(x) := \sup_{0 < a \leq 1}\left| \frac{A([0,a); N)}{N} - a \right|,
		\end{equation*}
		where $A(E;N)$ is defined with the sequence $\left\lbrace \ell x\right\rbrace_{\ell \geq 1}$ modulo $1$.
	\end{enumerate}
\end{definition}

\begin{lemma}[\cite{Kuipers and Niederreiter,Feldman JMPA}]
	Let us list some properties of the discrepancies as follows.
	\begin{enumerate}
		\item [(i)] The discrepancies $D_{N}$ and $D_{N}^{*}$ are equivalent up to constants:
		\begin{equation*}
		D_{N}^{*} \leq D_{N} \leq 2D_{N}^{*}.
		\end{equation*}
		\item [(ii)] Let $x_{1} \leq x_{2} \leq \cdots \leq x_{N}$ be in $[0, 1)$, then
		\begin{equation*}
		D_{N}^{*} = \frac{1}{2N} + \max_{i = 1, \cdots, N}\left| x_{i} - \frac{2i - 1}{2N}\right|.
		\end{equation*}
		\item [(iii)] Fix any $x \in \RR\diagdown \QQ$, the modified discrepancy function $D_{N}^{*}(\cdot)$ is continuous in a neighborhood of $x$.
	\end{enumerate}
\end{lemma}

\begin{definition}[\cite{Feldman JMPA}]\label{the definition of disrepancy for an irrational vector}
	Fix any $\nu = (\nu_{1}, \cdots, \nu_{n}) \in \mathbb{S}^{n-1}$, let $m_{j}(\nu)$, $1 \leq j \leq n$, be numbers such that $m_{j}(\nu)|\nu|_{\ell^{\infty}} = |\nu_{j}|$, where $|\nu|_{\ell^{\infty}} := \max_{1 \leq i \leq n}|\nu_{i}|$. Let us define the number $\omega_{\nu}(N)$ as follows:
	\begin{eqnarray*}
		\omega_{\nu}(N) := 2\min_{1 \leq j \leq n} D_{N}^{*}(m_{j}(\nu)), && N > 1.
	\end{eqnarray*}
	Let us denote by $H_{\nu}$ the hyperplane with $\nu \in \mathbb{S}^{n-1}$ as its normal direction. i.e.,
	\begin{equation*}
	H_{\nu} := \left\lbrace x \in \RR^{n} \big | x \cdot \nu = 0\right\rbrace.
	\end{equation*}
\end{definition}

\begin{lemma}[\cite{Kuipers and Niederreiter}]\label{properties regarding an index function of directions}
	Fix $\nu \in \mathbb{S}^{n-1}$, then
	\begin{enumerate}
		\item [(i)] If $\nu \in \mathbb{S}^{n-1}\cap\RR\ZZ^{n}$, then $\omega_{\nu}(\cdot) \in (0,1)$ has a positive lower bound;
		\item [(ii)] If $\nu \in \mathbb{S}^{n-1}\diagdown \RR\ZZ^{n}$, then $\lim\limits_{N\rightarrow\infty}\omega_{\nu}(N) = 0$.
	\end{enumerate} 
\end{lemma}

\begin{proposition}\label{a lattice point that is close to a hyperplane}
	For any $\nu \in \mathbb{S}^{n-1}\diagdown \RR\ZZ^{n}$ and $0 < \delta < 1$, there exists $R_{0}(\nu,\delta) > 0$, such that the following statement holds: fix any $R \geq R_{0}(\nu,\delta)$ and $x_{0} \in R\mathbb{S}^{n-1}\cap H_{\nu}$, there exists $z_{0} \in \RR^{n}$, such that (i)-(iii) as follows are satisfied.
	\begin{equation*}
	(i)\hspace{1mm} \frac{\delta}{3} < z_{0}\cdot\nu < \delta; \hspace{4mm} (ii)\hspace{1mm} |z_{0} - 2 x_{0}| < \frac{R}{3}; \hspace{4mm} (iii) \hspace{1mm} z_{0} - x_{0} \in \ZZ^{n}.
	\end{equation*}
\end{proposition}

\begin{figure}[h]
\centering
\includegraphics[width=0.7\linewidth]{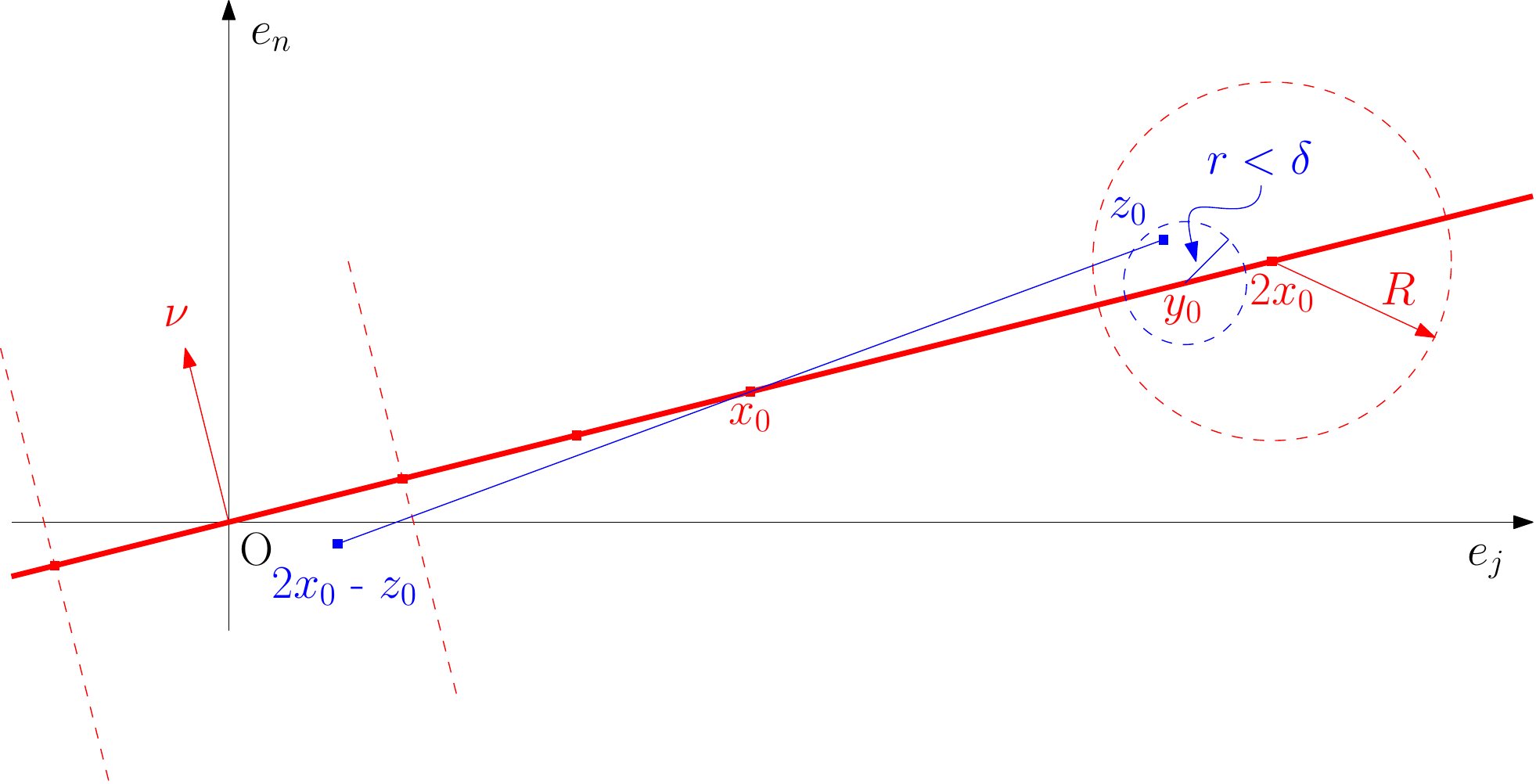}
\caption{A lattice point that is close to $H_{\nu}$}
\label{fig:2}
\end{figure}

\begin{proof}
	Since $\nu$ is an irrational direction, by Lemma \ref{properties regarding an index function of directions}, there exists a positive integer $N$ such that $0 < \omega_{\nu}(N) < \frac{\delta}{3|\nu|_{\ell^{\infty}}}$. Let us assume without loss of generality that 
	\begin{eqnarray*}
		\nu_{n} = |\nu|_{\ell^{\infty}} \hspace{4mm} \text{and} \hspace{4mm} \omega_{\nu}(N) = 2D_{N}^{*}(m_{j}(\nu)), \hspace{4mm} \text{where} \hspace{4mm} m_{j}(\nu) \in \RR\diagdown \QQ.
	\end{eqnarray*}
	Pick $s \in [0,1)^{n}$ such that $- x_{0} = s \mod \ZZ^{n}$, then
	\begin{eqnarray*}
		s = \left( s - \frac{s\cdot\nu}{\nu_{n}} e_{n}\right) + \frac{s\cdot\nu}{\nu_{n}} e_{n}, &\text{where}& \left( s - \frac{s\cdot\nu}{\nu_{n}} e_{n}\right) \in H_{\nu}.
	\end{eqnarray*}
	Let us only consider the case $|\nu_{j}| = -\nu_{j}$ since the other case can be analyzed similarly. Then the sequence of vectors $k\left( m_{j}(\nu)e_{n} + e_{j}\right)$, $1 \leq k \leq N$ lie on $H_{\nu}$. Because $0 < D_{N}(m_{j}(\nu)) \leq \omega_{\nu}(N) < \frac{\delta}{3|\nu|_{\ell^{\infty}}}$, there exists $1 \leq k_{0} \leq N$, such that if $ \frac{2\delta}{3|\nu|_{\ell^{\infty}}} \leq \frac{s\cdot\nu}{\nu_{n}} - \left[\frac{s\cdot\nu}{\nu_{n}}\right]$, then take $k_{0}m_{j}(\nu) - \left[k_{0}m_{j}(\nu)\right] \in \left(\frac{s\cdot\nu}{\nu_{n}} - \left[\frac{s\cdot\nu}{\nu_{n}}\right] - \frac{2\delta}{3|\nu|_{\ell^{\infty}}}, \frac{s\cdot\nu}{\nu_{n}} - \left[\frac{s\cdot\nu}{\nu_{n}}\right] - \frac{\delta}{3|\nu|_{\ell^{\infty}}}\right)$,
	\begin{equation*}
	\frac{\delta}{3|\nu|_{\ell^{\infty}}} < \flat := \left( \frac{s\cdot\nu}{\nu_{n}} - \left[\frac{s\cdot\nu}{\nu_{n}}\right] \right)  - \left( k_{0}m_{j}(\nu) - \left[k_{0}m_{j}(\nu)\right]  \right) \leq 2\omega_{\nu}(N) < \frac{\delta}{|\nu|_{\ell^{\infty}}}.
	\end{equation*}
    If $\frac{\delta}{3|\nu|_{\ell^{\infty}}} \leq \frac{s\cdot\nu}{\nu_{n}} - \left[\frac{s\cdot\nu}{\nu_{n}}\right] < \frac{2\delta}{3|\nu|_{\ell^{\infty}}}$, then take $k_{0}m_{j}(\nu) - \left[k_{0}m_{j}(\nu)\right] \in \left(1 - \frac{\delta}{3|\nu|_{\ell^{\infty}}}, 1\right)$,
    \begin{equation*}
    \frac{\delta}{3|\nu|_{\ell^{\infty}}} < \flat := \left( 1 + \frac{s\cdot\nu}{\nu_{n}} - \left[\frac{s\cdot\nu}{\nu_{n}}\right] \right)  - \left( k_{0}m_{j}(\nu) - \left[k_{0}m_{j}(\nu)\right]  \right) < 3\omega_{\nu}(N) < \frac{\delta}{|\nu|_{\ell^{\infty}}}.
    \end{equation*}
    And if $0 \leq \frac{s\cdot\nu}{\nu_{n}} - \left[\frac{s\cdot\nu}{\nu_{n}}\right] < \frac{\delta}{3|\nu|_{\ell^{\infty}}}$, then take $k_{0}m_{j}(\nu) - \left[k_{0}m_{j}(\nu)\right] \in \left(1 - \frac{2\delta}{3|\nu|_{\ell^{\infty}}}, 1 - \frac{\delta}{3|\nu|_{\ell^{\infty}}}\right)$,
    \begin{equation*}
    \frac{\delta}{3|\nu|_{\ell^{\infty}}} < \flat := \left( 1 + \frac{s\cdot\nu}{\nu_{n}} - \left[\frac{s\cdot\nu}{\nu_{n}}\right] \right)  - \left( k_{0}m_{j}(\nu) - \left[k_{0}m_{j}(\nu)\right]  \right) < 3\omega_{\nu}(N) < \frac{\delta}{|\nu|_{\ell^{\infty}}}.
    \end{equation*}
	Let us denote $y_{0}$ and $z_{0}$ as follows to guarantee (i) and (ii).
	\begin{eqnarray*}
		y_{0} := 2x_{0} + k_{0}\left( m_{j}(\nu)e_{n} + e_{j}\right) + \left( s - \frac{s\cdot\nu}{\nu_{n}} e_{n}\right) \in H_{\nu} &\text{and}&
		z_{0} := y_{0} + \flat e_{n}. 
	\end{eqnarray*}
	Finally, let us estimate $|z_{0} - 2x_{0}|$,
	\begin{eqnarray*}
		|z_{0} - 2x_{0}| \leq |z_{0} - y_{0}| + |y_{0} - 2x_{0}| < 2 + 2N + \left(1 + \sqrt{n} \right) = 2N + \sqrt{n} + 3.
	\end{eqnarray*}
	Hence, let us set $R_{0}(\nu,\delta) := 6N + 3\sqrt{n} + 9$, where $N$ is picked such that $0 < \omega_{\nu}(N) < \frac{\delta}{3|\nu|_{\ell^{\infty}}}$. And (iii) follows immediately. 
\end{proof}

\begin{remark}
	The above arguments are motivated by similar proofs in \cite{Choi Kim and Lee Analysis and PDEs, Feldman JMPA}.
\end{remark}

\begin{remark}
	Since $D_{N}^{*}(\cdot)$ is continuous in a neighborhood of any $x \in \RR\diagdown \QQ$, $\omega_{\cdot}(N)$ is also continuous in an $\mathbb{S}^{n-1}$-neighborhood of any $\nu \in \mathbb{S}^{n-1}\diagdown \RR\ZZ^{n}$. Therefore, the point $z_{0} = z_{0}(\nu)$, which is characterized as above, depends on $\nu$ continuously in a neighborhood of $\nu$.
\end{remark}

\subsection{The local comparison principle}\label{the subsection of LCP}

\begin{definition}\label{comparison consistent triplet}
	The triplet $(\nu, T, R) \in \mathbb{S}^{n-1} \times (0, \infty) \times \left( \frac{\sqrt{n}}{2}, \infty\right)$ is called \textit{comparison consistent} if the following property holds: there exists $0 < \delta < \delta(T)$, such that for any $x_{0} \in \partial\Upomega(0, R; \nu)$ (see (\ref{thin cylinder pointing to a specific direction})), we can find $z_{0}$, such that the following (i)-(iii) hold:
	\begin{equation*}
		(i) \hspace{1mm} \frac{\delta}{3} < \left( z_{0} - x_{0}\right) \cdot\nu < \delta; \hspace{2mm} (ii) \hspace{1mm} |z_{0} - 2x_{0}| < \frac{R}{3}; \hspace{2mm} (iii) \hspace{1mm} z_{0} - x_{0} \in \ZZ^{n},
	\end{equation*}
	here $\delta(T)$ is defined as follows.
	\begin{eqnarray*}
		\delta(T) := \frac{M_0 m_0}{M_0 - m_0}\cdot \frac{\gamma^{2}(T)}{\left( \sqrt{M_0\gamma(T) + (n - 1)} + \sqrt{n - 1}\right)^{2}}, &\text{with}& \gamma(t) := \frac{1}{2}e^{-2L_0t}.
	\end{eqnarray*}    
\end{definition}

\begin{proposition}[Local comparison principle]\label{LCP in the unit scale}
	Fix any $m_0 \leq s_{1} < s_{2} \leq M_0$, $q \in \RR^{n}\diagdown\left\lbrace 0\right\rbrace $ and $\xi_{0} \in \argmin\limits\limits_{\xi\in\ZZ^{n},\hspace{1mm} \xi \cdot \nu > 1}|\xi|$. Let $(\nu, T, R)$ be a comparison consistent triplet (c.f. Definition \ref{comparison consistent triplet}), such that $R \geq \frac{12(3n + M_0 + 27)}{L_0}$, then
	\begin{eqnarray*}
		\overline{\mathrm{U}}_{a_{2}}(x,t) < \underline{\mathrm{U}}_{a_{1}}(x - \xi_{0},t), &\text{if}& x\in\Upomega(0,R; \nu), \hspace{2mm} 0 \leq t \leq T,
	\end{eqnarray*}
	where $a_{i} = (\nu, R, \mathscr{R}, q, s_{i})$, $i = 1, 2$, $ \mathscr{R} = \frac{4M_0 R}{\delta}$, and $\Upomega(0,R; \nu)$ is defined in (\ref{thin cylinder pointing to a specific direction}).

\end{proposition}

\begin{proof}
	Let us denote $d := (\nu, R, \mathscr{R})$ and define functions
	\begin{eqnarray*}
		\gamma(t) = \frac{1}{2}e^{-2L_0t} &\text{and}& \mathrm{U}(x,t) := \underline{\mathrm{U}}_{a_{1}}(x - \xi_{0},t)
	\end{eqnarray*}
	The choice of $\xi_{0}$ and the location of the obstacle implies that (c.f. Definition \ref{the definition of the domain})
	\begin{eqnarray*}
		\overline{\mathrm{U}}_{a_{2}}(x,t) < \mathrm{U}(x,t), && x \in \mathrm{C}_{d} \cap \left( \mathrm{C}_{d} + \xi_{0}\right), \hspace{2mm}  0 \leq t < \frac{1}{s_{2} - s_{2}}.
	\end{eqnarray*}
	In order to establish the desired result, we assume on the contrary
	\begin{equation}\label{the contrary assumption in the local comparison principle}
		\sup\left\lbrace t > 0 \big| \overline{\mathrm{U}}_{a_{2}}(x,t) < \mathrm{U}(x,t), \hspace{2mm} x \in \Upomega(0,R;\nu) \right\rbrace \leq T, 
	\end{equation}
	and derive a contradiction through the following steps.\\
	\underline{Step 1.} Let $h := (\gamma(t),1) \in \mathbb{F}$ and $\mathrm{U}^{h_{-}}(x,t)$ be the $h$ inf-convolution of $\mathrm{U}(x,t)$ by Definition \ref{the definition of the inf convolution}. Then the above assumption (\ref{the contrary assumption in the local comparison principle}) shows that 
	\begin{equation}\label{the moment of the first crossing}
		\frac{1}{2(s_{2} - s_{1})} \leq t_{0} := \sup\left\lbrace t > 0 \big| \overline{\mathrm{U}}_{a_{2}}(x,t) < \mathrm{U}^{h_{-}}(x,t), \hspace{2mm} x \in \Upomega(0,R;\nu) \right\rbrace \leq T.
	\end{equation}
	Assume $x_{0}$ is the first crossing point between $\overline{\mathrm{U}}_{a_{2}}$ and $\mathrm{U}^{h_{-}}$, more precisely,
	\begin{eqnarray*}
		\overline{\mathrm{U}}_{a_{2}}(x_{0},t_{0}) = \mathrm{U}^{h_{-}}(x_{0},t_{0}).
	\end{eqnarray*}
	By applying Proposition \ref{the usual comparison principle} to $\overline{\mathrm{U}}_{a_{2}}(\cdot,t)$ and $\mathrm{U}(\cdot,t)$ in $\Upomega(0,R;\nu)$ (we can only consider a bounded subdomain if necessary), the maximum of $\overline{\mathrm{U}}_{a_{2}}(x,t) - \mathrm{U}(x,t)$ is obtained on $\partial\Upomega(0,R;\nu)$. Since $h$ is $x$-independent, $x_{0} \in \partial\Upomega(0,R;\nu)$. Therefore, $\overline{\mathrm{U}}_{a_{2}}(x_{0},t_{0}) = \mathrm{U}^{h_{-}}(x_{0},t_{0}) = \mu$ for some $\mu \in \RR$.\\
	\underline{Step 2.} By the choise of $(\nu, T, R)$, we can find $0 < \delta < \delta(T)$. Morevoer, there exists a constant $C > 1$, such that
	\begin{equation}\label{the small vertical distance}
	\delta = \left( \frac{M_0 m_0}{M_0 - m_0}\right)\cdot \frac{ (C - 1)\gamma^{2}(T)}{(n - 1)C^{2} + C(C - 1)\gamma(T)M_0}.
	\end{equation}
    Moreover, there exists $z_{0} \in \RR^{n}$ with the following (i)-(iii) satisfied.
    \begin{equation}\label{the choice of an integer shift}
    (i) \hspace{1mm} \frac{\delta}{3} < \left( z_{0} - x_{0}\right) \cdot\nu < \delta; \hspace{2mm} (ii) \hspace{1mm} |z_{0} - 2x_{0}| < \frac{R}{3};
    (iii) \hspace{1mm} z_{0} - x_{0} \in\ZZ^{n}.
    \end{equation}
	\underline{Step 3.} Let us investigate the ordering relation between $\overline{\mathrm{U}}_{a_{2}}(x,t)$ and $\mathrm{U}^{h_{-}}(x,t)$ in the cylinder domain $\Upomega\left( x_{0}, \frac{R}{3}; \nu\right)$. Let us choose 
	\begin{equation}\label{the data for the birkhoff property}
	\Delta z := z_{0} - x_{0}; \hspace{2mm} \sigma := |\Delta z \cdot \nu|; \hspace{2mm} \Delta t := \frac{\sigma}{s_{1}} - \frac{\sigma}{s_{2}}; \hspace{2mm} \mathscr{R} := \frac{4M_0 R}{\delta}.
	\end{equation}
	Then Proposition \ref{the birkhoff property for subsolution in an expanding domain} indicates that
	\begin{equation}\label{birkhoff applied to the subsolution}
	\overline{\mathrm{U}}_{a_{2}}\left( x, t\right) \leq  \overline{\mathrm{U}}_{a_{2}}\left(x - \Delta z, t - \frac{\sigma}{s_{2}}\right), \hspace{2mm} x \in \Upomega\left( x_{0}, \frac{R}{3}; \nu\right), \hspace{2mm}\frac{\sigma}{s_{2}} \leq t \leq t_{0} - \Delta t.
	\end{equation}
    Because of the inclusion $\Upomega\left( x_{0}, \frac{R}{3}; \nu\right) + \Delta z \subseteq \Upomega(0,R;\nu)$ and (\ref{the moment of the first crossing}), we have that
    \begin{equation}\label{the ordering relation between the subsolution and the supersolution at an earlier moment}
    \overline{\mathrm{U}}_{a_{2}}\left(y, t - \frac{\sigma}{s_{2}}\right) < \mathrm{U}^{h_{-}}\left(y, t - \frac{\sigma}{s_{2}}\right), \hspace{2mm} y \in \Upomega\left(0, R; \nu\right), \hspace{2mm}\frac{\sigma}{s_{2}} \leq t \leq t_{0} - \Delta t.
    \end{equation} 	
    The Proposition \ref{the birkhoff property for supersolution in an expanding domain} applied to the supersolution $\underline{\mathrm{U}}_{a_{1}}$ with above shift $\Delta z$ shows that
    \begin{eqnarray*}
    \underline{\mathrm{U}}_{a_{1}}\left(y, t - \frac{\sigma}{s_{2}}\right) &\leq& \underline{\mathrm{U}}_{a_{1}}\left(y + \Delta z, t + \frac{\sigma}{s_{1}} - \frac{\sigma}{s_{2}}\right)\\
    &=& \underline{\mathrm{U}}_{a_{1}}\left(y + \Delta z, t + \Delta t\right), \hspace{1cm} y \in \Upomega\left(0, R; \nu\right), \hspace{2mm}\frac{\sigma}{s_{2}} \leq t \leq t_{0} - \Delta t.
    \end{eqnarray*}
    Then apply the $\xi_{0}$ shift and the $h$ inf-convolution to both sides, we conclude that
    \begin{equation}\label{birkhoff applied to the supersolution}
    \mathrm{U}^{h_{-}}\left(y,t - \frac{\sigma}{s_{2}}\right) \leq \mathrm{U}^{h_{-}}\left(y + \Delta z, t + \Delta t\right), \hspace{1mm} y \in \Upomega\left(0, R; \nu\right), \hspace{1mm}\frac{\sigma}{s_{2}} \leq t \leq t_{0} - \Delta t.
    \end{equation}
    A combination of (\ref{birkhoff applied to the subsolution}), (\ref{the ordering relation between the subsolution and the supersolution at an earlier moment}) and (\ref{birkhoff applied to the supersolution}) gives the relation
    \begin{eqnarray*}
    \overline{\mathrm{U}}_{a_{2}}\left( x, t\right) < \mathrm{U}^{h_{-}}\left( x, t + \Delta t\right), && x \in \Upomega\left( x_{0}, \frac{R}{3}; \nu\right), \hspace{2mm} \frac{\sigma}{s_{2}} < t < t_{0} - \Delta t.
    \end{eqnarray*}
	In particular, 
	\begin{eqnarray*}
	\overline{\mathrm{U}}_{a_{2}}(x,t_{0} - \Delta t) < \mathrm{U}^{h_{-}}(x,t_{0}), && x \in \Upomega\left( x_{0}, \frac{R}{3}; \nu\right). 
	\end{eqnarray*}
	By the choice of $\Delta t$ through (\ref{the moment of the first crossing}), (\ref{the small vertical distance}), (\ref{the choice of an integer shift}) and (\ref{the data for the birkhoff property}),
	\begin{equation*}
	0 < \Delta t < \frac{(C - 1)\gamma^{2}(t_{0})}{(n - 1)C^{2} + C(C - 1)\gamma(t_{0})M_0}.
	\end{equation*}
	Then Proposition \ref{finite speed propagation of subsolution} implies that
	\begin{eqnarray*}
	\overline{\mathrm{U}}_{a_{2}}(x,t_{0}) < \mathrm{U}^{\hat{h}_{-}}(x,t_{0}), \hspace{2mm} x \in \Upomega\left( x_{0}, \frac{R}{3}; \nu\right), &\text{where}& \hat{h} := \left( \left( 1 - \frac{1}{C}\right)\gamma(t), 1 \right).
	\end{eqnarray*}
	\underline{Step 4.} Let us construct $\varphi(x): \Upomega\left( x_{0}, \frac{R}{3}; \nu\right)  \rightarrow \RR$ as follows:
	\begin{eqnarray*}
	\varphi(x) := - \frac{9(1 + C)}{2CR^{2}}\left|(x - x_{0})^{\top}\right|^{2} + \frac{1}{2}\left( 3 - \frac{1}{C}\right), 
	\end{eqnarray*}
	where $(x - x_{0})^{\top} := (x - x_{0}) - ((x - x_{0})\cdot\nu)\nu$. Then $\varphi(x)$ satisfies that
	\begin{eqnarray*}
		\begin{cases}
			\varphi|_{\left\lbrace x | (x - x_{0})^{\top} = 0\right\rbrace } = \frac{1}{2}\left( 3 - \frac{1}{C}\right) > 1, \\
			\varphi|_{\left\lbrace x | |(x - x_{0})^{\top}| = \frac{R}{3}\right\rbrace } = 1 - \frac{1}{C},
		\end{cases} &\text{and}& \begin{cases}
		\lVert D\varphi\rVert_{\infty} \leq \frac{6}{R},\\
		\lVert D^{2}\varphi\rVert_{\infty} \leq \frac{18}{R^{2}}.
	\end{cases}
\end{eqnarray*}
Because $R \geq \max\left\lbrace 6, \frac{12(3n + M_0 + 27)}{L_0} \right\rbrace $, then
\begin{eqnarray*}
\gamma^{\prime}(t) + \left( \frac{\left( n + 1\right) \lVert D^{2}\varphi \rVert_{\infty}}{\varphi(x)} + \frac{M_0|D\varphi(x)|}{\varphi(x)}
+ L_0\right)\gamma(t) + \frac{|D\varphi(x)|^{2}\gamma(t)}{\left( 1 - \gamma(t)|D\varphi(x)|\right)^{2}\varphi^{2}(x)} &\leq& 0.
\end{eqnarray*}
Then by Proposition \ref{the evolution law of inf convolution}, $\mathrm{U}^{(\gamma,\varphi)_{-}}(x,t)$ is a pseudo viscosity supersolution. Based on the finite speed of propagation regarding the subsolution (c.f. Proposition \ref{finite speed propagation of subsolution}), we have for some $\tau > 0$ that
\begin{equation}\label{the strict ordering relation on the boundary in the local comparison principle}
\overline{\mathrm{U}}_{a_{2}}(x,t) < \mathrm{U}^{(\gamma,\varphi)_{-}}(x,t), \hspace{2mm} \frac{1}{2(M_0 - m_0)}\leq t \leq t_{0} + \tau, \hspace{2mm} x \in \partial\Upomega\left( x_{0}, \frac{R}{3}; \nu\right). 
\end{equation}
Note that the following strict ordering relation holds.
\begin{equation}\label{an strict ordering relation at the initial moment}
\overline{\mathrm{U}}_{a_{2}}\left( x, \frac{1}{2(M_0 - m_0)}\right)  < \mathrm{U}^{(\gamma,\varphi)_{-}}\left( x, \frac{1}{2(M_0 - m_0)}\right), \hspace{2mm} x \in \Upomega\left( x_{0}, \frac{R}{3}; \nu\right).
\end{equation}
Recall that $x_{0}$ is the first touching point and $\overline{\mathrm{U}}_{a_{2}}(x_{0},t_{0}) = \mathrm{U}^{h_{-}}(x_{0},t_{0}) = \mu$, which is a contradiction due to Proposition \ref{the comparison principle regarding pseudo viscosity solutions}.

\end{proof}

\section{Head and tail speeds}\label{the section of head and tail speeds}

\subsection{Detachment}

\begin{definition}\label{the definition of detachment}
	Let $a = (\nu, R, \mathscr{R}, q, s) \in \mathbb{A}$ with $\mathscr{R} \geq 0$ and set $e := (\nu, q, s) \in \mathbb{E}$.
	\begin{enumerate}
		\item [(1)] Let $\mu \in \RR$, the obstacle subsolution $\overline{\mathrm{U}}_{a}(x,t)$ (resp. supersolution $\underline{\mathrm{U}}_{a}(x,t)$) detaches from the associated obstacle $\mathrm{O}_{e}(x,t)$ at the  $\mu$ level set if there exist $r > \frac{\sqrt{n}}{2}$ and $T > 0$, such that (c.f. Definition \ref{the strict ordering relation between level sets of two functions})
		\begin{eqnarray*}
		\overline{\mathrm{U}}_{a}(\cdot, t) \prec_{(\Upomega(0,r; \nu), \mu)} \mathrm{O}_{e}(\cdot, t), \hspace{2mm} t > T; \hspace{2mm} (\text{resp.} \hspace{2mm} \mathrm{O}_{e}(\cdot, t) \prec_{(\Upomega(0,r; \nu), \mu)}  \underline{\mathrm{U}}_{a}(\cdot, t), \hspace{2mm} t > T ).
		\end{eqnarray*}
		\item [(2)] The obstacle subsolution $\overline{\mathrm{U}}_{a}(x,t)$ (resp. supersolution $\underline{\mathrm{U}}_{a}(x,t)$) detaches from the associated obstacle $\mathrm{O}_{e}(x,t)$ if for any $\mu \in \RR$, there exists $T = T(\mu) > 0$, such that
		\begin{eqnarray*}
			\overline{\mathrm{U}}_{a}(\cdot, t) \prec_{(\Upomega(0,r; \nu), \mu)} \mathrm{O}_{e}(\cdot, t), \hspace{2mm} t > T(\mu); \hspace{2mm}
			(\text{resp.} \hspace{2mm} \mathrm{O}_{e}(\cdot, t) \prec_{(\Upomega(0,r; \nu), \mu)} \underline{\mathrm{U}}_{a}(\cdot, t), \hspace{2mm} t > T(\mu)). 
		\end{eqnarray*}
		\item [(3)] The obstacle subsolution $\overline{\mathrm{U}}_{a}(x,t)$ (resp. supersolution $\underline{\mathrm{U}}_{a}(x,t)$) detaches uniformly from the associated obstacle $\mathrm{O}_{e}(x,t)$ if there exist $r > \frac{\sqrt{n}}{2}$ and $T > 0$, such that
		\begin{eqnarray*}
			\overline{\mathrm{U}}_{a}(\cdot, t) \prec_{(\Upomega(0,r; \nu), \mu)} \mathrm{O}_{e}(\cdot, t), && t > T \hspace{2mm} \text{for any} \hspace{2mm} \mu \in \RR
		\end{eqnarray*}
		\begin{eqnarray*}
			(\text{resp.} \hspace{2mm} \mathrm{O}_{e}(\cdot, t)\prec_{(\Upomega(0,r; \nu), \mu)} \underline{\mathrm{U}}_{a}(\cdot, t), && t > T \hspace{2mm} \text{for any} \hspace{2mm} \mu \in \RR).
		\end{eqnarray*}
	\end{enumerate}
\end{definition}

\begin{lemma}\label{the geometric property in the detachment relation}
	Let $a := (\nu, R, \mathscr{R}, q, s) \in \mathbb{A}$ with $\mathscr{R} \geq 0$, and then set $e := (\nu, q, s) \in \mathbb{E}$. Assume that $\overline{\mathrm{U}}_{a}$ (resp. $\underline{\mathrm{U}}_{a}$) (uniformly) detaches from $\mathrm{O}_{e}$. Then for any number $\zeta > 0$, $\overline{\mathrm{U}}_{\hat{a}}$ (resp. $\underline{\mathrm{U}}_{\hat{a}}$) (uniformly) detaches from $\mathrm{O}_{\hat{e}}$, where
	\begin{eqnarray*}
	\hat{a} := (\nu, R, \mathscr{R}, \zeta q, s) &\text{and}& \hat{e} := (\nu, \zeta q, s).
	\end{eqnarray*}
\end{lemma}

\begin{proof}
	It follows from Definition \ref{the definition of detachment} and the facts as follows.
	\begin{equation*}
	\overline{\mathrm{U}}_{\hat{a}}(x,t) = \zeta \overline{\mathrm{U}}_{a}(x,t), \hspace{2mm} \underline{\mathrm{U}}_{\hat{a}}(x,t) = \zeta \underline{\mathrm{U}}_{a}(x,t), \hspace{2mm} \mathrm{O}_{\hat{e}}(x,t) = \zeta \mathrm{O}_{e}(x,t).
	\end{equation*}
\end{proof}

\begin{proposition}\label{a semi obstacle solution becomes a solution when detachment happens}
	Let $a := (\nu, R, \mathscr{R}, q, s) \in \mathbb{A}$, then set $d := (\nu, R, \mathscr{R}) \in \mathbb{D}$ and $e := (\nu, q, s) \in \mathbb{E}$, for any space time domain $\Sigma \subseteq \mathrm{C}_{d}$ (see (\ref{the space time domain})), we have that 
	\begin{enumerate}
		\item [(i)] If $\overline{\mathrm{U}}_{a}(x,t) < \mathrm{O}_{e}(x,t)$, for $(x,t) \in \Sigma$, then $\left( \overline{\mathrm{U}}_{a}\right)_{*}$ is a supersolution in $\Sigma$;
		\item [(ii)] If $\mathrm{O}_{e}(x,t) < \underline{\mathrm{U}}_{a}(x,t)$, for $(x,t) \in \Sigma$, then $\left( \underline{\mathrm{U}}_{a}\right)^{*}$ is a subsolution in $\Sigma$.
	\end{enumerate}
\end{proposition}

\begin{proof}
	Let us only prove $(i)$ since $(ii)$ can be established by parallel arguments. Let $\phi(x,t)$ be a $C^{2,1}$ function on $\Sigma$, assume that $\left( \overline{\mathrm{U}}_{a}\right)_{*}(x,t) - \phi(x,t)$ obtains its local mininum at $(x_{0}, t_{0}) \in \Sigma$, without loss of generality, we can assume that $\left( \overline{\mathrm{U}}_{a}\right)_{*}(x_{0},t_{0}) = \phi(x_{0},t_{0})$. To prove that $\left( \overline{\mathrm{U}}_{a}\right)_{*}$ is a viscosity supersolution in $\Sigma$, let us assume on the contrary that
	\begin{equation*}
	\phi_{t}(x_{0},t_{0}) < \mathscr{F}_{*}\left( D^{2}\phi(x_{0},t_{0}), D\phi(x_{0},t_{0}), x_{0}\right). 
	\end{equation*}
   Then by the continuity of $\phi_{t}$, $D\phi$, $D^{2}\phi$ and the lower semi continuity of $\mathscr{F}_{*}$, there exists a neighborhood of $(x_{0},t_{0})$ (let us still denote it by $\Sigma$), where the above strict inequality holds. By subtracting a multiple of $|x - x_{0}|^{4} + (t - t_{0})^{2}$ from $\phi(x,t)$ if necessary (this does not change the above inequality), we can assume that $\left( \overline{\mathrm{U}}_{a}\right)_{*}(x,t) - \phi(x,t)$ has a strict minimum at $(x_{0},t_{0})$, over $\Sigma$. Then there exists $\delta > 0$, such that 
   \begin{enumerate}
   	\item [(a)] $\phi(x,t) + \delta < \mathrm{O}_{e}(x,t)$, for $(x,t) \in \Sigma$;
   	\item [(b)] $\phi(x,t) + \delta < \left( \underline{\mathrm{U}}_{a}\right)_{*}(x,t)$, for $(x,t) \in \partial \Sigma$.
   \end{enumerate}
   Then we can define the function as follows, 
   \begin{equation*}
   \tilde{\mathrm{U}}(x,t) := \max\left\lbrace \overline{\mathrm{U}}_{a}(x,t), \phi(x,t) + \delta\right\rbrace, 
   \end{equation*}
   which is also a viscosity subsolution in $\mathrm{C}_{d}$. On the other hand, there exists $(\hat{x},\hat{t})$ arbitrarily close to $(x_{0}, t_{0})$, such that $\overline{\mathrm{U}}_{a}(\hat{x},\hat{t}) < \phi(\hat{x},\hat{t}) + \delta$. Therefore, $\tilde{\mathrm{U}}(\hat{x},\hat{t}) > \overline{\mathrm{U}}_{a}(\hat{x},\hat{t})$, which contradicts to the maximality of $\overline{\mathrm{U}}_{a}$.
\end{proof}

\subsection{Irrational directions}\label{the subsection of irrational directions}

\subsubsection{Basic facts of head/tail speed}
\begin{definition}\label{the head speed in an irrational direction}
	Let $\nu \in \mathbb{S}^{n-1}\diagdown \RR\ZZ^{n}$, the head speed in $\nu$ direction, denoted by $\overline{s}(\nu)$, is defined as the smallest number, such that: for any $\delta > 0$, there exists $R > 0$, such that $\overline{\mathrm{U}}_{a}(x,t)$ detaches (c.f. Definition \ref{the definition of detachment}) from $\mathrm{O}_{e}(x,t)$, where
	\begin{eqnarray*}
	a := (\nu, R, 0, q, \overline{s}(\nu) + \delta) \in \mathbb{A}, && e := (\nu, q, \overline{s}(\nu) + \delta) \in \mathbb{E}.
	\end{eqnarray*}
\end{definition}

\begin{definition}\label{the tail speed in an irrational direction}
	Let $\nu \in \mathbb{S}^{n-1}\diagdown \RR\ZZ^{n}$, the tail speed in $\nu$ direction, denoted by $\underline{s}(\nu)$, is defined as the largest number, such that: for any $\delta > 0$, there exists $R > 0$, such that $\underline{\mathrm{U}}_{a}(x,t)$ detaches (c.f. Definition \ref{the definition of detachment}) from $\mathrm{O}_{e}(x,t)$, where
	\begin{eqnarray*}
		a := (\nu, R, 0, q, \underline{s}(\nu) - \delta) \in \mathbb{A}, && e := (\nu, q, \underline{s}(\nu) - \delta) \in \mathbb{E}.
	\end{eqnarray*}
\end{definition}

\begin{remark}
	By Lemma \ref{the geometric property in the detachment relation}, the head speed $\overline{s}(\nu)$ and the tail speed $\underline{s}(\nu)$ are independent of $q$, therefore, they are both well-defined.
\end{remark}

\begin{lemma}
	Fix any $\nu \in \mathbb{S}^{n-1}\diagdown \RR\ZZ^{n}$, then $m_0 \leq \overline{s}(\nu), \underline{s}(\nu) \leq M_0$.
\end{lemma}

\begin{proof}
	Let us only prove $m_0 \leq \overline{s}(\nu) \leq M_0$, since a parallel argument applies to $\underline{s}(\nu)$. Fix any $R$, $q$ and $s < m_0$ (resp. $s > M_0$), such that $a := (\nu, R, 0, q, s) \in \mathbb{A}$, then set $e := (\nu, q, s) \in \mathbb{E}$. Then $\mathrm{O}_{e}(x,t)$ is a subsolution (resp. supersolution) and $\overline{\mathrm{U}}_{a}(x,t) \equiv \mathrm{O}_{e}(x,t)$ (resp. $\underline{\mathrm{U}}_{a}(x,t) \equiv \mathrm{O}_{e}(x,t)$). In this case, there is no detachment between $\overline{\mathrm{U}}_{a}$ (resp. $\underline{\mathrm{U}}_{a}$) and $\mathrm{O}_{e}$, hence $\overline{s}(\nu) \geq m_0$ (resp. $\overline{s}(\nu) \leq M_0$).
\end{proof}

The next Lemma says that in the case of expanding domain, i.e., $\mathscr{R} > 0$, if the detachment happens at certain level set, then for any $r > 0$, the sub and super solution shall be away from the obstacle in $\Upomega(0,r;\nu)$, after certain amount of time. i.e., the detachment expands as time evolves.
\begin{lemma}\label{the expansion of detachment}
	Let $a := (\nu, R, \mathscr{R}, q, s) \in \mathbb{A}$ with $\mathscr{R} > 0$, $\mu \in \RR$ and then set $e := (\nu, q, s) \in \mathbb{E}$, if $\overline{\mathrm{U}}_{a}(x,t)$ (resp. $\underline{\mathrm{U}}_{a}(x,t)$) detaches from $\mathrm{O}_{e}(x,t)$ in the $\mu$ level set (c.f. Definition \ref{the definition of detachment}), then for any $r > 0$, there exists $T := T(\mu, s, \mathscr{R}) > 0$, such that
	\begin{eqnarray*}
	\overline{\mathrm{U}}_{a}(\cdot,t) \prec_{(\Upomega(0,r; \nu),\mu)} \mathrm{O}_{e}(\cdot,t), \hspace{2mm} ( \text{resp.}\hspace{2mm} \underline{\mathrm{U}}_{a}(\cdot,t) \prec_{(\Upomega(0,r; \nu), \mu)} \mathrm{O}_{e}(\cdot,t)), && t > T.
	\end{eqnarray*}
\end{lemma}

\begin{proof}
	Let us only prove the case associated to the obstacle subsolution $\overline{\mathrm{U}}_{a}$, since the other case can be argued in a similar way. By Definition \ref{the definition of detachment}, there exist $r_{0} > \frac{\sqrt{n}}{2}$ and $T_{0} = T_{0}(\mu) > 0$, such that
	\begin{eqnarray*}
	\overline{\mathrm{U}}_{a}(\cdot,t) \prec_{(\Upomega(0,r_{0};\nu),\mu)} \mathrm{O}_{e}(\cdot,t), && t > T_{0}.
	\end{eqnarray*}
    Denote
    \begin{equation*}
    K := \Upomega(0, r_{0}; \nu) \cap \left\lbrace x \in \RR^{n} \big| sT_{0} \leq x \cdot \nu \leq sT_{0} + \sqrt{n}\right\rbrace. 
    \end{equation*}
    For any $x \in \Upomega(0,r;\nu)\diagdown \Upomega(0,r_{0};\nu)$, there exists $B \in \RR$, such that
    \begin{equation*}
    \tilde{x} := x + B\nu \in \left\lbrace y\in\RR^{n} \big| y\cdot \nu > s\left(T_{0} + \frac{\sqrt{n}}{s} + \frac{r + r_{0}}{\mathscr{R}}\right) \right\rbrace.
    \end{equation*}
    Then there exists $\tilde{\omega} \in K$, such that
    \begin{eqnarray*}
    \Delta z := \tilde{x} - \tilde{\omega} \in \ZZ^{n} &\text{and}& x - \Delta z \in \Upomega(0, r_{0}; \nu).
    \end{eqnarray*}
    Set $\Delta t := \frac{\Delta z \cdot \nu}{s}$, then $\Delta z$ and $\Delta t$ satisfy the assumptions in Proposition \ref{the birkhoff property for subsolution in an expanding domain}. Let us take $T(\mu, s, \mathscr{R}) := T_{0} + \Delta t$, then we have for any $t > T$ that
    \begin{eqnarray*}
    \overline{\mathrm{U}}_{a}(x,t) \leq \overline{\mathrm{U}}_{a}(x - \Delta z, t - \Delta t) \prec_{(\Upomega(0,r_{0};\nu), \mu)} \mathrm{O}_{e}\left(x - \Delta z, t - \Delta t \right) = \mathrm{O}_{e}(x,t).
    \end{eqnarray*}
\end{proof}

The next Lemma is a static version of Lemma \ref{the expansion of detachment}. i.e., we can expect that the sub and supersolution be away from its obstacle in $\Upomega(0,r;\nu)$, for any $r > 0$, as long as the obstacle sub and supersolution is initially defined in $\Upomega(0, R; \nu)$ with sufficiently large $R$.

\begin{lemma}\label{the expansion of detachment in a static domain for the obstacle subsolution}
	Fix $\nu \in \mathbb{S}^{n-1}\diagdown \RR\ZZ^{n}$, $\delta > 0$ and set $s := \overline{s}(\nu) + \delta$. Then for any $\mu \in \RR$ and $r > 0$, there exist $R = R(\nu, \mu,\delta, r) > 0$ and $T = T(\nu, \mu,\delta) > 0$, such that
	\begin{eqnarray*}
	\overline{\mathrm{U}}_{a}(\cdot,t) \prec_{(\Upomega(0,r; \nu), \mu)} \mathrm{O}_{e}(\cdot,t), && t > T,
	\end{eqnarray*}
    where
    \begin{eqnarray*}
    a := (\nu, R, 0, q, s) \in \mathbb{A} &\text{and}& e := (\nu, q, s) \in \mathbb{E}.
    \end{eqnarray*}
\end{lemma}

\begin{proof}
	As $s > \overline{s}(\nu)$, there exist three numbers
	\begin{eqnarray*}
	r_{0} > \frac{\sqrt{n}}{2}, \hspace{2mm} R_{0} = R_{0}(\nu, \mu, \delta) > 0 &\text{and}& T_{0} = T_{0}(\nu, \mu, \delta)> 0,
	\end{eqnarray*}
	such that (where $a_{0} := (\nu, R_{0}, 0, q, s) \in \mathbb{A}$)
	\begin{eqnarray}\label{detachment assumption in proving detachment in a large domain}
	\overline{\mathrm{U}}_{a_{0}}(\cdot,t) \prec_{(\Upomega(0,r_{0}; \nu),\mu)} \mathrm{O}_{e}(\cdot,t), && t > T_{0}.
	\end{eqnarray}
    Let us take $R$ and $T$, such that
    \begin{eqnarray*}
    R > R_{0} + r + r_{0} &\text{and}& T > T_{0} + \frac{2\sqrt{n}}{m_0}.
    \end{eqnarray*}
    For any $x \in \Upomega(0,r;\nu)$, let us choose $A \in \RR$, such that 
    \begin{eqnarray*}
    \tilde{x}\cdot\nu - T_{0}s \in \left(\sqrt{n}, 2\sqrt{n}\right) &\text{where}& \tilde{x} = x + A\nu.
    \end{eqnarray*}
    Then there exists $\hat{x} \in \Omega(0,r_{0})$, such that
    \begin{eqnarray*}
    \hat{x}\cdot \nu - T_{0}s \in (0, \sqrt{n}) &\text{and}& \Delta z := \tilde{x} - \hat{x} \in \ZZ^{n}.
    \end{eqnarray*}
    Let us set $\Delta t := \frac{\Delta z \cdot \nu}{s}$, then by Proposition \ref{the birkhoff property for subsolutions in two static domains}, we conclude that
    \begin{eqnarray*}
    \overline{\mathrm{U}}_{a}(x, t) = \overline{\mathrm{U}}_{a}(x - \Delta z + \Delta z, t - \Delta t + \Delta t) \leq \overline{\mathrm{U}}_{a_{0}}(x - \Delta z, t - \Delta t), &&  x \in \Upomega(0,R; \nu).
    \end{eqnarray*}
    Then for any $t > T$, we have that $t - \Delta t > T_{0}$, hence
    \begin{eqnarray*}
    \overline{\mathrm{U}}_{a}(\cdot, t) \prec_{(\Upomega(0, r; \nu), \mu)} \mathrm{O}_{e}(\cdot,t), && t > T,
    \end{eqnarray*}
    where we have used (\ref{detachment assumption in proving detachment in a large domain}) and the fact that $\mathrm{O}_{e}(x - \Delta z, t - \Delta t) = \mathrm{O}_{e}(x,t)$.

\end{proof}

\begin{lemma}\label{the expansion of detachment in a static domain for the obstacle supersolution}
	Fix $\nu \in \mathbb{S}^{n-1}\diagdown \RR\ZZ^{n}$, $\delta > 0$, and set $s := \underline{s}(\nu) - \delta$. Then for any $\mu \in \RR$ and $r > 0$, there exist $R = R(\nu, \mu, \delta,r) > 0$ and $T = T(\nu, \mu, \delta)> 0$, such that
	\begin{eqnarray*}
		 \mathrm{O}_{e}(\cdot,t) \prec_{(\Upomega(0,r;\nu), \mu)} \underline{\mathrm{U}}_{a}(\cdot,t), && t > T,
	\end{eqnarray*}
	where
	\begin{eqnarray*}
		a := (\nu, R, 0, q, s) \in \mathbb{A} &\text{and}& e := (\nu, q, s) \in \mathbb{E}.
	\end{eqnarray*}
\end{lemma}

\begin{proof}
	It is similar to the above Lemma \ref{the expansion of detachment in a static domain for the obstacle subsolution}. Instead of using Proposition \ref{the birkhoff property for subsolutions in two static domains}, Proposition \ref{the birkhoff property for supersolutions in two static domains} should be applied. Hence, we omit the details here.
\end{proof}

\subsubsection{The second version of local comparison principles}
Recall that in Proposition \ref{LCP in the unit scale}, we established an ordering relation between an obstacle subsolution with a fast obstacle speed and an obstacle supersolution with a slow obstacle speed. In this part, we shall compare two obstacle subsolutions with the obstacle speed over its tail speed. In this case, the obstacle subsolution detaches from its obstacle, so it is actually a supersolution. In order to apply the Birkhoff property with a right monotonicity, we shall require a shrinking domain. 

\begin{proposition}\label{the second local comparison principle subsolution}
	Let $(\nu, T, R)$ be a comparison consistent triplet (c.f. Definition \ref{comparison consistent triplet}), fix $\mu \in \RR$ and $s_{i} := \overline{s}(\nu) + \delta_{i}$, $i = 1, 2$, where $\delta_{2} > \delta_{1} > 0$ are two fixed numbers. Then for any $r > \max \left\lbrace R, R_{0}, \frac{12(3n + M_0 + 27)}{L_0}\right\rbrace $, $T > 0$, where $R_{0}$ is the radius such that the detachment occurs for the following $\overline{\mathrm{U}}_{a_{i}}$, $i = 1, 2$. Then there exist $A > 0$ (independent of $r$ and $T$), $R_{i}$ and $\mathscr{R}_{i} > 0$, $i = 1, 2$, such that the following (i) and (ii) hold.
	\begin{enumerate}
		\item [(i)] $R_{1} \geq R_{2} + \left( \mathscr{R}_{1} + \mathscr{R}_{2}\right)T$;
		\item [(ii)] $\overline{\mathrm{U}}_{a_{2}}(\cdot,t) \prec_{(\Upomega(0,r; \nu),\mu)} \overline{\mathrm{U}}_{a_{1}}(\cdot - \xi_{A}, t)$, $0 \leq t \leq T$,
	\end{enumerate}
    where
    \begin{equation*}
    a_{1} := (\nu, R_{1}, -\mathscr{R}_{1}, q, s_{1}), \hspace{2mm} a_{2} := (\nu, R_{2}, \mathscr{R}_{2}, q, s_{2}) \in \mathbb{A} \hspace{2mm} \text{and} \hspace{2mm} \xi_{A} \in \argmin\limits\limits_{\xi \in \ZZ^{n},\hspace{1mm} \xi \cdot \nu > A} |\xi|.
    \end{equation*}
\end{proposition}

\begin{proof}
	The key idea of the argument is similar to that of Proposition \ref{LCP in the unit scale}. Since there are several modifications in both the Proposition and the proof, we still provide the details as follows. Because $s_{i} > \overline{s}(\nu)$, there exist $T_{0} > 0$ and $R_{0}  > 2r$, such that
	\begin{eqnarray*}
	\overline{\mathrm{U}}_{b_{i}}(\cdot,t) \prec_{(\Upomega(0,2r; \nu),\mu)} \mathrm{O}_{e}(\cdot,t), \hspace{2mm} t > T_{0}, &\text{where}& b_{i} := (\nu, R_{0}, 0, q, s_{i}), \hspace{2mm} i = 1, 2.
	\end{eqnarray*}
   Let us set $A := (s_{2} - m_0)T_{0} + 1$ and $R_{2} > R_{0}$, then 
   \begin{eqnarray*}
   \overline{\mathrm{U}}_{a_{2}}(\cdot,t) \prec_{(\Upomega(0,2r;\nu), \mu)} \left( \overline{\mathrm{U}}_{a_{1}}\right)_{*}(x - \xi_{A},t), && 0 \leq t \leq T_{0}.
   \end{eqnarray*}
   Recall the Proposition \ref{the birkhoff property for subsolutions in two static domains} (with $\Delta z = 0$ and $\Delta t = 0$), we get that
   \begin{eqnarray*}
   \overline{\mathrm{U}}_{a_{i}}(\cdot,t) \prec_{(\Upomega(0,2r;\nu),\mu)} \overline{\mathrm{U}}_{b_{i}}(\cdot,t) \prec_{(\Upomega(0,2r;\nu),\mu)} \mathrm{O}_{e}(\cdot,t), && T_{0} < t < \infty, \hspace{2mm} i = 1, 2.
   \end{eqnarray*} 
   Let us set
   \begin{eqnarray*}
   \mathrm{Z}(x,t) &:=& \begin{cases}
   1, & x \in L_{\mu}^{+}\left( \overline{\mathrm{U}}_{a_{2}}(\cdot, t); \Upomega(0, 2r; \nu)\right),\\
   -1, & x \in \Upomega(0, 2r; \nu) \diagdown L_{\mu}^{+}\left( \overline{\mathrm{U}}_{a_{2}}(\cdot, t); \Upomega(0, 2r; \nu)\right),
   \end{cases}\\
   \mathrm{Y}(x,t) &:=& \begin{cases}
   	2, & x \in \Upomega(0, 2r; \nu) \diagdown L_{\mu}^{-}\left(\left( \overline{\mathrm{U}}_{a_{1}}\right)_{*}(\cdot - \xi_{A}, t); \Upomega(0, 2r; \nu) \right), \\
   	0, & x \in L_{\mu}^{-}\left(\left(  \overline{\mathrm{U}}_{a_{1}}\right)_{*}(\cdot - \xi_{A},t); \Upomega(0, 2r; \nu)\right). 
   \end{cases}
   \end{eqnarray*}
   Because $\overline{\mathrm{U}}_{a_{1}}(x,t)$ detaches from the obstacle $\mathrm{O}_{e}(x,t)$ in the $\mu$ level set, by upper semicontinuity of $\overline{\mathrm{U}}_{a_{1}}(x,t)$, Proposition \ref{a semi obstacle solution becomes a solution when detachment happens}, the operator $\mathscr{F}(\cdot, \cdot, \cdot)$ is geometric and the fact that $\xi_{A} \in \ZZ^{n}$, we can conclude that $\mathrm{Y}(x,t)$ is a viscosity supersolution for $x\in \Upomega(0, 2r; \nu)$ and $t > T_{0}$. In order to prove the Proposition, we assume on the contrary as follows, where $\mathrm{U}(\cdot, t) := \left( \overline{\mathrm{U}}_{a_{1}}\right)_{*}(\cdot - \xi_{A}, t)$:
   \begin{eqnarray}\label{the opposite assumption in the second local comparison principle}
   \sup\left\lbrace t > 0 \big| \overline{\mathrm{U}}_{a_{2}}(\cdot,t) \prec_{(\Upomega(0,r; \nu), \mu)} \mathrm{U}(\cdot,t)\right\rbrace \leq T.
   \end{eqnarray}
   \underline{Step 1.} Let us define (recall Definition \ref{comparison consistent triplet})
   \begin{eqnarray*}
   \gamma(t) := \frac{1}{2}e^{-2L_0t} &\text{and}& h := (\gamma(t), 1) \in \mathbb{F}.
   \end{eqnarray*}
   Let $\mathrm{U}^{h_{-}}(x,t)$ be the $h$ inf-convolution of $\mathrm{U}(x,t)$ by Definition \ref{the definition of the inf convolution}. Then the above assumption (\ref{the opposite assumption in the second local comparison principle}) implies that
   \begin{equation}\label{the first touching time in the second local comparison principle}
   T_{0} < t_{0} := \sup \left\lbrace t > 0 \big| \overline{\mathrm{U}}_{a_{2}}(\cdot,t) \prec_{(\Upomega(0,r;\nu), \mu)} \mathrm{U}^{h_{-}}(\cdot,t) \right\rbrace \leq T.
   \end{equation}
   Assume $x_{0}$ is the first touching point between the $\mu$ superlevel set of $\overline{\mathrm{U}}_{a_{2}}$ and the $\mu$ sublevel set of $\mathrm{U}^{h_{-}}$ in $\Upomega(0,r;\nu)$. i.e.,
   \begin{eqnarray*}
   \overline{\mathrm{U}}_{a_{2}}(x_{0},t_{0}) = \mathrm{U}^{h_{-}}(x_{0},t_{0}) = \mu &\text{and}&  \mathrm{Z}(x_{0},t_{0}) = 1 > 0 = \mathrm{Y}^{h_{-}}(x_{0},t_{0}).
   \end{eqnarray*}
   By applying the comparison principle (i.e., Proposition \ref{the usual comparison principle}) to $\mathrm{Z}$ and $\mathrm{Y}$ in $\Upomega(0,r;\nu)$, the maximum difference of $\mathrm{Z}(x,t) - \mathrm{Y}(x,t)$ is achieved on $\partial\Upomega(0,r;\nu)$. As $h$ does not depend on the space variable, we must have that $x_{0} \in \partial\Upomega(0,r;\nu)$.\\
  \underline{Step 2.}   
  By the choice of $r$, there exists $z_{0} \in \RR^{n}$ with the following (i)-(iii) hold:
  \begin{equation}\label{the choice of the integer in the second local comparison principle}
  (i) \hspace{1mm} \frac{\delta}{3} < (z_{0} - x_{0})\cdot\nu < \delta; \hspace{4mm} (ii) \hspace{1mm} |z_{0} - 2x_{0}| < \frac{r}{3}; \hspace{4mm}
  (iii) \hspace{1mm} z_{0} - x_{0} \in \ZZ^{n},
  \end{equation}
  where $0 < \delta < \delta(T)$ with $\delta(T)$ defined in Definition \ref{comparison consistent triplet}. Therefore, there exists $C > 1$, such that $\delta$ can be written as follows:
   \begin{equation}\label{the small vertical shift in the second local comparson principle}
   \delta := \left( \frac{M_0 m_0}{M_0 - m_0}\right) \cdot \frac{(C - 1)\gamma^{2}(T)}{(n - 1)C^{2} + C(C - 1)\gamma(T)M_0} > 0.
   \end{equation}
  \underline{Step 3.} Let us investigate the ordering relation between the $\mu$ superlevel set of $\overline{\mathrm{U}}_{a_{2}}(x,t)$ and the $\mu$ sublevel set of $\mathrm{U}^{h_{-}}(x,t)$ in $\Upomega(x_{0},\frac{r}{3}; \nu)$. Let us introduce the notations
  \begin{equation}\label{some notations regarding the interger shift in the second local comparison principle}
  \Delta z := z_{0} - x_{0}; \hspace{2mm} \sigma := |\Delta z \cdot \nu|; \hspace{2mm} \Delta t := \frac{\sigma}{m_0} - \frac{\sigma}{s_{2}}; \hspace{2mm} \mathscr{R}_{1} = \mathscr{R}_{2} := \frac{4M_0R}{\delta}.
  \end{equation}
  Then Proposition \ref{the birkhoff property for subsolution in an expanding domain} indicates that
  \begin{equation}\label{the birkhoff property applied to the subsolution the second local comparison principle}
  \overline{\mathrm{U}}_{a_{2}}(x,t) \leq \overline{\mathrm{U}}_{a_{2}}\left( x - \Delta z, t - \frac{\sigma}{s_{2}}\right), \hspace{2mm} x \in \Upomega\left(x_{0},\frac{r}{3}; \nu \right), \hspace{2mm} \frac{\sigma}{s_{2}} \leq t \leq t_{0} - \Delta t.
  \end{equation}
  Because of the inclusion $\Upomega(0,\frac{r}{3};\nu) + \Delta z\subseteq \Upomega(0,r;\nu)$ and (\ref{the first touching time in the second local comparison principle}), we have that
  \begin{equation}\label{the ordering relation at an earlier moment the second local comparison principle}
  \overline{\mathrm{U}}_{a_{2}}\left(\cdot, t - \frac{\sigma}{s_{2}}\right) \prec_{(\Upomega(0,r;\nu), \mu)} \mathrm{U}^{h_{-}}\left(\cdot, t - \frac{\sigma}{s_{2}}\right), \hspace{2mm} \frac{\sigma}{s_{2}} \leq t \leq t_{0} - \Delta t  
  \end{equation}
  The Proposition \ref{the birkhoff property for subsolution in a shrinking domain} applied to the subsolution $\overline{\mathrm{U}}_{a_{1}}$ with above shift $\Delta z$ shows that
  \begin{eqnarray*}
  \overline{\mathrm{U}}_{a_{1}}\left(y,t - \frac{\sigma}{s_{2}}\right) \leq \overline{\mathrm{U}}_{a_{1}}\left( y + \Delta z, t + \frac{\sigma}{m_0} - \frac{\sigma}{s_{2}}\right) = \overline{\mathrm{U}}_{a_{1}}\left( y + \Delta z, t + \Delta t\right), \hspace{2mm} y \in \Upomega(0, r; \nu), \hspace{2mm} \frac{\sigma}{s_{2}} \leq t \leq t_{0} - \Delta t.
  \end{eqnarray*}
  Then apply the $\xi_{A}$ shift and the $h$ inf-convolution to both sides, we conclude that
  \begin{equation}\label{the birkhoff property applied to the subsolution in a shrinking domain the second local comparison principle}
  \mathrm{U}^{h_{-}}\left( y, t - \frac{\sigma}{s_{2}}\right) \leq \mathrm{U}^{h_{-}}(y + \Delta z, t + \Delta t), \hspace{2mm} y \in \Upomega(0,r;\nu), \hspace{2mm} \frac{\sigma}{s_{2}} \leq t \leq t_{0} - \Delta t. 
  \end{equation}
  A combination of (\ref{the birkhoff property applied to the subsolution the second local comparison principle}), (\ref{the ordering relation at an earlier moment the second local comparison principle}) and (\ref{the birkhoff property applied to the subsolution in a shrinking domain the second local comparison principle}) gives the relation
  \begin{eqnarray*}
  \overline{\mathrm{U}}_{a_{2}}(\cdot, t) \prec_{\left(\Upomega\left(x_{0},\frac{r}{3}; \nu\right), \mu\right)} \mathrm{U}^{h_{-}}(\cdot, t + \Delta t), \hspace{2mm} \frac{\sigma}{s_{2}} < t < t_{0} - \Delta t.
  \end{eqnarray*}
  In particular,
  \begin{eqnarray*}
	\overline{\mathrm{U}}_{a_{2}}(\cdot, t_{0} - \Delta t) \prec_{\left(\Upomega\left( 0, \frac{r}{3}; \nu\right), \mu\right)} \mathrm{U}^{h_{-}}(\cdot, t_{0}).
  \end{eqnarray*}  
  By the choice of $\Delta t$ through (\ref{the first touching time in the second local comparison principle}), (\ref{the small vertical shift in the second local comparson principle}), (\ref{the choice of the integer in the second local comparison principle}) and (\ref{some notations regarding the interger shift in the second local comparison principle}), we have that
  \begin{equation*}
  0 < \Delta t < \frac{(C - 1)\gamma^{2}(t_{0})}{(n - 1)C^{2} + C(C - 1)\gamma(t_{0})M_0}.
  \end{equation*}
  Then the Proposition \ref{finite speed propagation of subsolution} implies that
  \begin{equation*}
  \overline{\mathrm{U}}_{a_{2}}(\cdot,t_{0}) \prec_{\left(\Upomega\left( 0,\frac{r}{3}; \nu\right), \mu\right)} \mathrm{U}^{\hat{h}_{-}}(\cdot,t_{0}), \hspace{2mm} \text{where} \hspace{2mm} \hat{h} := \left( \left( 1 - \frac{1}{C}\right)\gamma(t), 1 \right) \in \mathbb{F}.
  \end{equation*}
  \underline{Step 4.} Let us construct $\varphi(x) : \Upomega\left( x_{0}, \frac{r}{3}; \nu\right) \rightarrow (0, \infty)$ as follows:
  \begin{equation*}
  \varphi(x) := - \frac{9(1 + C)}{2Cr^{2}}|(x - x_{0})^{\top}|^{2} + \frac{1}{2}\left( 3 - \frac{1}{C}\right), 
  \end{equation*}
  where $(x - x_{0})^{\top} := (x - x_{0}) - ((x - x_{0})\cdot \nu)\nu$. Then $\varphi(x)$ satisfies that
  	\begin{eqnarray*}
  		\begin{cases}
  			\varphi|_{\left\lbrace x | (x - x_{0})^{\top} = 0\right\rbrace } = \frac{1}{2}\left( 3 - \frac{1}{C}\right) > 1, \\
  			\varphi|_{\left\lbrace x | |(x - x_{0})^{\top}| = \frac{r}{3}\right\rbrace } = 1 - \frac{1}{C},
  		\end{cases} &\text{and}& \begin{cases}
  		\lVert D\varphi\rVert_{\infty} \leq \frac{6}{r},\\
  		\lVert D^{2}\varphi\rVert_{\infty} \leq \frac{18}{r^{2}}.
  	\end{cases}
  \end{eqnarray*}
  Then, because $r \geq \max\left\lbrace 6, \frac{12(3n + M_0 + 27)}{L_0} \right\rbrace $, we have that
  \begin{eqnarray*}
  	\gamma^{\prime}(t) + \left( \frac{\left( n + 1\right) \lVert D^{2}\varphi \rVert_{\infty}}{\varphi(x)} + \frac{M_0|D\varphi(x)|}{\varphi(x)}
  	+ L_0\right)\gamma(t) + \frac{|D\varphi(x)|^{2}\gamma(t)}{\left( 1 - \gamma(t)|D\varphi(x)|\right)^{2}\varphi^{2}(x)} &\leq& 0.
  \end{eqnarray*}
  Then, by Proposition \ref{the evolution law of inf convolution}, $\mathrm{U}^{(\gamma,\varphi)}$ is a pseudo viscosity supersolution in $\Upomega\left( x_{0}, \frac{r}{3}; \nu\right) $. Based on the finite speed of propagation regarding the subsolution (c.f. Proposition \ref{finite speed propagation of subsolution}), there exists $\tau > 0$ such that
  \begin{eqnarray}\label{the strict ordering relation on the boundary in the second local comparison principle}
  \overline{\mathrm{U}}_{a_{2}}(\cdot,t) \prec_{(\partial\Upomega\left(x_{0}, \frac{r}{3}; \nu\right), \mu)} \mathrm{U}^{(\gamma,\varphi)_{-}}(\cdot,t), && T_{0} \leq t \leq t_{0} + \tau.
  \end{eqnarray}
  Note that the following strict ordering relation holds.
  \begin{equation}\label{an strict ordering relation at the initial moment in the second local comparison principle}
  \overline{\mathrm{U}}_{a_{2}}\left( \cdot, T_{0}\right)  \prec_{\left(\Upomega\left( x_{0}, \frac{r}{3}; \nu\right), \mu \right)} \mathrm{U}^{(\gamma,\varphi)_{-}}\left( \cdot, T_{0}\right).
  \end{equation}
  Let us introduce the function
  \begin{equation*}
  \mathrm{W}(x,t) := \begin{cases}
  2, & x \in \Upomega\left( x_{0}, \frac{r}{3}; \nu\right) \diagdown L_{\mu}^{-}\left( \mathrm{U}^{(\gamma, \varphi)}(\cdot, t); \Upomega\left( x_{0}, \frac{r}{3}; \nu\right) \right), \\
  0, & L_{\mu}^{-}\left( \mathrm{U}^{(\gamma, \varphi)}(\cdot, t); \Upomega\left( x_{0}, \frac{r}{3}; \nu\right) \right).
  \end{cases}
  \end{equation*}
  Since $\mathscr{F}(\cdot,\cdot,\cdot)$ is geometric, $\mathrm{W}(x,t)$ is a pseudo viscosity supersolution in $\Upomega(x_{0}, \frac{r}{3};\nu)$. \\
  \underline{Step 5.} Based on observations of (\ref{the strict ordering relation on the boundary in the second local comparison principle}), (\ref{an strict ordering relation at the initial moment in the second local comparison principle}), we have equivalently that
  \begin{eqnarray*}
  \mathrm{Z}(x,t) < \mathrm{W}(x,t), && (x,t) \in \left( \Upomega\left( x_{0}, \frac{r}{3}; \nu\right) \times \left\lbrace T_{0}\right\rbrace \right) \cup \left( \partial\Upomega\left( x_{0}, \frac{r}{3}; \nu\right) \times \left[ T_{0}, t_{0} + \tau\right]  \right). 
  \end{eqnarray*}
  Then we apply Proposition \ref{the comparison principle regarding pseudo viscosity solutions} to the above $\mathrm{Z}$ and $\mathrm{W}$ and conclude that
  \begin{eqnarray*}
  \mathrm{Z}(x,t) < \mathrm{W}(x,t), && T_{0} \leq t \leq t_{0} + \tau, \hspace{2mm} x \in \Upomega\left( x_{0}, \frac{r}{3}; \nu\right). 
  \end{eqnarray*}
  On the other hand, recall that $h = \left( \gamma(t), 1\right)$ and consider the facts
  \begin{eqnarray*}
  \overline{\mathrm{U}}_{a_{2}}(x_{0},t_{0}) = \mathrm{U}^{h_{-}}(x_{0},t_{0}) = \mu &\text{and}&  \varphi|_{\left\lbrace x | (x - x_{0})^{\top} = 0\right\rbrace } = \frac{1}{2}\left( 3 - \frac{1}{C}\right) > 1.
  \end{eqnarray*}
  It follows that $\mathrm{Z}(x_{0}, t_{0}) = 1 > 0 = \mathrm{W}(x_{0}, t_{0})$, which is a contradiction.
  
\end{proof}

\begin{proposition}\label{the second local comparison principle supersolution}
	Let $(\nu, T, R)$ be a comparison consistent triplet (c.f. Definition \ref{comparison consistent triplet}), fix $\mu \in \RR$ and $s_{i} := \underline{s}(\nu) - \delta_{i}$, $i = 1, 2$, where $\frac{m_0}{2} > \delta_{2} > \delta_{1} > 0$ are two fixed numbers. Then for any $r > \max \left\lbrace R, \frac{12(3n + M_0 + 27)}{L_0}\right\rbrace $, $T > 0$, there exist $A > 0$(independent of $r$ and $T$), $R_{i}$ and $\mathscr{R}_{i} > 0$, $i = 1, 2$, such that the following (i) and (ii) hold.
	\begin{enumerate}
		\item [(i)] $R_{1} \geq R_{2} + \left( \mathscr{R}_{1} + \mathscr{R}_{2}\right)T$;
		\item [(ii)] $\underline{\mathrm{U}}_{a_{1}}(\cdot,t) \prec_{(\Upomega\left( 0, r; \nu\right), \mu )} \underline{\mathrm{U}}_{a_{2}}(\cdot + \xi_{A}, t)$, $0 \leq t \leq T$;
	\end{enumerate}
	where
	\begin{eqnarray*}
	a_{1} := (\nu, R_{1}, -\mathscr{R}_{1}, q, s_{1}), \hspace{2mm} a_{2} := (\nu, R_{2}, \mathscr{R}_{2}, q, s_{2}) \in \mathbb{A} &\text{and}& \xi_{A} \in \argmin\limits\limits_{\xi \in \ZZ^{n},\hspace{1mm} \xi \cdot \nu > A} |\xi|.
	\end{eqnarray*}
\end{proposition}

\begin{proof}
	The idea of proof is similar to Proposition \ref{the second local comparison principle subsolution}, in view of Proposition \ref{finite speed propagation of subsolution}, \ref{the evolution law of inf convolution}, \ref{a lattice point that is close to a hyperplane}, \ref{a semi obstacle solution becomes a solution when detachment happens}, etc. The difference is that, this time, we use the Proposition \ref{the birkhoff property for supersolution in an expanding domain}, \ref{the birkhoff property for supersolutions in two static domains}, \ref{the birkhoff property for supersolution in a shrinking domain}, instead of Proposition \ref{the birkhoff property for subsolution in an expanding domain}, \ref{the birkhoff property for subsolutions in two static domains}, \ref{the birkhoff property for subsolution in a shrinking domain}. Note that $\frac{m_0}{2} \leq s_{2} < s_{1} \leq M_0$, the speeds have positive bounds, the arguments of Proposition \ref{the second local comparison principle subsolution} or Proposition \ref{LCP in the unit scale} still apply, merely with $m_0$ replaced by $\frac{m_0}{2}$. Thus, we omit the details here.
\end{proof}

\subsubsection{The detachment lemma}

\begin{proposition}[Detachment Property]\label{the detachment lemma subsolution}
	Let $(\nu, T, r)$ be a comparison triplet (c.f. Definition \ref{comparison consistent triplet}), fix $\sigma > 0$ and set $s := \overline{s}(\nu) + \sigma$, then for any $\mu \in \RR$, there exists $R > 0$ and $B := B(\nu,\sigma) > 0$, such that the following statement holds. 
	\begin{eqnarray*}
	\overline{\mathrm{U}}_{a}\left(\cdot - \left( \frac{1}{2}\sigma t - B\right)\nu,t\right) \prec_{\left(\Upomega(0,r;\nu), \mu \right) } \mathrm{O}_{e}(\cdot,t), && 0 \leq t \leq T,
	\end{eqnarray*}
    where
    \begin{eqnarray*}
    a := (\nu, R, 0, q, s) \in \mathbb{A} &\text{and}& e := (\nu, q, s) \in \mathbb{E}.
    \end{eqnarray*}
    Moreover, there exists a constant $C > 0$ independent of $\nu$ and $T$ such that $B = CT_{0}(\nu,\sigma) + C$, where $T_{0}(\nu, \sigma)$ is the time after which the $\mu$ level set of $\overline{\mathrm{U}}_{a}$ detaches from $\mathrm{O}_{e}$ in $\Upomega(0, 2r; \nu)$.
\end{proposition}

\begin{proof}
	Let us set
	\begin{eqnarray*}
	s_{1} := \overline{s}(\nu) + \frac{\sigma}{2} &\text{and}& s_{2} := \overline{s}(\nu) + \sigma.
	\end{eqnarray*}
    Then by Proposition \ref{the second local comparison principle subsolution}, there are numbers
    \begin{eqnarray*}
    A := A(\nu,\sigma) > 0, \hspace{2mm} R_{i} := R_{i}(\nu, \sigma, r, T), \hspace{2mm} \mathscr{R}_{i} := \mathscr{R}_{i}(\nu, \sigma, r, T) > 0, \hspace{2mm} i = 1, 2,
    \end{eqnarray*}
    such that the following (i) and (ii) hold:
	\begin{enumerate}
		\item [(i)] $R_{1} \geq R_{2} + \left( \mathscr{R}_{1} + \mathscr{R}_{2}\right) T$;
		\item [(ii)] $\overline{\mathrm{U}}_{a_{2}}(\cdot,t) \prec_{(\Upomega(0,r; \nu),\mu)} \overline{\mathrm{U}}_{a_{1}}(\cdot - \xi_{A}, t)$, $0 \leq t \leq T$,
	\end{enumerate}
	where
	\begin{eqnarray*}
	a_{1} := (\nu, R_{1}, -\mathscr{R}_{1}, q, s_{1}), \hspace{2mm} a_{2} := (\nu, R_{2}, \mathscr{R}_{2}, q, s_{2}) \in \mathbb{A} &\text{and}& \xi_{A} \in \argmin\limits\limits_{\xi \in \ZZ^{n},\hspace{1mm} \xi \cdot \nu > A} |\xi|.
	\end{eqnarray*} 
    Clearly, $A < \xi_{A} \cdot \nu < A + \sqrt{n}$. Let us also take
    \begin{eqnarray*}
    B(\nu, \sigma) := A(\nu, \sigma) + \sqrt{n} &\text{and}& R := R_{1}(\nu, \sigma, r, T),
    \end{eqnarray*}
    then the Proposition \ref{the birkhoff property for subsolutions in two static domains} indicates that
	\begin{eqnarray*}
	\overline{\mathrm{U}}_{a}(x,t) \leq \overline{\mathrm{U}}_{a_{2}}(x,t), && x \in \Upomega(0, r; \nu), \hspace{2mm} 0 \leq t \leq T.
	\end{eqnarray*}
    Then for any $x \in \Upomega(0, r; \nu)$ and $0 \leq t \leq T$, we get the relations
    \begin{eqnarray*}
    \overline{\mathrm{U}}_{a}\left(\cdot - \left( \frac{1}{2}\sigma t - B\right)\nu, t\right) &\leq& \overline{\mathrm{U}}_{a_{2}}\left(\cdot - \left( \frac{1}{2}\sigma t - B\right)\nu, t\right)\\
    & \prec_{(\Upomega(0,r;\nu),\mu)} & \overline{\mathrm{U}}_{a_{1}}\left(\cdot - \xi_{A} - \left( \frac{1}{2}\sigma t - B\right)\nu, t\right)\\
    &\leq& \mathrm{O}_{e_{1}}\left(\cdot - \xi_{A} - \left( \frac{1}{2}\sigma t - B\right)\nu, t\right)\\
    &\leq & \mathrm{O}_{e_{1}}\left(\cdot - \frac{1}{2}\sigma t\nu, t\right) = \mathrm{O}_{e}(\cdot, t)
    \end{eqnarray*}
   where $e_{1} := (\nu, q, s_{1})$. Hence the desired ordering relation is established.
\end{proof}

\begin{proposition}[Detachment Property]\label{the detachment lemma supersolution}
	Let $(\nu, T, r)$ be a comparison triplet, fix $\sigma > 0$ and set $s := \underline{s}(\nu) - \sigma > 0$, then for any $\mu \in \RR$, there exists $R > 0$ and $B := B(\nu,\sigma) > 0$, and the following statement holds.
	\begin{eqnarray*}
		\mathrm{O}_{e}(\cdot,t) \prec_{(\Upomega(0,r;\nu), \mu)} \underline{\mathrm{U}}_{a}\left(\cdot + \left( \frac{1}{2}\sigma t - B\right)\nu,t\right), && 0 \leq t \leq T,
	\end{eqnarray*}
	where
	\begin{eqnarray*}
		a := (\nu, R, 0, q, s) \in \mathbb{A} &\text{and}& e := (\nu, q, s) \in \mathbb{E}.
	\end{eqnarray*}
	Moreover, there exists a constant $C > 0$ independent of $\nu$ and $T$ such that $B = CT_{0}(\nu,\sigma) + C$, where $T_{0}(\nu, \sigma)$ is the time after which the $\mu$ level set of $\underline{\mathrm{U}}_{a}$ detaches from $\mathrm{O}_{e}$ in $\Upomega(0, 2r; \nu)$. 
\end{proposition}

\begin{proof}
	It is similar to that of Proposition \ref{the detachment lemma subsolution}. The difference is that instead of using Proposition \ref{the birkhoff property for subsolutions in two static domains}, \ref{the second local comparison principle subsolution}, we use Proposition \ref{the birkhoff property for supersolutions in two static domains}, \ref{the second local comparison principle supersolution}. Hence, we do not repeat the details any more.
\end{proof}

Next, let us show that if a speed is strictly larger (resp. strictly smaller) than the head (resp. tail) speed in an irrational direction $\nu$, then all superlevel (resp. sublevel) sets of the obstacle subsolution (resp. supersolution) detaches linearly and uniformly from its obstacle.

\begin{proposition}\label{uniform detachment for subsolution with a sub strict detached speed}
	Let $\nu \in \mathbb{S}^{n-1}\diagdown \RR\ZZ^{n}$, assume that $s > \overline{s}(\nu)$. Then there exists $\delta > 0$ and $B := B(\nu,\delta) > 0$, such that the following statement holds. For any $\mu \in \RR$, $r > 0$ and $T > 0$, there exists $R > 0$, independent of $\mu$, such that 
	\begin{eqnarray*}
		\overline{\mathrm{U}}_{a}\left( \cdot - \left( \delta t - B\right)\nu, t\right) \prec_{\left( \Upomega(0,r;\nu), \mu\right) } \mathrm{O}_{e}(\cdot, t), && 0 \leq t \leq T,
	\end{eqnarray*}
	where
	\begin{eqnarray*}
		a := (\nu, R, 0, q, s) \in \mathbb{A} &\text{and}& e := (\nu, q, s) \in \mathbb{E}.
	\end{eqnarray*}
\end{proposition}

\begin{proof}
	Based on Proposition \ref{the detachment lemma subsolution}, there exists $\delta > 0$ and $B_{0} := B_{0}(\nu,\delta) > 0$, such that for any $T > 0$, there exists $R_{0} > r + \sqrt{n}$, with the following statement holds.
	\begin{eqnarray*}
		\overline{\mathrm{U}}_{a_{0}}\left( \cdot - \left( \delta t - B_{0}\right)\nu, t \right) \prec_{\left( \Upomega(0, r+\sqrt{n};\nu), 0\right) } \mathrm{O}_{e}(\cdot, t), && 0 \leq t \leq T,
	\end{eqnarray*}
	where
	\begin{eqnarray*}
		a_{0} := (\nu, R_{0}, 0, q, s) \in \mathbb{A} &\text{and}& e:= (\nu, q, s) \in \mathbb{E}.
	\end{eqnarray*}
	Let us consider the general $\mu$ level set, there exists $\Delta z \in \ZZ^{n}$, such that
	\begin{eqnarray*}
		-\frac{\mu}{|q|} \leq \Delta z \cdot \nu \leq - \frac{\mu}{|q|} + \sqrt{n} &\text{and}& \left| \Delta z - \left( \Delta z \cdot \nu\right)\nu \right| \leq \sqrt{n}. 
	\end{eqnarray*}
	Then we denote
	\begin{equation*}
	\mathrm{U}(x,t) := \overline{\mathrm{U}}_{a_{0}}\left( x - \Delta z, t \right) + \mu. 
	\end{equation*}
	Then $\mathrm{U}(x,t)$ is the largest subsolution that is in $\Upomega(0, R_{0}; \nu) + \Delta z$ and bounded from above by $\mathrm{O}(x,t) := \mathrm{O}_{e}(x - \Delta z, t) + \mu$, which is no less than $\mathrm{O}_{e}(x,t)$. Moreover,
	\begin{eqnarray*}
		\mathrm{U}(\cdot - \left( \delta t - B_{0}\right)\nu, t) \prec_{(\Upomega(0, r; \nu), \mu)} \mathrm{O}(x,t), && 0 \leq t \leq T.
	\end{eqnarray*}    
	Now, let us take $R := R_{0} + \sqrt{n}$, then the associated obstacle subsolution $\overline{\mathrm{U}}_{a}(x,t)$, restricted to $\Upomega(0, R_{0}; \nu) + \Delta z$ (which includes $\Upomega(0,r;\nu)$), is no larger than $\mathrm{U}(x,t)$. Therefore,
	\begin{eqnarray*}
		\overline{\mathrm{U}}_{a}\left( \cdot - \left( \delta t - B\right)\nu, t\right) \prec_{\left( \Upomega(0,r;\nu), \mu\right) } \mathrm{O}_{e}(\cdot, t), && 0 \leq t \leq T,
	\end{eqnarray*}
	where $B = B_{0} + \sqrt{n}$. 
\end{proof}

\begin{proposition}\label{uniform detachment for supersolution with a super strict detached speed}
	Let $\nu \in \mathbb{S}^{n-1}\diagdown \RR\ZZ^{n}$, assume that $0 < s < \underline{s}(\nu)$. Then there exists $\delta > 0$ and $B := B(\nu,\delta) > 0$, such that the following statement holds. For any $\mu \in \RR$, $r > 0$ and $T > 0$, there exists $R > 0$, independent of $\mu$, such that 
	\begin{eqnarray*}
		\mathrm{O}_{e}(\cdot, t) \prec_{\left( \Upomega(0,r;\nu), \mu\right) } \overline{\mathrm{U}}_{a}\left( \cdot + \left( \delta t - B\right)\nu, t\right), && 0 \leq t \leq T,
	\end{eqnarray*}
	where
	\begin{eqnarray*}
		a := (\nu, R, 0, q, s) \in \mathbb{A} &\text{and}& e := (\nu, q, s) \in \mathbb{E}.
	\end{eqnarray*}
\end{proposition}

\begin{proof}
	It follows the same idea as Proposition \ref{uniform detachment for subsolution with a sub strict detached speed}, we omit the details here.
\end{proof}

\subsubsection{The ordering relation}

Our goal of introducing the head/tail speed is to model the highest/lowest speed of the level set of a real solution. The next Proposition shows that the head speed is indeed no less than the tail speed, at least in each irrational direction.

\begin{proposition}[Ordering relation]\label{the ordering relation in irrational directions}
	Fix any $\nu \in \mathbb{S}^{n-1}\diagdown\RR\ZZ^{n}$, then $\underline{s}(\nu) \leq \overline{s}(\nu)$.
\end{proposition}

\begin{proof}
	Let us assume on the contrary that $\theta := \frac{\underline{s}(\nu) - \overline{s}(\nu)}{3} > 0$ and then we derive a contradiction through the following steps.\\
	\underline{Step 1.} For any $\mu \in \RR$, we set
	\begin{eqnarray*}
	s_{1} := \overline{s}(\nu) + \theta &\text{and}& s_{2} := \underline{s}(\nu) - \theta.
	\end{eqnarray*}
	Then there exist $\mathcal{R}_{0} = \mathcal{R}_{0}(\nu) > r_{0} = r_{0}(\nu) > \frac{\sqrt{n}}{2}$ and $T_{0} = T_{0}(\nu) > 0$, such that
	\begin{eqnarray*}
		\begin{cases}
			\overline{\mathrm{U}}_{a_{1}}(\cdot,t) \prec_{(\Upomega(0,r_{0}; \nu), \mu)} \mathrm{O}_{e_{1}}(\cdot,t)\\
			\mathrm{O}_{e_{2}}(\cdot,t) \prec_{(\Upomega(0,r_{0}; \nu),\mu)} \underline{\mathrm{U}}_{a_{2}}(\cdot,t)
		\end{cases}, && t > T_{0},
	\end{eqnarray*}
    where
    \begin{eqnarray*}
    a_{i} := (\nu, \mathcal{R}_{0}, 0, q, s_{i}) \in \mathbb{A} &\text{and}& e_{i} := (\nu, q, s_{i}) \in \mathbb{E}, \hspace{2mm} i = 1, 2.
    \end{eqnarray*}
    Let us denote
    \begin{eqnarray*}
    A := \left(M_0 - m_0\right) T_{0} + 2\sqrt{n}, \hspace{4mm} \xi_{A} \in \argmin\limits\limits_{\substack{\xi \in \ZZ^{n}, \hspace{1mm} \xi\cdot\nu > A}}|\xi| &\text{and}& T := \frac{\left( M_0 - m_0\right)T_{0} + 3\sqrt{n}}{\theta}.
    \end{eqnarray*}
   For a fixed $C > 1$, we set a small positive number as follows, where $\gamma(t) := \frac{1}{2}e^{-2L_0t}$:
   \begin{equation}\label{the choice of vertical gap in the third lcp}
   \delta := \left( \frac{M_0 m_0}{M_0 - m_0}\right)\cdot \frac{ (C - 1)\gamma^{2}(T)}{(n - 1)C^{2} + C(C - 1)\gamma(T)M_0} > 0.
   \end{equation}
   Denote by $R_{0}(\nu,\delta)$ the number defined in Proposition \ref{a lattice point that is close to a hyperplane}, associated to the above $(\nu, \delta)$. And then denote
   \begin{equation*}
   R := \max\left\lbrace \mathcal{R}_{0}(\nu), R_{0}(\nu,\delta), \frac{12\left(3n + M_0 + 27 \right) }{L_0}\right\rbrace. 
   \end{equation*}
   The Proposition \ref{a lattice point that is close to a hyperplane} indicates that for any $\hat{x}$ with $|\hat{x} - (\hat{x}\cdot\nu)\nu| = R$, there exists $\hat{z}$, such that
   \begin{equation*}
   (i) \hspace{2mm} 0 < \left( \hat{z} - \hat{x}\right) \cdot \nu < \delta; \hspace{4mm} (ii)\hspace{2mm} |\hat{z} - 2\hat{x}| < \frac{R}{3}; \hspace{4mm} (iii) \hspace{2mm} \Delta \hat{z} := \hat{z} - \hat{x} \in \ZZ^{n}.
   \end{equation*}
   Since there are finite such interger vectors $\Delta \hat{z}$, we can set
   \begin{eqnarray*}
   \sigma := \min\left\lbrace \Delta \hat{z} \cdot \nu \big| |\hat{x} - (\hat{x}\cdot\nu)\nu| = R\right\rbrace  > 0 &\text{and}& \mathscr{R} := 2\left( \frac{1}{\sigma} + 1\right)M_0 R.
   \end{eqnarray*}
   Consider
   \begin{eqnarray*}
   b_{i} := (\nu, R_{i}, -\mathscr{R}, q, s_{i}), &\text{where}& R_{i} > R + \mathscr{R}T + \sqrt{n}, \hspace{2mm} i = 1, 2.
   \end{eqnarray*}
   And denote $\mathrm{U}(x,t) := \left( \overline{\mathrm{U}}_{b_{1}}\right)_{*}(x - \xi_{A}, t)$, then by Lemma \ref{the expansion of detachment in a static domain for the obstacle subsolution}, we have that
   \begin{eqnarray*}
   \left( \underline{\mathrm{U}}_{b_{2}}\right)^{*}(\cdot,t) \prec_{(\Upomega(0,R;\nu), \mu)} \mathrm{U}(\cdot,t), && 0 \leq t \leq T_{0} + \frac{\sqrt{n}}{\theta}.
   \end{eqnarray*}
   Then for any $t > T_{0} + \frac{\sqrt{n}}{\theta}$, we have that
   \begin{eqnarray*}
   \mathrm{O}_{e_{2}}(\cdot,t) \prec_{\left( \Upomega\left(0, \frac{4R}{3}; \nu\right), \mu\right) } \underline{\mathrm{U}}_{b_{2}}(\cdot,t) &\text{and}& \overline{\mathrm{U}}_{b_{1}}(\cdot,t) \prec_{\left( \Upomega\left(0, \frac{4R}{3}; \nu\right), \mu\right) } \mathrm{O}_{e_{1}}(\cdot,t).
   \end{eqnarray*}
   Because $\theta\cdot T > |\xi_{A}\cdot\nu|$, we conclude that
   \begin{equation*}
   T_{0} + \frac{\sqrt{n}}{\theta} < \sup\left\lbrace t > 0 \big| \left( \underline{\mathrm{U}}_{b_{2}}\right)^{*}(\cdot,t) \prec_{(\Upomega(0,R;\nu), \mu)} \mathrm{U}(\cdot,t) \right\rbrace < T. 
   \end{equation*}
   Let us set
   \begin{eqnarray*}
   \mathrm{Z}(x,t) &:=& \begin{cases}
   	\mu, & x \in L_{\mu}^{+}\left( \left( \underline{\mathrm{U}}_{b_{2}}\right)^{*}(\cdot, t); \Upomega\left( 0, \frac{4R}{3}; \nu\right) \right), \\
   	\mu - 1, & x \in \Upomega\left( 0, \frac{4R}{3}; \nu\right) \diagdown  L_{\mu}^{+}\left( \left( \underline{\mathrm{U}}_{b_{2}}\right)^{*}(\cdot, t); \Upomega\left( 0, \frac{4R}{3}; \nu\right) \right),
   \end{cases}\\
   \mathrm{Y}(x,t) &:=& \begin{cases}
   	\mu, & x \in \Upomega\left( 0, \frac{4R}{3}; \nu\right) \diagdown L_{\mu}^{-}\left( \left( \overline{\mathrm{U}}_{b_{1}}\right)_{*}(\cdot - \xi_{A},t); \Upomega\left(0, \frac{4R}{3};\nu \right)  \right), \\
   	\mu - 1, & x \in L_{\mu}^{-}\left( \left( \overline{\mathrm{U}}_{b_{1}}\right)_{*}(\cdot - \xi_{A}, t); \Upomega\left(0, \frac{4R}{3};\nu \right)  \right). 
   \end{cases}
   \end{eqnarray*}
   Because $\underline{\mathrm{U}}_{b_{2}}(x,t)$ detaches from the obstacle $\mathrm{O}_{e_{2}}(x,t)$ in the $\mu$ level set, by lower semicontinuity of $\underline{\mathrm{U}}_{b_{2}}(x,t)$, Proposition \ref{a semi obstacle solution becomes a solution when detachment happens}, the operator $\mathscr{F}(\cdot,\cdot,\cdot)$ is geometric and the fact that $\xi_{A} \in \ZZ^{n}$, we can conclude that $\mathrm{Z}(x,t)$ is a viscosity subsolution for $x \in \Upomega\left( 0, \frac{4R}{3}; \nu\right)$ and $t > T_{0} + \frac{\sqrt{n}}{\theta}$. Similarly, we have that $\mathrm{Y}(x,t)$ is a viscosity supersolution in the same space time domain.\\  
   \underline{Step 2.} Let $h := (\gamma(t), 1) \in \mathbb{F}$ and $\mathrm{U}^{h_{-}}(x,t)$ be the $h$ inf-convolution of $\mathrm{U}(x,t)$ by Definition \ref{the definition of the inf convolution}. Then the choice of $A$ and $\xi_{A}$ indicates that
   \begin{equation}\label{the first moment of touchment in the third local comparison principle}
   T_{0} + \frac{\sqrt{n}}{\theta} < t_{0} := \sup\left\lbrace t > 0 \big| \left( \underline{\mathrm{U}}_{b_{2}}\right)^{*}(\cdot,t) \prec_{(\Upomega(0,R; \nu), \mu)} \mathrm{U}^{h_{-}}(\cdot,t)\right\rbrace < T.
   \end{equation}
   Assume $x_{0}$ is the first crossing point between (the $\mu$ superlevel set of) $\left( \underline{\mathrm{U}}_{b_{2}}\right)^{*}$ and (the $\mu$ sublevel set of) $\mathrm{U}^{h_{-}}$ over $\Upomega(0,R;\nu)$.  i.e.,
   \begin{equation*}
   \left( \underline{\mathrm{U}}_{b_{2}}\right)^{*}(x_{0},t_{0}) = \mathrm{Z}(x_{0},t_{0}) = \mathrm{Y}^{h_{-}}(x_{0},t_{0}) = \mathrm{U}^{h_{-}}(x_{0},t_{0}) = \mu.
   \end{equation*}
   Then by applying Proposition \ref{the usual comparison principle} to $\underline{\mathrm{U}}_{b_{2}}(\cdot,t)$ and $\mathrm{U}(\cdot,t)$ in $\Upomega(0,R;\nu)$, the maximum of  $\left( \underline{\mathrm{U}}_{b_{2}}\right)^{*}(x,t) - \mathrm{U}(x,t)$ over $\Upomega(0, R; \nu)$ is obtained on $\partial\Upomega(0,R;\nu)$. Since $h$ is $x$-independent, $x_{0} \in \partial\Upomega(0,R;\nu)$. As a result of Proposition \ref{a lattice point that is close to a hyperplane}, there exists $z_{0} \in \RR^{n}$, such that
  \begin{equation*}
  (i) \hspace{2mm} \frac{\delta}{3} < \left(z_{0} - x_{0}\right) \cdot \nu < \delta; \hspace{4mm} (ii)\hspace{2mm} |z_{0} - 2x_{0}| < \frac{R}{3}; \hspace{4mm} (iii) \hspace{2mm} \Delta z := z_{0} - x_{0} \in \ZZ^{n}.
  \end{equation*}   
  \underline{Step 3.} Let us denote 
  \begin{eqnarray}\label{the exact vertical gap and the time gap in the third lcp}
  \sigma_{0} := |\Delta z \cdot \nu| &\text{and}& \Delta t := \frac{\sigma_{0}}{m_0} - \frac{\sigma_{0}}{M_0}.
  \end{eqnarray}
  It is clear that $\sigma_{0} \geq \sigma$. Then by Proposition \ref{the birkhoff property for supersolution in a shrinking domain}, we have the ordering relation
  \begin{equation}\label{the birkhoff property applied to the subsolution in the third lcp}
  \underline{\mathrm{U}}_{b_{2}}(x,t) \leq \underline{\mathrm{U}}_{b_{2}}\left(x - \Delta z, t - \frac{\sigma_{0}}{M_0}\right), \hspace{2mm} x \in \Upomega\left( x_{0}, \frac{R}{3}; \nu\right), \hspace{2mm} \frac{\sigma_{0}}{M_0} \leq t \leq t_{0} - \Delta t.
  \end{equation}
  Because of the inclusion $\Upomega\left( x_{0}, \frac{R}{3}; \nu\right) + \Delta z \subseteq \Upomega(0, R; \nu)$ and (\ref{the first moment of touchment in the third local comparison principle}), we have that
  \begin{equation}\label{the ordering relation at an earlier moment in the third lcp}
  \left( \underline{\mathrm{U}}_{b_{2}}\right)^{*}\left(\cdot, t - \frac{\sigma_{0}}{M_0}\right) \prec_{(\Upomega(0,R;\nu), \mu)} \mathrm{U}^{h_{-}}\left(\cdot, t - \frac{\sigma_{0}}{M_0}\right), \hspace{2mm} \frac{\sigma_{0}}{M_0} \leq t \leq t_{0} - \Delta t. 
  \end{equation}
  The Proposition \ref{the birkhoff property for subsolution in a shrinking domain} applied to $\overline{\mathrm{U}}_{b_{1}}$ with the above shift $\Delta z$ implies that
  \begin{eqnarray*}
  \overline{\mathrm{U}}_{b_{1}}\left( y, t - \frac{\sigma_{0}}{M_0}\right) &\leq& \overline{\mathrm{U}}_{b_{1}}\left( y + \Delta z, t + \frac{\sigma_{0}}{m_0} - \frac{\sigma_{0}}{M_0}\right)\\
  &=& \overline{\mathrm{U}}_{b_{1}}\left( y + \Delta z, t + \Delta t\right), \hspace{1cm} y \in \Upomega(0,R;\nu), \hspace{2mm} \frac{\sigma_{0}}{M_0} \leq t \leq t_{0} - \Delta t.
  \end{eqnarray*}
  Then apply the $\xi_{A}$ shift and the $h$ inf-convolution to both sides, we conclude that
  \begin{equation}\label{the birkhoff property applied to the supersolution in the third lcp}
  \mathrm{U}^{h_{-}}\left( y, t - \frac{\sigma_{0}}{M_0}\right) \leq \mathrm{U}^{h_{-}}(y + \Delta z, t + \Delta t), \hspace{2mm} y \in \Upomega(0, R; \nu), \hspace{2mm} \frac{\sigma_{0}}{M_0} \leq t \leq t_{0} - \Delta t.
  \end{equation}
  A combination of (\ref{the birkhoff property applied to the subsolution in the third lcp}), (\ref{the ordering relation at an earlier moment in the third lcp}) and (\ref{the birkhoff property applied to the supersolution in the third lcp}) gives the following relation:
  \begin{equation*}
  \left( \underline{\mathrm{U}}_{b_{2}}\right)^{*}(\cdot,t) \prec_{\left( \Upomega\left( x_{0}, \frac{R}{3}; \nu\right), \mu\right) } \mathrm{U}^{h_{-}}(\cdot, t + \Delta t), \hspace{2mm} \frac{\sigma_{0}}{M_0} \leq t \leq t_{0} - \Delta t. 
  \end{equation*}
  In particular, we have that 
    \begin{equation*}
    \left( \underline{\mathrm{U}}_{b_{2}}\right)^{*}(\cdot, t_{0} - \Delta t) \prec_{\left( \Upomega\left( x_{0}, \frac{R}{3}; \nu\right), \mu\right) } \mathrm{U}^{h_{-}}(\cdot, t_{0}).
    \end{equation*}
    By the choice of $\Delta t$ through (\ref{the choice of vertical gap in the third lcp}), (\ref{the first moment of touchment in the third local comparison principle}) and (\ref{the exact vertical gap and the time gap in the third lcp}), then
    \begin{equation*}
    0 < \Delta t \leq \frac{(C - 1)\gamma^{2}(t_{0})}{(n - 1)C^{2} + C(C - 1)\gamma(t_{0})M_0}.
    \end{equation*}
    Then the Proposition \ref{finite speed propagation of subsolution} implies that
    \begin{eqnarray*}
    \left( \underline{\mathrm{U}}_{b_{2}}\right)^{*}(\cdot,t_{0}) \prec_{\left( \Upomega\left( x_{0}, \frac{R}{3}; \nu\right), \mu\right) } \mathrm{U}^{\hat{h}_{-}}(\cdot,t_{0}), &\text{where}& \hat{h} := \left( \left( 1 - \frac{1}{C}\right)\gamma(t), 1 \right). 
    \end{eqnarray*}
    	\underline{Step 4.} Let us still construct $\varphi(x): \Upomega\left( x_{0}, \frac{R}{3}; \nu\right)  \rightarrow \RR$ as follows:
    	\begin{eqnarray*}
    		\varphi(x) := - \frac{9(1 + C)}{2CR^{2}}\left|(x - x_{0})^{\top}\right|^{2} + \frac{1}{2}\left( 3 - \frac{1}{C}\right), 
    	\end{eqnarray*}
    	where $(x - x_{0})^{\top} := (x - x_{0}) - ((x - x_{0})\cdot\nu)\nu$. Then $\varphi(x)$ satisfies that
    	\begin{eqnarray*}
    	\begin{cases}
    		\varphi|_{\left\lbrace x | (x - x_{0})^{\top} = 0\right\rbrace } = \frac{1}{2}\left( 3 - \frac{1}{C}\right) > 1, \\
    		\varphi|_{\left\lbrace x | |(x - x_{0})^{\top}| = \frac{R}{3}\right\rbrace } = 1 - \frac{1}{C},
    		\end{cases} &\text{and}& \begin{cases}
    		\lVert D\varphi\rVert_{\infty} \leq \frac{6}{R},\\
    		\lVert D^{2}\varphi\rVert_{\infty} \leq \frac{18}{R^{2}}.
    	\end{cases}
    \end{eqnarray*}
    Then, because $R \geq \max\left\lbrace 6, \frac{12(3n + M_0 + 27)}{L_0} \right\rbrace $, we have that
    \begin{eqnarray*}
    	\gamma^{\prime}(t) + \left( \frac{\left( n + 1\right) \lVert D^{2}\varphi \rVert_{\infty}}{\varphi(x)} + \frac{M_0|D\varphi(x)|}{\varphi(x)}
    	+ L_0\right)\gamma(t) + \frac{|D\varphi(x)|^{2}\gamma(t)}{\left( 1 - \gamma(t)|D\varphi(x)|\right)^{2}\varphi^{2}(x)} \leq 0.
    \end{eqnarray*}
    Then, based on the finite speed of propagation regarding the subsolution (c.f. Proposition \ref{finite speed propagation of subsolution}), we have for some $\tau > 0$ that
    \begin{equation}\label{the strict ordering relation on the boundary in the third local comparison principle}
    \left( \underline{\mathrm{U}}_{b_{2}}\right)^{*}(\cdot,t) \prec_{\left( \partial\Upomega\left( x_{0}, \frac{R}{3}; \nu\right), \mu\right) } \mathrm{U}^{(\gamma,\varphi)_{-}}(\cdot,t), \hspace{2mm} T_{0} + \frac{\sqrt{n}}{\theta} \leq t \leq t_{0} + \tau.
    \end{equation}
    Note that the following strict ordering relation holds.
    \begin{equation}\label{an strict ordering relation at the initial moment in the third local comparison principle}
    \left( \underline{\mathrm{U}}_{b_{2}}\right)^{*}\left(\cdot, T_{0} + \frac{\sqrt{n}}{\theta}\right)  \prec_{\left( \Upomega\left(x_{0}, \frac{R}{3}; \nu\right), \mu\right) } \mathrm{U}^{(\gamma,\varphi)_{-}}\left(\cdot, T_{0} + \frac{\sqrt{n}}{\theta}\right).
    \end{equation}
    Let us introduce the function
    \begin{eqnarray*}
    \mathrm{W}(x,t) := \begin{cases}
    	\mu, & x \in \Upomega\left( x_{0}, \frac{R}{3}; \nu\right)\diagdown L_{\mu}^{-}\left( \mathrm{U}^{(\gamma,\varphi)}(\cdot,t); \Upomega\left(x_{0}, \frac{R}{3}; \nu \right) \right), \\
    	\mu - 1, & x \in L_{\mu}^{-}\left( \mathrm{U}^{(\gamma,\varphi)}(\cdot,t); \Upomega\left(x_{0}, \frac{R}{3}; \nu \right) \right).
    \end{cases}
    \end{eqnarray*}
    \underline{Step 5.} Based on the observations of (\ref{the strict ordering relation on the boundary in the third local comparison principle}) and (\ref{an strict ordering relation at the initial moment in the third local comparison principle}), let us apply the proof of Proposition \ref{the comparison principle regarding pseudo viscosity solutions} to the above $\mathrm{Z}(x,t)$ and $\mathrm{W}(x,t)$, and we can conclude that
    \begin{eqnarray*}
    \mathrm{Z}(x,t) \leq \mathrm{W}(x,t), && T_{0} + \frac{\sqrt{n}}{\theta} \leq t \leq t_{0} + \tau, \hspace{2mm} x \in \Upomega\left( x_{0}, \frac{R}{3}; \nu\right). 
    \end{eqnarray*}
    On the other hand, recall that $h = (\gamma(t), 1)$ and consider the facts
    \begin{eqnarray*}
    \left( \underline{\mathrm{U}}_{b_{2}}\right)^{*}(x_{0},t_{0}) = \mathrm{U}^{h_{-}}(x_{0},t_{0}) = \mu &\text{and}& \varphi|_{\left\lbrace x | (x - x_{0})^{\top} = 0\right\rbrace } = \frac{1}{2}\left( 3 - \frac{1}{C}\right) > 1. 
    \end{eqnarray*}
    It follows that $\mathrm{Z}(x_{0},t_{0}) = \mu > \mu - 1 = \mathrm{W}(x_{0},t_{0})$, which is a contradiction.
\end{proof}

\subsection{General directions}
\subsubsection{The extension of the head/tail speed}
Up to now, we have defined the head/tail speed in all irrational directions $\nu \in \mathbb{S}^{n-1}\diagdown \RR\ZZ^{n}$, in which the local comparison principle (Proposition \ref{LCP in the unit scale}) and the detachment property (Proposition \ref{the detachment lemma subsolution}) hold. In order to study the homogenization, it is necessary to extend the concept to all directions $\nu \in \mathbb{S}^{n-1}$. In particular, let us define the head/tail speed in rational directions, it turns out that the detachment property is the essential ingredient of the concept.

\begin{definition}\label{sub or super strict detached speeds}
	Fix $\vartheta \in \mathbb{S}^{n-1}$, a number $s$ is called sub-strict-detached (resp. super-strict-detached) with respect to $\vartheta$ if the following holds: There exists $\delta > 0$, $B := B(\vartheta, \delta) > 0$, such that for any $\mu \in \RR$, $r > 0$, $q \in \RR^{n}\diagdown \left\lbrace 0\right\rbrace$ and $T > 0$, there exists $R := R(\vartheta, \mu, \delta, r, T) > 0$, such that we have the following relation, where $a := (\vartheta, R, 0, q, s) \in \mathbb{A}$.
	\begin{eqnarray*}
	\overline{\mathrm{U}}_{a}(\cdot - (\delta t - B)\vartheta, t) \prec_{\left( \Upomega(0, r; \vartheta), \mu\right)} \mathrm{O}_{e}(\cdot, t), && 0 \leq t \leq T
	\end{eqnarray*}
	\begin{eqnarray*}
	(resp. \hspace{2mm} \mathrm{O}_{e}(\cdot, t) \prec_{\left( \Upomega(0,r;\vartheta), \mu\right)} \underline{\mathrm{U}}_{a}(\cdot + (\delta t - B)\vartheta, t), && 0 \leq t \leq T ).
	\end{eqnarray*}
\end{definition}

\begin{proposition}
	Fix $\nu \in \mathbb{S}^{n-1}\diagdown \RR\ZZ^{n}$, then any $s \in \left( \overline{s}, \infty\right)$ is sub-strict-detached; any $s \in (0, \underline{s})$ is super-strict-detached. Moreover, we have the following expression of head/tail speed.
	\begin{eqnarray*}
	\overline{s}(\nu) &:=& \inf \left\lbrace s > 0 | s \text{ is sub-strict-detached with respect to } \nu\right\rbrace, \\
	\underline{s}(\nu) &:=& \sup \left\lbrace s > 0 | s \text{ is super-strict-detached with respect to } \nu\right\rbrace.
	\end{eqnarray*}
\end{proposition}

\begin{proof}
	It is an immediate corollary of Proposition \ref{the detachment lemma subsolution}, \ref{the detachment lemma supersolution}.
\end{proof}

\begin{definition}
	Let $\vartheta \in \mathbb{S}^{n-1}\cap \RR\ZZ^{n}$, the head speed (resp. tail speed) in $\vartheta$, denoted by $\overline{s}(\vartheta)$ (resp. $\underline{s}(\vartheta)$), is defined as follows.
	\begin{eqnarray*}
	\overline{s}(\vartheta) &:=& \inf \left\lbrace s > 0 | s \text{ is sub-strict-detached with respect to } \vartheta\right\rbrace, \\
	(\text{resp.} \hspace{2mm} \underline{s}(\vartheta) &:=& \sup \left\lbrace s > 0 | s \text{ is sup-strict-detached with respect to } \vartheta\right\rbrace ).
	\end{eqnarray*}
\end{definition}

\subsubsection{An equivalent description}

\begin{definition}\label{the global head speed in an rational direction}
	Let $\vartheta \in \mathbb{S}^{n-1}$, the global head (resp. tail) speed in $\vartheta$ direction, denoted by $\overline{s}^{\infty}(\vartheta)$ (resp. $\underline{s}_{\infty}(\vartheta)$), is defined as the smallest (resp. largest) number, such that: for any $\delta > 0$, $\overline{\mathrm{U}}_{a^{\infty}}(x,t)$ (resp. $\underline{\mathrm{U}}_{a_{\infty}}(x,t))$ detaches (c.f. Definition \ref{the definition of detachment}) from $\mathrm{O}_{e^{\infty}}(x,t)$ (resp.  $\mathrm{O}_{e_{\infty}}(x,t)$) in $\Upomega(0, r; \vartheta)$ for some $r > \frac{\sqrt{n}}{2}$, where
	\begin{eqnarray*}
		a^{\infty} := (\vartheta, \infty, 0, q, \overline{s}^{\infty}(\nu) + \delta) \in \mathbb{A}, && e^{\infty} := (\nu, q, \overline{s}^{\infty}(\nu) + \delta) \in \mathbb{E}\\
		(\text{resp.} \hspace{2mm} a_{\infty} := (\vartheta, \infty, 0, q, \underline{s}_{\infty}(\nu) - \delta) \in \mathbb{A}, && e_{\infty} := (\nu, q, \underline{s}_{\infty}(\nu) - \delta) \in \mathbb{E}).
	\end{eqnarray*}
\end{definition}

\begin{proposition}[Equivalence]\label{the equivalence between domain obstacle semi solutions and global semi solutions in defining head and tail speeds}
	Fix any $\vartheta \in \mathbb{S}^{n-1}$, then 
	\begin{eqnarray*}
	\overline{s}(\vartheta) = \overline{s}^{\infty}(\vartheta) &\text{and}& \underline{s}(\vartheta) = \underline{s}_{\infty}(\vartheta).
	\end{eqnarray*}
\end{proposition}

\begin{proof}
	Let us only prove the equality regarding the head speed, the other one follows by a similar pattern. \\
    \underline{Step 1.} Under the same obstacle speed, the global obstacle subsolution is clearly less or equal to the obstacle subsolution associated to domain with finte radius. Therefore, we have that $\overline{s}(\vartheta) \geq \overline{s}^{\infty}(\vartheta)$.\\
    \underline{Step 2.} Let us assume on the contrary that $\overline{s}(\vartheta) = \overline{s}^{\infty}(\vartheta) + \gamma$ for some $\gamma > 0$. Set
    \begin{eqnarray*}
    \delta := \frac{\gamma}{2}, \hspace{4mm} a^{\infty} := (\vartheta, \infty, 0, q, \underbrace{\overline{s}^{\infty}(\vartheta) + \delta}_{s}) \in \mathbb{A} &\text{and}& a_{\ell} := (\vartheta, \ell, 0, q, \underbrace{\overline{s}(\vartheta) - \delta}_{s}) \in \mathbb{A}.
    \end{eqnarray*}
    Define
    \begin{equation*}
    \mathrm{U}_{\infty}(x,t) := \limsup_{\eta \rightarrow 0}\left\lbrace \overline{\mathrm{U}}_{a_{\ell}}(y,s) \Big| |y - x| + |s - t| + \frac{1}{\ell} < \eta\right\rbrace. 
    \end{equation*}
   Then by the standard theory of viscosity solution (c.f. \cite{Crandall Ishii Lions BAMS}), $\mathrm{U}_{\infty}(x,t)$ is a viscosity subsolution bounded from above by $\mathrm{O}_{e^{\infty}}(x,t)$, where $e^{\infty} := (\vartheta, q, s) \in \mathbb{E}$. Therefore, by the maximality of the global obstacle subsolution, we have that
   \begin{eqnarray}\label{the limit of domain obstacle subsolution is no bigger than the global obstacle subsolution}
   \mathrm{U}_{\infty}(x,t) \leq \overline{\mathrm{U}}_{a^{\infty}}(x,t), && (x,t) \in \RR^{n} \times (0, \infty).
   \end{eqnarray}
  Let us now consider the set $\mathfrak{T}_{\ell} \subseteq \Upomega\left(0, \frac{\sqrt{n}}{2}; \vartheta\right) \times (0, \infty)$, on which $\overline{\mathrm{U}}_{a_{\ell}}(x,t)$ touches $\mathrm{O}_{e^{\infty}}(x,t)$. To be more precise,
  \begin{equation*}
  \mathfrak{T}_{\ell} := \left\lbrace (x,t) \in \Upomega\left(0, \frac{\sqrt{n}}{2}; \vartheta\right) \times (0, \infty) \Big| \overline{\mathrm{U}}_{a_{\ell}}(x,t) = \mathrm{O}_{e^{\infty}}(x,t) \right\rbrace. 
  \end{equation*}
  Since $\overline{\mathrm{U}}_{a_{\ell}}(x,t)$ is upper semicontinuous, $\mathfrak{T}_{\ell}$, $\ell \in \NN$ are all closed sets. Moreover, we have that (c.f. an independent Lemma \ref{a simpler description of detachment for subsolution})
  \begin{eqnarray*}
  \mathfrak{T}_{\ell} \cap \left\lbrace x \in \RR^{n}\big | \alpha \leq x \cdot \vartheta \leq \alpha + \sqrt{n} \right\rbrace \neq \emptyset, &\text{for any}& \alpha \geq 0.
  \end{eqnarray*}
  Let us also set 
  \begin{equation*}
  \mathfrak{T}_{\infty} := \left\lbrace (x,t) \in \Upomega\left(0, \frac{\sqrt{n}}{2}; \vartheta\right) \times (0, \infty) \Big| \mathrm{U}_{\infty}(x,t) = \mathrm{O}_{e^{\infty}}(x,t) \right\rbrace. 
  \end{equation*}
  Then by the definition of $\mathrm{U}_{\infty}$, we have that $\mathfrak{T}_{\infty} = \cap_{\ell = 1}^{\infty}\mathfrak{T}_{\ell}$. Accordingly,
   \begin{eqnarray*}
   	\mathfrak{T}_{\infty} \cap \left\lbrace x \in \RR^{n}\big | \alpha \leq x \cdot \vartheta \leq \alpha + \sqrt{n} \right\rbrace \neq \emptyset, &\text{for any}& \alpha \geq 0, 
   \end{eqnarray*}  
   which means that $\mathrm{U}_{\infty}(x,t)$ does not detach from $\mathrm{O}_{e^{\infty}}(x,t)$, neither does $\overline{\mathrm{U}}_{a^{\infty}}(x,t)$ (this is due to (\ref{the limit of domain obstacle subsolution is no bigger than the global obstacle subsolution})). Then, $\overline{s}^{\infty}(\vartheta) \leq s$, which contradicts to the definition of $\overline{s}^{\infty}(\vartheta)$.
\end{proof}

\begin{proposition}[Detachment Property]\label{the detachment property for global obstacle subsolution in irrational directions}
	Fix $\nu \in \mathbb{S}^{n-1}\diagdown \RR\ZZ^{n}$, $q \in \RR^{n}\diagdown \left\lbrace 0\right\rbrace $ and $\mu \in \RR$, then for any $\delta > 0$, there exists a number $\hat{B} := \hat{B}(\nu, \delta) > 0$, such that
	\begin{eqnarray*}
		\overline{\mathrm{U}}_{a^{\infty}}\left( \cdot - \left(\frac{\delta}{2}t - \hat{B}\right)\nu, t \right) \prec_{\left( \RR^{n}, \mu\right)} \mathrm{O}_{e^{\infty}}(\cdot, t), &&  t \geq 0, 
	\end{eqnarray*}
	where
	\begin{eqnarray*}
		a^{\infty} := (\nu, \infty, 0, q, \overline{s}^{\infty}(\nu) + \delta) \in \mathbb{A} &\text{and}& e^{\infty} := (\nu, q, \overline{s}^{\infty}(\nu) + \delta) \in \mathbb{E}.
	\end{eqnarray*}
\end{proposition}

\begin{proof}
	For any $T > 0$ and $r > \sqrt{n}$, by Proposition \ref{the detachment lemma subsolution}, there exist $R = R(\nu, T, r, \delta) > 0$ and $B = B(\nu, \delta) > 0$, such that
	\begin{eqnarray*}
		\overline{\mathrm{U}}_{a}\left( \cdot - \left(\frac{\delta}{2}t - B\right)\nu, t \right) \prec_{\left( \Upomega(0, r; \nu), \mu\right)} \mathrm{O}_{e}(\cdot, t), &&  0 \leq t \leq T.
	\end{eqnarray*}	
   Then since $\overline{\mathrm{U}}_{a^{\infty}} \leq \overline{\mathrm{U}}_{a}$, we have for all $t \geq 0$ (because $a^{\infty}$ is independent of $T$) that
	\begin{eqnarray*}
		\overline{\mathrm{U}}_{a^{\infty}}\left( \cdot - \left(\frac{\delta}{2}t - B\right)\nu, t \right) \prec_{\left( \Upomega(0, r; \nu), \mu\right)} \mathrm{O}_{e^{\infty}}(\cdot, t).
	\end{eqnarray*}	
    Let us state a Birkhoff property for the global obstacle subsolution $\overline{\mathrm{U}}_{a^{\infty}}(x,t)$. For any $\Delta z \in \ZZ^{n}$, such that $\Delta z \cdot \nu > 0$, let $\Delta t := \frac{\Delta z \cdot \nu}{\overline{s}^{\infty}(\nu) + \delta}$, then by the maximality of $\overline{\mathrm{U}}_{a^{\infty}}$, we get that
    \begin{eqnarray*}
    \overline{\mathrm{U}}_{a^{\infty}}(x + \Delta z, t + \Delta t) \leq \overline{\mathrm{U}}_{a^{\infty}}(x, t), && (x,t) \in \RR^{n} \times (0,\infty).
    \end{eqnarray*}
    Due to $r > \sqrt{n}$, for any $x \in \RR^{n}$, there exists $\Delta z \in \ZZ^{n}$, such that $|(x - \Delta z) - [(x - \Delta z)\cdot \nu] \nu| < \sqrt{n}$, $\Delta z \cdot \nu > 0$ and $\Delta t := \frac{\Delta z \cdot \nu}{\overline{s}^{\infty}(\nu) + \delta} < \frac{\sqrt{n}}{m_0}$. Consider any $t > \frac{\sqrt{n}}{m_0}$, we have that
    \begin{eqnarray*}
    \overline{\mathrm{U}}_{a^{\infty}}(x, t) &\leq& \overline{\mathrm{U}}_{a^{\infty}}(x - \Delta z, t - \Delta t) \\
    &\prec_{(\Upomega(0, r; \nu), \mu)}& \mathrm{O}_{e^{\infty}}(x - \Delta z + (\frac{\delta}{2}(t - \Delta t) - B)\nu, t - \Delta t)\\
    &=& \mathrm{O}_{e^{\infty}}\left( x + \left[ \frac{\delta}{2}t - (B + \frac{\delta}{2}\Delta t)\right]\nu, t \right).
    \end{eqnarray*}
   Therefore,
   \begin{equation*}
   \overline{\mathrm{U}}_{a^{\infty}}\left( \cdot - \left[ \frac{\delta}{2}t - (B + \frac{\delta}{2}\Delta t)\right]\nu, t\right) \prec_{(\RR^{n}, \nu)} \mathrm{O}_{e^{\infty}}(\cdot, t). 
   \end{equation*}
   By taking $\hat{B} = B + \frac{\delta \sqrt{n}}{2m_0}$, we have the desired result.
\end{proof}

Even though we do not have detachment property for obstacle sub/super solutions associated to rational directions in a cylinder with finite radius, we have similar result for the global obstacle sub/super solutions.

\begin{proposition}[Detachment Property]\label{the detachment property for rational directions}
	Fix $\vartheta \in \mathbb{S}^{n-1}\cap \RR\ZZ^{n}$, $q \in \RR^{n}\diagdown \left\lbrace 0\right\rbrace $ and $\mu \in \RR$, then for any $\delta > 0$, there exists a number $B := B(\vartheta, \delta) > 0$, such that
	\begin{eqnarray*}
	\overline{\mathrm{U}}_{a^{\infty}}\left( \cdot - \left(\frac{\delta}{2}t - B\right)\vartheta, t \right) \prec_{\left( \RR^{n}, \mu\right)} \mathrm{O}_{e^{\infty}}(\cdot, t), &&  t \geq 0, 
	\end{eqnarray*}
    where
    \begin{eqnarray*}
    a^{\infty} := (\vartheta, \infty, 0, q, \overline{s}^{\infty}(\vartheta) + \delta) \in \mathbb{A} &\text{and}& e^{\infty} := (\vartheta, q, \overline{s}^{\infty} + \delta) \in \mathbb{E}.
    \end{eqnarray*}
\end{proposition}

\begin{proof}
	Let us also set 
	\begin{eqnarray*}
	\hat{a}^{\infty} := \left(\vartheta, \infty, 0, q, \overline{s}^{\infty} + \frac{\delta}{2}\right) &\text{and}& \hat{e}^{\infty} := \left(\vartheta, q, \overline{s}^{\infty} + \frac{\delta}{2}\right) \in \mathbb{E}.
	\end{eqnarray*}
   By the definition of $\overline{s}^{\infty}(\vartheta)$, there exist $r > \frac{\sqrt{n}}{2}$ and $t_{0} := t_{0}(\vartheta, \delta) > 0$, such that 
   \begin{eqnarray*}
   \overline{\mathrm{U}}_{a^{\infty}}(\cdot, t) \prec_{\left(\Upomega(0, r; \vartheta), \mu \right)} \mathrm{O}_{e^{\infty}}(\cdot, t) \hspace{2mm} \text{and} \hspace{2mm} \overline{\mathrm{U}}_{\hat{a}^{\infty}}(\cdot, t) \prec_{\left(\Upomega(0, r; \vartheta), \mu \right)} \mathrm{O}_{\hat{e}^{\infty}}(\cdot, t), && t \geq t_{0}.
   \end{eqnarray*}
   Because $\vartheta \in \mathbb{S}^{n-1} \cap \RR\ZZ^{n}$, both $\overline{\mathrm{U}}_{a^{\infty}}(\cdot, t)$ and $\overline{\mathrm{U}}_{\hat{a}^{\infty}}(\cdot, t)$ have periodic structure (through not necessarily $\ZZ^{n}$ periodic), therefore (chose a larger $t_{0}$ if necessary), 
   \begin{eqnarray*}
   	\overline{\mathrm{U}}_{a^{\infty}}(\cdot, t) \prec_{\left(\RR^{n}, \mu \right)} \mathrm{O}_{e^{\infty}}(\cdot, t) \hspace{2mm} \text{and} \hspace{2mm} \overline{\mathrm{U}}_{\hat{a}^{\infty}}(\cdot, t) \prec_{\left(\RR^{n}, \mu \right)} \mathrm{O}_{\hat{e}^{\infty}}(\cdot, t), && t \geq t_{0}.
   \end{eqnarray*}   
   Again since $\vartheta \in \mathbb{S}^{n-1} \cap \RR\ZZ^{n}$, there exists $B \geq (M_0 - m_0)t_{0}$, such that $B \vartheta \in \ZZ^{n}$, then
   \begin{eqnarray*}
   \overline{\mathrm{U}}_{a^{\infty}}(\cdot + B\vartheta, t) \prec_{\left(\RR^{n}, \mu \right)} \overline{\mathrm{U}}_{\hat{a}^{\infty}}(\cdot, t), && 0 \leq t \leq t_{0}.
   \end{eqnarray*}
   However, for any $t > t_{0}$, both $\overline{\mathrm{U}}_{a^{\infty}}(x + B\vartheta,t)$ and $\overline{\mathrm{U}}_{\hat{a}^{\infty}}(x,t)$ are globle solutions, therefore, the comparison principle implies the above ordering relation for $t > t_{0}$. As a combination, we conclude that
   \begin{eqnarray*}
   	\overline{\mathrm{U}}_{a^{\infty}}(\cdot + B\vartheta, t) \prec_{\left(\RR^{n}, \mu \right)} \overline{\mathrm{U}}_{\hat{a}^{\infty}}(\cdot, t), && t \geq 0.
   \end{eqnarray*}
   Finally, the desired result follows (for any $t \geq 0$) as
   \begin{eqnarray*}
   \overline{\mathrm{U}}_{a^{\infty}}\left(\cdot - \left(\frac{\delta}{2}t - B\right)\vartheta, t\right) &\prec_{(\RR^{n},\mu)}& \overline{\mathrm{U}}_{\hat{a}^{\infty}}\left(\cdot - \frac{\delta}{2}t \vartheta, t\right) \leq \mathrm{O}_{\hat{e}^{\infty}}\left(\cdot - \frac{\delta}{2}t \vartheta, t\right) = \mathrm{O}_{e^{\infty}}\left( \cdot, t\right).  
   \end{eqnarray*}
\end{proof}

\subsection{Continuity and ordering}

\subsubsection{The semicontinuity}

\begin{lemma}\label{a simpler description of detachment for subsolution}
	Fix $\mu \in \RR$, $a := (\nu, R, 0, q, s) \in \mathbb{A}$ with $R > \frac{\sqrt{n}}{2}$ and set $e := (\nu, q, s) \in \mathbb{E}$, if there exists $r > \frac{\sqrt{n}}{2}$ and $T_{0} > 0$, such that
	\begin{eqnarray*}
	\overline{\mathrm{U}}_{a}(\cdot, t) \prec_{\left( \Upomega(0, r; \nu), \mu\right)} \mathrm{O}_{e}(\cdot, t), && T_{0} \leq t \leq T_{0} + \frac{\sqrt{n}}{s}\\
	(\text{resp.} \hspace{2mm} \mathrm{O}_{e}(\cdot, t) \prec_{\left( \Upomega(0, r; \nu), \mu\right)} \underline{\mathrm{U}}_{a}(\cdot, t) , && T_{0} \leq t \leq T_{0} + \frac{\sqrt{n}}{s} ).
	\end{eqnarray*}
    Then we have for $\hat{a} := (\nu, R + \sqrt{n}, 0, q, s)$ that
 	\begin{eqnarray*}
 	\overline{\mathrm{U}}_{\hat{a}}(\cdot, t) \prec_{\left( \Upomega(0, r; \nu), \mu\right)} \mathrm{O}_{e}(\cdot, t) \hspace{2mm} (\text{resp.} \hspace{2mm} \mathrm{O}_{e}(\cdot, t) \prec_{\left( \Upomega(0, r; \nu), \mu\right)} \underline{\mathrm{U}}_{\hat{a}}(\cdot, t)), && t \geq T_{0},
 	\end{eqnarray*} 	
   i.e., $\overline{\mathrm{U}}_{\hat{a}}(x,t)$ (resp. $\underline{\mathrm{U}}_{\hat{a}}(x,t)$) detaches from $\mathrm{O}_{e}(x,t)$ at the $\mu$ level set. 
\end{lemma}

\begin{proof}
	By Proposition \ref{the birkhoff property for subsolutions in two static domains}, we have that 
 	\begin{eqnarray*}
 		\overline{\mathrm{U}}_{\hat{a}}(\cdot, t) \leq \overline{\mathrm{U}}_{a}(\cdot, t) \prec_{\left( \Upomega(0, r; \nu), \mu\right)} \mathrm{O}_{e}(\cdot, t), && T_{0} \leq t \leq T_{0} + \frac{\sqrt{n}}{s}.
 	\end{eqnarray*}
    We only need to prove the following ordering relation.
    \begin{eqnarray*}
    \overline{\mathrm{U}}_{\hat{a}}(\cdot, t) \prec_{\left( \Upomega(0, r; \nu), \mu\right)} \mathrm{O}_{e}(\cdot, t), && t > T_{0} + \frac{\sqrt{n}}{s}.
    \end{eqnarray*}
   For any $x_{0} \in \Upomega(0, r; \nu)$ and $t_{0} > T_{0} + \frac{\sqrt{n}}{s}$, then there exists $\Delta z \in \ZZ^{n}$, such that
   \begin{eqnarray*}
   \left( t_{0} - T_{0}\right)s - \sqrt{n}  \leq \Delta z \cdot \nu \leq \left( t_{0} - T_{0}\right)s  &\text{and}& x_{0} - \Delta z \in \Upomega(0,r;\nu).
   \end{eqnarray*}
   Let us set $\Delta t := \frac{\Delta z \cdot \nu}{s}$, then
   \begin{eqnarray*}
   \overline{\mathrm{U}}_{\hat{a}}(x_{0},t_{0}) \leq \overline{\mathrm{U}}_{\hat{a}}(x_{0} - \Delta z, t_{0} - \Delta t) \leq \overline{\mathrm{U}}_{a}(x_{0} - \Delta z, t_{0} - \Delta t) < \mathrm{O}_{e}(x_{0} - \Delta z, t_{0} - \Delta t) = \mathrm{O}_{e}(x_{0}, t_{0}).
   \end{eqnarray*} 
   Since the above $(x_{0}, t_{0})$ is arbitrary, the desired result follows. 
\end{proof}

\begin{proposition}[semicontinuity]\label{the semicontinuity of the head and tail speed}
	The head (resp. tail) speed $\overline{s}(\nu)$ (resp. $\underline{s}(\nu)$) $: \mathbb{S}^{n-1} \rightarrow \left[m_0, M_0\right] $ is upper (resp. lower) semicontinuous.
\end{proposition}

\begin{proof}
    \begin{figure}[h]
    \centering
    \includegraphics[width=0.7\linewidth]{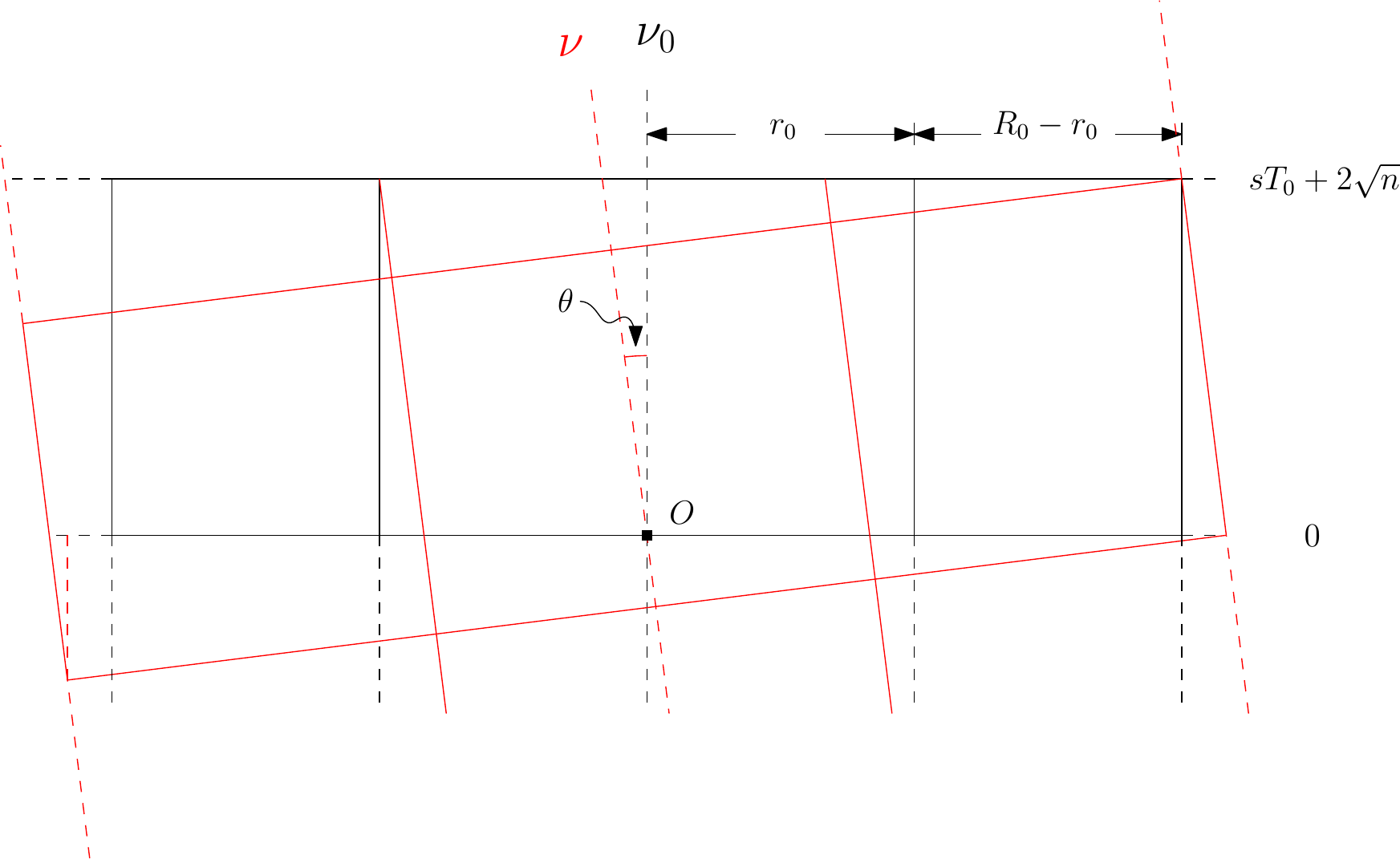}
    \caption{The upper semicontinuity of $\overline{s}(\nu): \mathbb{S}^{n-1} \rightarrow \left[ m_0, M_0\right] $}
    \label{fig:3}
    \end{figure}
	Fix any $\nu_{0} \in \mathbb{S}^{n-1}$, $\delta > 0$ and set $s := \overline{s}(\nu_{0}) + \delta$. To prove the upper semicontinuity of $\overline{s}(\nu_{0})$, it suffices to show the statement: there exists a neighborhood $\mathcal{N}(\nu_{0})$ of $\nu_{0}$ in $\mathbb{S}^{n-1}$, such that for any $\nu \in \mathcal{N}(\nu_{0})$, there exists $R > 0$, such that $\overline{\mathrm{U}}_{a}(x,t)$ detaches from $\mathrm{O}_{e}(x,t)$, where
	\begin{eqnarray*}
	a := (\nu, R, 0, q, s) &\text{and}& e := (\nu, q, s) \in \mathbb{E}.
	\end{eqnarray*}
    For the simplicity of the notation, we shall only prove the detachment at the zero level set. The case of general $\mu$ level set can be established similarly.\\
    \underline{Step 1.} As $s > \overline{s}(\nu_{0})$, there exist $R_{0} > r_{0} > 2\sqrt{n}$ and $T_{0} > 0$, such that
    \begin{eqnarray*}
    \overline{\mathrm{U}}_{a_{0}}(\cdot, t) \prec_{\left(\Upomega(0,r_{0};\nu_{0}); 0 \right)} \mathrm{O}_{e_{0}}(\cdot,t), && t \geq T_{0},
    \end{eqnarray*}
    where
    \begin{eqnarray*}
    a_{0} := (\nu_{0}, R_{0}, 0, q, s) \in \mathbb{A} &\text{and}& e_{0} := (\nu_{0}, q, s) \in \mathbb{E}.
    \end{eqnarray*}
    By Proposition \ref{the detachment lemma subsolution}, we can adjust $R_{0}$ and $T_{0}$ if necessary, such that
    \begin{equation}\label{the big detachment for the initial obstacle solution in proving semicontinuity}
    \dist\left( L_{0}^{+}\left( \overline{\mathrm{U}}_{a_{0}}(\cdot, t); \Upomega(0,r_{0};\nu_{0})\right), L_{0}^{-}\left( \mathrm{O}_{e_{0}}(\cdot, t)\right); \Upomega(0,r_{0};\nu_{0})\right) > 2\sqrt{n},
    \end{equation}
    for any $t \in \left[ T_{0}, T_{0} + \frac{2\sqrt{n}}{s}\right]$. \\
    \underline{Step 2.} Let us define a set of angles
    \begin{equation}\label{the neighborhood of the angle in proving semicontinuity}
    \Theta := \left\lbrace \theta \in \left( 0, \frac{\pi}{2}\right) \Big | \begin{cases}
    r_{0}\cos\theta - \left( sT_{0} + 2\sqrt{n}\right)\sin\theta > 2\sqrt{n}\\
    R_{0}\sin\theta\cos\theta + \left( sT_{0} + 2\sqrt{n}\right)\sin^{2}\theta < \frac{\sqrt{n}}{4} 
    \end{cases}  \right\rbrace .
    \end{equation}
    Then we consider the neighborhood of direction $\nu_{0}$.
    \begin{equation*}
    \mathcal{N}(\nu_{0}) := \left\lbrace \nu \in \mathbb{S}^{n-1} \Big| \theta \in \Theta, \text{ where } \theta \text{ is the angle between } \nu \text{ and } \nu_{0}\right\rbrace. 
    \end{equation*}
    Our aim is to construct a larger cylinder of the form $\Upomega(0, R; \nu)$ that includes the interesting part of the above cylinder, i.e., 
    \begin{equation*}
    \left\lbrace x \in \Upomega(0, R_{0}; \nu_{0}) \big| 0 \leq x\cdot \nu_{0} \leq sT_{0} + 2\sqrt{n}\right\rbrace. 
    \end{equation*}
    Let us take $R$ and $r$ as follows, where $\nu \in \mathcal{N}(\nu_{0})$.
    \begin{eqnarray*}
    R_{1} := R_{0}\cos\theta + \left( sT_{0} + 2\sqrt{n}\right) \sin\theta, && r_{1} := r_{0}\cos\theta - \left( sT_{0} + 2\sqrt{n}\right) \sin\theta.
    \end{eqnarray*}
   The cylinder and the obstacle are as follows.
   \begin{eqnarray*}
   \mathrm{C} := \Upomega(-\left( R_{1}\tan\theta\right)  \nu, R_{1}; \nu), && \mathrm{O}(x,t) := \left[ \left( x \cdot \nu\right) + R_{1}\sin\theta - st \right] \left( - |q|\right).
   \end{eqnarray*}
   Then we have properties
   \begin{equation*}
   \left\lbrace x \in \Upomega(0, R_{0}; \nu_{0}) \big| 0 \leq x \cdot \nu_{0} \leq sT_{0} + 2\sqrt{n}\right\rbrace \subseteq \Upomega(0, R_{1}; \nu),
   \end{equation*}     
   \begin{eqnarray*}
   \mathrm{O}(x, t) \prec_{(\Upomega(0, R_{0}; \nu_{0}), 0)} \mathrm{O}_{e}(x, t), \hspace{2mm} x \in \Upomega(0, R_{0}; \nu_{0}), \hspace{2mm} 0 \leq t \leq T_{0} + \frac{2\sqrt{n}}{s}.
   \end{eqnarray*}  
  Similar to Definition \ref{the definitions of obstacle solutions}, let us consider the set of subsolutions bounded from above by $\mathrm{O}(x,t)$.
  \begin{equation*}
  \overline{\mathscr{S}} := \left\lbrace u \in \USC(\Upomega(0, R_{1}; \nu) \times (0, \infty))\big | \hspace{1mm} u_{t} \leq \mathscr{F}\left(D^{2}u, Du, x \right), \hspace{1mm} u(x,t) \leq \mathrm{O}(x,t) \right\rbrace. 
  \end{equation*}  
  And in particular, we have the associated largest subsolution $\mathrm{U}(x,t)$ as follows.  
  \begin{equation*}
  \mathrm{U}(x,t) := \left( \sup\left\lbrace u(x,t) \big| u \in \overline{\mathscr{S}}\right\rbrace \right)^{*}.
  \end{equation*}
  \underline{Step 3.} Since $\mathrm{U}(x,t)$ is a subsolution in $\Upomega(0, R; \nu)$, so is $\mathrm{Z}(x,t)$ defined as follows:
  \begin{equation*}
  \mathrm{Z}(x,t) := \begin{cases}
  0, & x \in L_{0}^{+}\left( \mathrm{U}(\cdot,t); \Upomega(-(R_{1}\tan\theta)\nu, R_{1}; \nu)\right), \\
  -\infty, & x \in \Upomega(-(R_{1}\tan\theta)\nu, R_{1}; \nu) \diagdown L_{0}^{+}\left( \mathrm{U}(\cdot,t); \Upomega(-(R_{1}\tan\theta)\nu, R_{1}; \nu)\right).
  \end{cases}
  \end{equation*}
  Consider the modified super zero level set $\mathscr{L}_{0}^{+}$ in $\Upomega(0, R_{0}; \nu_{0})$
  \begin{equation*}
  \mathscr{L}_{0}^{+} := \left\lbrace x \in \Upomega(0, R_{0}; \nu_{0}) \big| x \cdot \nu \leq - R_{1}\tan\theta + \frac{m_0 t}{\cos\theta}\right\rbrace \cup \left\lbrace x \in \Upomega(0, R_{0}; \nu_{0}) \big| \mathrm{U}(x,t) \geq 0 \right\rbrace.
  \end{equation*}
  Then we can define the modified characteristic function.
  \begin{equation*}
  \mathscr{Z}(x,t) := \begin{cases}
  0, & x \in \mathscr{L}_{0}^{+}\\
  -\infty, & x \in \Upomega(0, R_{0}; \nu_{0}) \diagdown \mathscr{L}_{0}^{+}
  \end{cases} \in \overline{\mathscr{S}}_{a}.
  \end{equation*}
  Therefore, by the maximality of $\overline{\mathrm{U}}_{a_{0}}(x,t)$, we conclude that
  \begin{eqnarray*}
  \mathscr{Z}(x,t) \leq \overline{\mathrm{U}}_{a_{0}}(x,t), && x \in \Upomega(0, R_{0}; \nu_{0}), \hspace{2mm} 0 \leq t \leq T_{0} + \frac{2\sqrt{n}}{s}.
  \end{eqnarray*}
  In particular, we have that
  \begin{eqnarray*}
  \mathscr{Z}(\cdot, t) \prec_{\left( \Upomega(0, R_{0}; \nu_{0}), 0\right)} \overline{\mathrm{U}}_{a_{0}}(\cdot, t), && 0 \leq t \leq T_{0} + \frac{2\sqrt{n}}{s}.
  \end{eqnarray*}
  From (\ref{the big detachment for the initial obstacle solution in proving semicontinuity}) and (\ref{the neighborhood of the angle in proving semicontinuity}), we conclude for any $T_{0} \leq t \leq T_{0} + \frac{2\sqrt{n}}{s}$ that
  \begin{eqnarray*}
  &&\dist\left( \mathscr{L}_{0}^{+}, L_{0}^{-}\left( \mathrm{O}(\cdot,t); \Upomega(0, r_{0}; \nu_{0})\right) \right)\\
  &\geq& \dist\left( L_{0}^{+}\left( \overline{\mathrm{U}}_{a_{0}}(\cdot,t); \Upomega(0, r_{0}; \nu_{0})\right), L_{0}^{-}\left( \mathrm{O}(\cdot,t); \Upomega(0, r_{0}; \nu_{0})\right)\right) \\
  &\geq& \dist\left( L_{0}^{+}\left( \overline{\mathrm{U}}_{a_{0}}(\cdot, t); \Upomega(0,r_{0};\nu_{0})\right), L_{0}^{-}\left( \mathrm{O}_{e_{0}}(\cdot, t)\right); \Upomega(0,r_{0};\nu_{0})\right) - 2R_{1}\tan\theta > \frac{3\sqrt{n}}{2}
  \end{eqnarray*}
  \underline{Step 4.} Now let us set
  \begin{eqnarray*}
  R = R_{1} + \sqrt{n}, \hspace{2mm} r := r_{1} - \sqrt{n} &\text{and}& \xi_{0} := \argmin\limits\limits_{\xi \in \ZZ^{n}, \hspace{1mm} \xi \cdot \nu \geq R\tan\theta} |\xi|.
  \end{eqnarray*}
  Then $0 < \xi_{0} \cdot \nu < \frac{3\sqrt{n}}{2}$, we shall compare the standard obstacle subsolution $\overline{\mathrm{U}}_{a}(x,t)$ and $\mathrm{U}(x,t)$ at the zero level set in $\Upomega(0, r; \nu)$. Therefore,
  \begin{eqnarray*}
  &&\dist\left( L_{0}^{+}\left( \overline{\mathrm{U}}_{a}(\cdot,t); \Upomega(0, r; \nu)\right),  L_{0}^{-}\left( \mathrm{O}_{e}(\cdot, t); \Upomega(0, r; \nu)\right) \right) \\
  &\geq& \dist\left( L_{0}^{+}\left( \mathrm{U}(\cdot - \xi_{0}, t); \Upomega(0, r; \nu)\right), L_{0}^{-}\left(\mathrm{O}(\cdot - \xi_{0}, t); \Upomega(0, r; \nu) \right)  \right) \\
  &>& \frac{3\sqrt{n}}{2} - \xi_{0}\cdot \nu > 0, \hspace{2cm} T_{0} \leq t \leq T_{0} + \frac{2\sqrt{n}}{s}.
  \end{eqnarray*}
  By Lemma \ref{a simpler description of detachment for subsolution}, we then conclude that $\overline{\mathrm{U}}_{a}(x,t)$ detaches from $\mathrm{O}_{e}(x,t)$. Hence,
  \begin{eqnarray*}
  \overline{s}(\nu_{0}) + \delta \geq \overline{s}(\nu), &\text{for any}& \nu \in \mathcal{N}(\nu_{0}).
  \end{eqnarray*}
  In other words, $\overline{s}(\nu): \mathbb{S}^{n-1} \rightarrow \left[m_0, M_0 \right]$ is upper semicontinuous.\\
  Similarly, we can show that $\underline{s}(\nu) : \mathbb{S}^{n-1} \rightarrow \left[ m_0, M_0 \right]$ is lower semicontinuous.
\end{proof}

\subsubsection{The continuity of head and tail speeds}

\begin{proposition}[Continuity]\label{the continuity of head and tail speeds}
	The functions $\overline{s}(\nu), \underline{s}(\nu): \mathbb{S}^{n-1} \rightarrow \left[m_0, M_0\right] $ are both continuous.
\end{proposition}

\begin{proof}
	Let us only prove the continuity of $\overline{s}$, since the case of $\underline{s}$ can be argued similarly. By Proposition \ref{the semicontinuity of the head and tail speed}, it suffice to show that for any $\vartheta \in \mathbb{S}^{n-1}$ and $\nu_{\ell} \in \mathbb{S}^{n-1}$, such that $\nu_{\ell} \rightarrow \vartheta$, then $\liminf_{\ell \rightarrow \infty}\overline{s}(\nu_{\ell}) \geq \overline{s}(\vartheta)$. Assume this is not true, then according to Proposition \ref{the semicontinuity of the head and tail speed} and Proposition \ref{the equivalence between domain obstacle semi solutions and global semi solutions in defining head and tail speeds}, we have that (up to a subsequence if necessary)
	\begin{eqnarray*}
	\overline{s}^{\infty}(\vartheta) = \overline{s}(\vartheta) = \lim_{\ell \rightarrow \infty} \overline{s}(\nu_{\ell}) + \delta = \lim_{\ell \rightarrow \infty} \overline{s}^{\infty}(\nu_{\ell}) + \delta, &\text{with}& \delta > 0.
	\end{eqnarray*}
    Fix $0 < \sigma < \frac{\delta}{5}$ and $s := \overline{s}^{\infty}(\vartheta) + \sigma$. Then from Proposition \ref{the detachment property for global obstacle subsolution in irrational directions}, we have ($\mu \in \RR$) that
	\begin{eqnarray*}
		\overline{\mathrm{U}}_{a_{\ell}^{\infty}}\left( \cdot - \left(\frac{s - \overline{s}^{\infty}(\nu_{\ell})}{2}t - \hat{B}_{\ell}\right)\nu_{\ell}, t \right) \prec_{\left( \RR^{n}, \mu\right)} \mathrm{O}_{e_{\ell}^{\infty}}(\cdot, t), &&  t > \frac{\sqrt{n}}{m_0},
	\end{eqnarray*}
	where
	\begin{eqnarray*}
		a_{\ell}^{\infty} := (\nu_{\ell}, \infty, 0, q, s) \in \mathbb{A} &\text{and}& e_{\ell}^{\infty} := (\nu_{\ell}, q, s) \in \mathbb{E}.
	\end{eqnarray*}    
    Then similar to the argument of Proposition \ref{the semicontinuity of the head and tail speed}, we have the upper semicontinuity of the detachment time with respect to the direction. Since the obstacle subsolution associated to the speed $s$ and direction $\vartheta$ detaches at a finite time. Then we have that $\limsup_{\ell \rightarrow \infty} \hat{B}_{\ell} \leq \hat{B} < \infty$ for some number $\hat{B}$. Then, we have that
    \begin{equation*}
    \liminf_{\ell \rightarrow \infty} \left( \frac{s - \overline{s}^{\infty}(\nu_{\ell})}{2}t - \hat{B}_{\ell}\right) \geq \frac{\delta + \sigma}{2}t - \hat{B} > 3\sigma t - \hat{B}.
    \end{equation*}
    Let us set $T_{0} := \frac{2(\hat{B} + 1)}{\delta + \sigma}$, recalling Proposition \ref{a semi obstacle solution becomes a solution when detachment happens}, there exists $\ell_{0}$, such that for any $t > T_{0}$ and any $\ell > \ell_{0}$, $\overline{\mathrm{U}}_{a_{\ell}^{\infty}}(x,t)$ is a global solution. Denote
    \begin{eqnarray*}
    \mathrm{U}_{\star}(x,t) := \liminf_{\eta \rightarrow 0} \left\lbrace \left( \overline{\mathrm{U}}_{a_{\ell}^{\infty}}\right)_{*}(y,s)\Big| |y - x| + |s - t| + \frac{1}{\ell} < \eta \right\rbrace, && t > T_{0},
    \end{eqnarray*} 
    which is a global supersolution detached from (at each $\mu$ level set) the obstacle by at least $3\sigma t - \hat{B}$. Next, we consider the obstacle subsolution
    \begin{eqnarray*}
    \overline{\mathrm{U}}_{a^{\infty}}(x,t), &\text{with}& a^{\infty} := (\vartheta, \infty, 0, q, s) \in \mathbb{A}.
    \end{eqnarray*}
   Then by comparison principle (c.f. Proposition \ref{the usual comparison principle}), we have that
   \begin{eqnarray*}
   \overline{\mathrm{U}}_{a^{\infty}}(x,t) - \mathrm{U}_{\star}(x,t) \leq \sup_{y \in \RR^{n}}\left( \overline{\mathrm{U}}_{a^{\infty}}(y,T_{0}) - \mathrm{U}_{\star}(y,T_{0}) \right) \leq sT_{0}|q| < \infty, && t > T_{0}.
   \end{eqnarray*}
  On the other hand, due to the ordering relation
  \begin{eqnarray*}
  \mathrm{U}_{\star}(x,t) - \mathrm{O}_{e^{\infty}}(x,t) \leq - \left( 3\sigma t - \hat{B}\right) |q|, && e^{\infty} := (\vartheta, q, s) \in \mathbb{E}, \hspace{2mm} t > T_{0}.
  \end{eqnarray*}
  We then have
  \begin{eqnarray*}
  \overline{\mathrm{U}}_{a^{\infty}}(x,t) - \mathrm{O}_{e^{\infty}}(x,t) \leq - \left( 3\sigma t - \hat{B}\right) |q| + sT_{0}|q|, && t > T_{0}, 
  \end{eqnarray*}
   which implies that
   \begin{eqnarray*}
   \overline{\mathrm{U}}_{a^{\infty}}(\cdot - \left( 3\sigma t - \hat{B} - sT_{0} \right)\vartheta, t ) \prec_{(\RR^{n}, \mu)} \mathrm{O}_{e^{\infty}}(\cdot, t), && t > T_{0}.
   \end{eqnarray*}
   Finally, let us consider $\hat{a}^{\infty} := (\vartheta, \infty, 0, q, \hat{s}) \in \mathbb{A}$ and $\hat{e}^{\infty} := (\vartheta, q, \hat{s})$ with $\hat{s} = s - 2\sigma = \overline{s}^{\infty}(\vartheta) - \sigma$, then
   \begin{eqnarray*}
   	\overline{\mathrm{U}}_{\hat{a}^{\infty}}\left( \cdot - (\sigma t - \hat{B} - sT_{0})\vartheta, t\right) \leq \overline{\mathrm{U}}_{a^{\infty}}(\cdot - (\sigma t - \hat{B} - sT_{0})\vartheta, t) \prec_{\left( \RR^{n}, \mu\right)} \mathrm{O}_{e^{\infty}}(\cdot + 2\sigma t, t) = \mathrm{O}_{\hat{e}^{\infty}}(\cdot, t). 
   \end{eqnarray*}
   By Definition \ref{the global head speed in an rational direction}, we must have that $\overline{s}^{\infty}(\vartheta) \leq \hat{s}$, which is a contradiction. 
\end{proof}

\subsubsection{The ordering relation in all directions}

\begin{proposition}[Ordering]\label{the ordering relation in all directions}
	For any $\nu \in \mathbb{S}^{n-1}$, we have $M_0 \geq \overline{s}(\nu) \geq \underline{s}(\nu) \geq m_0$.
\end{proposition}

\begin{proof}
	From Proposition \ref{the ordering relation in irrational directions}, we already have the ordering relation for all irrational directions. Then according to the upper semicontinuity of $\overline{s}(\cdot)$ and the lower semicontinuity of $\underline{s}(\cdot)$, we can prove the relation for all rational directions. More precisely, let $\vartheta\in\mathbb{S}^{n-1}\cap \RR\ZZ^{n}$ and let $\left\lbrace \nu_{k}\right\rbrace_{k \geq 1} \subseteq \mathbb{S}^{n-1}\diagdown \RR\ZZ^{n}$ such that $\lim_{k \rightarrow 0}|\nu_{k} - \vartheta| = 0$, then
	\begin{equation*}
	\overline{s}(\vartheta) \geq \limsup_{k \rightarrow \infty}\overline{s}(\nu_{k}) \geq \limsup_{k\rightarrow \infty} \underline{s}(\nu_{k}) \geq \liminf_{k\rightarrow \infty} \underline{s}(\nu_{k}) \geq \underline{s}(\vartheta)
	\end{equation*}
\end{proof}

\section{Homogenization}\label{the section of homogenization}

\begin{definition}\label{half relaxed limits}
	For $0 < \varepsilon < 1$, let $u^{\varepsilon}(x,t)$ be the solution in (\ref{the scaled forced mean curvature flow}), for any $(x, t) \in \RR^{n} \times (0, \infty)$, let us denote the upper and lower half relaxed limits:
	\begin{eqnarray*}
		u^{\star}(x,t) &:=& \limsup_{\eta \rightarrow 0}\left\lbrace u^{\varepsilon}(y,s) \big| |y - x| + |s - t| + \varepsilon < \eta \right\rbrace, \\
		u_{\star}(x,t) &:=& \liminf_{\eta \rightarrow 0}\left\lbrace u^{\varepsilon}(y,s) \big| |y - x| + |s - t| + \varepsilon < \eta \right\rbrace. 
	\end{eqnarray*}
\end{definition}

\begin{proposition}\label{the sub equation satisfied by the limit}
	Let $\nu \in \mathbb{S}^{n-1}$ and $\overline{s}(\nu)$ be the head speed in the $\nu$ direction. Let $\phi(x,t)$ be a $C^{2,1}$ function, assume $u^{\star}(x,t) - \phi(x,t)$ obtains a strict local maximum at $(x_{0},t_{0}) \in \RR^{n} \times (0, \infty)$, denote $q_{0} := D\phi(x_{0},t_{0})$, then
	\begin{eqnarray*}
		\begin{cases}
			\phi_{t}(x_{0},t_{0}) \leq \overline{s}(\nu)|D\phi(x_{0},t_{0})|, & q_{0} \neq 0, \hspace{2mm} \nu := - \frac{q_{0}}{|q_{0}|},\\
			\phi_{t}(x_{0},t_{0}) \leq 0, & q_{0} = 0.
		\end{cases}
	\end{eqnarray*}
\end{proposition}

\begin{proof}
	\underline{Case 1:} $D\phi(x_{0},t_{0}) \neq 0$. To prove the statement, let us derive a contradiction from the following contrary hypothesis. Assume the existence of $\delta > 0$, such that
	\begin{eqnarray*}
	\phi_{t}(x,t) > \left( \overline{s}(\nu) + 3\delta\right) |D\phi(x,t)|, && (x,t) \in U(\delta) := C(x_{0}, \delta;\nu) \times (t_{0} - \delta, t_{0} + \delta),
	\end{eqnarray*}
    where
    \begin{equation*}
    C(x_{0}, \delta; \nu) := \left\lbrace x \in \RR^{n} \big| |(x - x_{0}) - ((x - x_{0})\cdot\nu)\nu| < \delta, \hspace{2mm}|(x - x_{0})\cdot \nu| < 2\delta\right\rbrace. 
    \end{equation*}
    We can choose $\delta$ so small that
    \begin{eqnarray*}
    \text{if} \hspace{2mm} y \in \left\lbrace x \in C(x_{0}, \delta; \nu) \big| u^{\star}(x,t_{0}) = 0 \text{ or } \phi(x,t_{0}) = 0\right\rbrace, && \text{then} \hspace{2mm} (y - x_{0}) \cdot \nu < \delta.
    \end{eqnarray*}
    Let us also assume without loss of generality that $u^{\star}(x_{0},t_{0}) = 0$. Then there exists $\left\lbrace (x_{\varepsilon},t_{\varepsilon})\right\rbrace_{0 < \varepsilon \ll 1} \subseteq U(\delta)$, such that $u^{\varepsilon}(x,t) - \phi(x,t)$ obtains a strict local maximum in $U(\delta)$ at $(x_{\varepsilon},t_{\varepsilon})$. Moreover,
    \begin{eqnarray*}
    \lim_{\varepsilon\rightarrow 0}\left(\left| u^{\varepsilon}\left( x_{\varepsilon}, t_{\varepsilon}\right) - u^{\star}(x_{0}, t_{0}) \right| +  \left|x_{\varepsilon} - x_{0}\right| + \left|t_{\varepsilon} - t_{0}\right|\right) = 0. 
    \end{eqnarray*}
    Apply perturbations if necessary, suppose that in $U(\delta)$, $\phi(x,t)$ has only linear term in $t$ and quadratic terms in $x$. Denote $\phi_{t}(x_{0},t_{0}) = \hat{s}|D\phi(x_{0},t_{0})|$, where we have $\hat{s} > \overline{s}(\nu) + 3\delta$. Then the following holds (where $A := \frac{D^{2}\phi(x_{0},t_{0})}{2}$):
    \begin{eqnarray*}
    \phi(x,t) = |q_{0}|\hat{s}(t - t_{0}) + q_{0}\cdot\left(x - x_{0} \right) + (x - x_{0})\cdot A(x - x_{0})^{T}, && (x,t) \in U(\delta).
    \end{eqnarray*}
    Let us also denote $q_{\varepsilon} := D\phi(x_{\varepsilon},t_{\varepsilon})$ for small $\varepsilon$ and $(x,t) \in U(\delta)$, then
    \begin{eqnarray*}
    \phi(x,t) - \phi(x_{\varepsilon}, t_{\varepsilon}) &=& \phi_{t}(x_{\varepsilon},t_{\varepsilon}) (t - t_{\varepsilon}) + D\phi(x_{\varepsilon}, t_{\varepsilon})\cdot (x - x_{\varepsilon}) + (x - x_{\varepsilon}) \cdot A(x - x_{\varepsilon})^{T}\\
    &\leq& |q_{0}|\hat{s}(t - t_{\varepsilon}) + q_{0} \cdot (x - x_{\varepsilon}) + |q_{0} - q_{\varepsilon}||x - x_{\varepsilon}| + \lVert A\rVert_{\infty}|x - x_{\varepsilon}|^{2}.
    \end{eqnarray*}
    Let $B := B(\nu, 3\delta)$ be the constant from Proposition \ref{the detachment lemma subsolution}. We shall select a small constant $h$ as follows,
    \begin{eqnarray*}
    h := H\varepsilon &\text{with}& H := \frac{B + 2\sqrt{n} + 1}{\delta} + R,
    \end{eqnarray*}
    where $R := R\left(\nu, 3\delta, \sqrt{n}, \frac{B + 2\sqrt{n} + 1}{\delta}\right)$ is the radius from Proposition \ref{the detachment lemma subsolution}, associated to the time range $0 \leq t \leq \frac{B + 2\sqrt{n} + 1}{\delta}$. Then we have that $0 < \varepsilon R < h$. Next, we shall shift $(x_{\varepsilon},t_{\varepsilon})$ backwards as follows:
    \begin{eqnarray*}
    \tilde{y}_{\varepsilon} &:=& x_{\varepsilon} - h\hat{s}\nu + \left( \frac{\lVert A\rVert_{\infty}h^{2} + |q_{\varepsilon} - q_{0}|h}{|q_{0}|} + \sqrt{n}\varepsilon \right) \nu, \\
    y_{\varepsilon} &\in& \argmin\limits\limits_{x \in \varepsilon \ZZ^{n}}|x - \tilde{y}_{\varepsilon}|, \hspace{1cm} \tau_{\varepsilon} := t_{\varepsilon} - h.
    \end{eqnarray*}
    A direct calculation shows that
    \begin{eqnarray*}
    \phi(x,t) - \phi(x_{\varepsilon},t_{\varepsilon}) < |q_{0}|\left( \hat{s}(t - \tau_{\varepsilon}) - (x - y_{\varepsilon})\cdot \nu\right), && (x,t) \in U(h).
    \end{eqnarray*}
   Moreover, we have the estimates
   \begin{eqnarray*}
   \dist\left( y_{\varepsilon} + h\hat{s}\nu, x_{\varepsilon}\right) \leq |y_{\varepsilon} - \tilde{y}_{\varepsilon}| + |\tilde{y}_{\varepsilon} + h\hat{s}\nu - x_{\varepsilon}| \leq 2\sqrt{n}\varepsilon + \frac{\lVert A\rVert_{\infty}h^{2} + |q_{\varepsilon} - q_{0}|h}{|q_{0}|}.
   \end{eqnarray*}
   Let us consider $\varepsilon$ so small that
   \begin{eqnarray*}
   \frac{\lVert A\rVert_{\infty} H^{2}\varepsilon + |q_{\varepsilon} - q_{0}|H}{|q_{0}|} < 1 &\text{and}& 0 < h < \delta.
   \end{eqnarray*}
   Then on one hand, we have that (by rescalling $(x,t)$ to $(\varepsilon x, \varepsilon t)$ in Proposition \ref{the detachment lemma subsolution})
   \begin{equation*}
   \dist\left( y_{\varepsilon} + h\hat{s}\nu, x_{\varepsilon}\right) < \delta h - B\varepsilon.
   \end{equation*}
   On the other hand, based on the above calculations, we get (when $0 < \varepsilon \ll 1$) that
   \begin{eqnarray*}
   u^{\varepsilon}(x,t) - u^{\varepsilon}(x_{\varepsilon},t_{\varepsilon}) < |q_{0}| \left( \hat{s}(t - \tau_{\varepsilon}) - (x - y_{\varepsilon})\cdot \nu\right), && (x,t) \in U(h). 
   \end{eqnarray*}
   Because at the moment $t_{\varepsilon}$, the center of zero level set of the obstacle is $y_{\varepsilon} + h\hat{s}\nu$. We shall shift the above relation by $(-y_{\varepsilon}, -\tau_{\varepsilon})$ and rescale it to the unit scale, then apply the Proposition \ref{the detachment lemma subsolution} with the time range $0 \leq t \leq H$, finally, we scale it back to the $\varepsilon$ scale and shift it by $(y_{\varepsilon}, \tau_{\varepsilon})$, this process indicates that
   \begin{equation*}
   \dist\left( y_{\varepsilon} + h\hat{s}\nu, x_{\varepsilon}\right) > \left( 1.5\delta H - B\right) \varepsilon > \delta h - B\varepsilon,
   \end{equation*}
   which is the desired contradiction.\\
   \underline{Case 2:} $D\phi(x_{0},t_{0}) = 0$. Let us assume on the contrary that $\phi_{t}(x_{0},t_{0}) > 0$. Since $u^{\star}(x,t) - \phi(x,t)$ has a strict local maximum at $(x_{0},t_{0})$, there exist small numbers $r, \sigma > 0$, and the following hold, where $V_{r}(x_{0},t_{0}) := B_{r}(x_{0}) \times (t_{0} - r, t_{0} + r)$.
   \begin{eqnarray}\label{the assumption on the strict local maximum in proving sub solution inequality of the upper half relaxed limit}
   V_{r}(x_{0},t_{0}) \subseteq \RR^{n} \times (0, \infty), && \max_{\partial_{p} V_{r}}\left( u^{\star} - \phi\right) < \max_{V_{r}}\left( u^{\star} - \phi\right),
   \end{eqnarray}
   \begin{eqnarray*}
   \min_{(x,t) \in V_{r}}\phi_{t}(x,t) > \sigma, && \sup_{(x,t) \in V_{r}}|D\phi(x,t)| < \frac{\sigma}{2M_0}.
   \end{eqnarray*}
   Since $D^{2}\phi(x,t)$ is bounded on $V_{r}(x_{0},t_{0})$, so is $\tr \left\lbrace  D^{2}\phi\left( I - \nu\otimes \nu\right) \right\rbrace$ with $\nu \in \mathbb{S}^{n-1}$, therefore, if $\varepsilon$ is small, we have that
   \begin{equation*}
   \sup_{\nu \in \mathbb{S}^{n-1}}\varepsilon \tr \left\lbrace  D^{2}\phi\left( I - \nu\otimes \nu\right) \right\rbrace < \frac{\sigma}{2}.
   \end{equation*}
   Hence
   \begin{eqnarray*}
   \mathscr{F}^{*}\left( \varepsilon D^{2}\phi(x,t), D\phi(x,t), \frac{x}{\varepsilon}\right) < \sigma < \phi_{t}(x,t), && (x,t) \in V_{r}(x_{0},t_{0}),
   \end{eqnarray*}
   which means that $\phi(x,t)$ is a (classical) supersolution of (\ref{the scaled forced mean curvature flow}), then the Proposition \ref{the usual comparison principle} indicates that
   \begin{equation*}
   \max_{V_{r}}\left( u^{\varepsilon}(x,t) - \phi(x,t)\right) \leq \max_{\partial_{p}V_{r}}\left( u^{\varepsilon}(x,t) - \phi(x,t)\right).
   \end{equation*}
   Let us apply the upper half relaxed limit operator (Definition \ref{half relaxed limits}) on both sides and derive that
  \begin{equation*}
  \max_{V_{r}}\left( u^{\star}(x,t) - \phi(x,t)\right) \leq \max_{\partial_{p}V_{r}}\left( u^{\star}(x,t) - \phi(x,t)\right), 
  \end{equation*}
  which contradicts (\ref{the assumption on the strict local maximum in proving sub solution inequality of the upper half relaxed limit}).
\end{proof}

\begin{proposition}\label{the super equation satisfied by the limit}
	Let $\nu \in \mathbb{S}^{n-1}$ and $\underline{s}(\nu)$ be the tail speed in the $\nu$ direction. Let $\psi(x,t)$ be a $C^{2,1}$ function, assume $u_{\star}(x,t) - \psi(x,t)$ obtains a strict local minimum at $(x_{0},t_{0}) \in \RR^{n} \times (0, \infty)$, denote $q_{0} := D\psi(x_{0},t_{0})$, then
	\begin{eqnarray*}
		\begin{cases}
			\psi_{t}(x_{0},t_{0}) \geq \underline{s}(\nu)|D\psi(x_{0},t_{0})|, & q_{0} \neq 0, \hspace{2mm} \nu := - \frac{q_{0}}{|q_{0}|},\\
			\psi_{t}(x_{0},t_{0}) \geq 0, & q_{0} = 0.
		\end{cases}
	\end{eqnarray*}
\end{proposition}

\begin{proof}
	It is similar to that of Proposition \ref{the sub equation satisfied by the limit}, we omit it here.
\end{proof}

\begin{definition}
Consider the equation \eqref{the homogenized equation} as follows, where
\begin{equation}\label{the homogenized equation}
\begin{cases}
u_{t} = s\left( -\widehat{Du}\right) \left| Du\right|, & (x,t) \in \RR^{n} \times (0,\infty),\\
u(x,0) = u_{0}(x), & x \in \RR^{n}.
\end{cases} \tag{$\overline{\text{E}}$}
\end{equation}

\begin{enumerate}
	\item [(a)] Let $s(\cdot) = \overline{s}(\cdot)$, an upper semicontinuous function $u(x,t): \RR^{n} \times (0,\infty) \rightarrow \RR$ is called a viscosity subsolution of (\ref{the homogenized equation}), if the following hold.
	\begin{enumerate}
	\item [(i)] Let $\phi(x,t)$ be a $C^{2,1}$ function, assume $u(x,t) - \phi(x,t)$ obtains a local maximum at $(x_{0},t_{0}) \in \RR^{n} \times (0, \infty)$, denote $q_{0} := D\phi(x_{0},t_{0})$, then
\begin{eqnarray*}
	\begin{cases}
		\phi_{t}(x_{0},t_{0}) \leq \overline{s}(\nu)|D\phi(x_{0},t_{0})|, & q_{0} \neq 0, \hspace{2mm} \nu := - \frac{q_{0}}{|q_{0}|},\\
		\phi_{t}(x_{0},t_{0}) \leq 0, & q_{0} = 0.
	\end{cases}
\end{eqnarray*}
	\item [(ii)] $u(x,0) \leq u_{0}(x)$, $x \in \RR^{n}$.
	\end{enumerate}
	\item [(b)] Let $s(\cdot) = \underline{s}$, a lower semicontinuous function $v(x,t): \RR^{n} \times (0,\infty) \rightarrow \RR$ is called a viscosity supersolution of (\ref{the homogenized equation}), if the following hold.
	\begin{enumerate}
	\item [(i)] Let $\psi(x,t)$ be a $C^{2,1}$ function, assume $v(x,t) - \psi(x,t)$ obtains a local minimum at $(x_{0},t_{0}) \in \RR^{n} \times (0, \infty)$, denote $q_{0} := D\psi(x_{0},t_{0})$, then
	\begin{eqnarray*}
	\begin{cases}
	\psi_{t}(x_{0},t_{0}) \geq s(\nu)|D\psi(x_{0},t_{0})|, & q_{0} \neq 0, \hspace{2mm} \nu := - \frac{q_{0}}{|q_{0}|},\\
	\psi_{t}(x_{0},t_{0}) \geq 0, & q_{0} = 0.
	\end{cases}
	\end{eqnarray*}
	\item [(ii)] $v(x,0) \geq u_{0}(x)$, $x \in \RR^{n}$.
	\end{enumerate}
	\item[(c)] If $s(\cdot) = \overline{s}(\cdot) = \underline{s}(\cdot)$, then a continuous function $w(x,t)$ is called a viscosity solution of (\ref{the homogenized equation}) if $w(x,0) = u_{0}(x)$, and that $w$ is both a viscosity subsolution and a viscosity supersolution of (\ref{the homogenized equation}).
	\end{enumerate}
\end{definition}

\begin{proposition}\label{homogenization occurs when the head speed and tail speed coincide}
	If $\overline{s}(\cdot) \equiv \underline{s}(\cdot)$, we denote it by $s(\cdot)$. Let $u^{\varepsilon}(x,t)$ be the unique viscosity solution of (\ref{the scaled forced mean curvature flow}), then $u^{\varepsilon}(x,t)$ converges locally uniformly, as $\varepsilon \rightarrow 0$, to a continuous function $\overline{u}(x,t)$ in $\RR^{n} \times (0, \infty)$, which is the unique viscosity solution of (\ref{the homogenized equation}).
\end{proposition}

\begin{proof}
	The uniqueness, if $u(x,t)$ is a solution of (\ref{the homogenized equation}), then $w(x,t) := e^{-t}u(x,t)$ is a solution of the following equation, which has a unique solution.
	\begin{equation*}
	\begin{cases}
	w_{t} + w = s\left( - \frac{Dw}{|Dw|}\right)|Dw|, & (x,t) \in \RR^{n} \times (0, \infty),\\
	w(x,0) = u_{0}(x), & x \in \RR^{n}.
	\end{cases}
	\end{equation*}
    Therefore, the equation (\ref{the homogenized equation}) has a unique solution. On the other hand, by Proposition \ref{the sub equation satisfied by the limit} and the Proposition \ref{the super equation satisfied by the limit}, we have that $u^{\star}(x,t) \leq u_{\star}(x,t)$, $(x,t) \in \RR^{n} \times (0,\infty)$. Clearly, by Definition \ref{half relaxed limits}, we have $u_{\star}(x,t) \leq u^{\star}(x,t)$. Therefore, $u_{\star}(x,t) = u^{\star}(x,t)$, let us denote it by $\overline{u}(x,t)$. By Definition \ref{half relaxed limits} again, we have that
    \begin{equation*}
    \lim_{\varepsilon\rightarrow 0}u^{\varepsilon}(x,t) = \overline{u}(x,t) \hbox{ locally uniformly in } \RR^{n} \times (0,\infty).
    \end{equation*}

\end{proof}

\section{Nonhomogenization}\label{the section of nonhomogenization}

In this section, we study the case that the head speed is not identically equal to the tail speed. i.e., there exists $\nu_{0} \in \mathbb{S}^{n-1}$, with $\underline{s}(\nu_{0}) < \overline{s}(\nu_{0})$. It turns out that in this case we can find ``long fingers", growing linealy in time in the $\nu_{0}$ direction, in certain level set of the real solution.

\subsection{An ordering relation}
\begin{definition}
	For any $q \in \RR^{n}$, let $u := u(x,t;q)$ and $u^{\varepsilon} := u^{\varepsilon}(x, t; q)$ be the unique solution of the following equation (\ref{the forced mean curvature flow with linear initial data}) and (\ref{the forced mean curvature flow with linear initial data in epsilon scale}) (c.f. (\ref{the differential operator regarding space derivatives})), respectively.
   \begin{equation}\label{the forced mean curvature flow with linear initial data}
   \begin{cases}
   u_{t} = \mathscr{F}\left( D^{2}u, Du, x\right), & (x,t) \in \RR^{n} \times (0,\infty),\\
   u(x,0;q) := q\cdot x, & x \in \RR^{n}.
   \end{cases}
   \end{equation}
	\begin{equation}\label{the forced mean curvature flow with linear initial data in epsilon scale}
	\begin{cases}
	u^{\varepsilon}_{t} = \mathscr{F}\left( \varepsilon D^{2}u^{\varepsilon}, Du^{\varepsilon}, \frac{x}{\varepsilon}\right), & (x,t) \in \RR^{n} \times (0,\infty),\\
	u^{\varepsilon}(x,0;q) := q\cdot x, & x \in \RR^{n}.
	\end{cases}
	\end{equation}
\end{definition}

\begin{lemma}\label{comparison between the real solution and the obstacle solutions}
	Let $a := (\nu, R, 0, q, \overline{s}(\nu) + \sigma) \in \mathbb{A}$ (resp. $a := (\nu, R, 0, q, \underline{s}(\nu) - \sigma) \in \mathbb{A}$) with $\sigma > 0$, then there exists $A = A(\nu,\sigma) > 0$, with $\xi_{A} \in \argmin\limits\limits_{\xi \in \ZZ^{n}, \hspace{1mm} \xi \cdot \nu \geq A}|\xi|$, we have for any $\mu$ and $t \geq 0$ that
	\begin{eqnarray*}
	u(\cdot, t, q) \prec_{\left( \Upomega(0, R; \nu), \mu\right) } \overline{\mathrm{U}}_{a}(\cdot - \xi_{A},t) && (\text{resp. } \underline{\mathrm{U}}_{a}(\cdot + \xi_{A},t) \prec_{(\Upomega(0,R;\nu), \mu)} u(\cdot, t, q)).
	\end{eqnarray*}
\end{lemma}

\begin{proof}
	Because sub-strict-detachment (c.f. Definition \ref{sub or super strict detached speeds}) implies uniform detachment, without loss of generality, we can take $\mu = 0$. By Lemma \ref{the expansion of detachment in a static domain for the obstacle subsolution}, for any $r > 0$, there exists $R := R(\nu, \sigma) > 0$ and $T = T(\nu, \sigma) > 0$, such that for any $t > T$, we have
	\begin{eqnarray*}
	\overline{\mathrm{U}}_{a}(\cdot,t) \prec_{(\Upomega(0,r;\nu), 0)} \mathrm{O}_{e}(\cdot, t), \hspace{2mm} t > T, &\text{with}& e := (\nu, q, \overline{s}(\nu) + \sigma) \in \mathbb{E}.
	\end{eqnarray*}
    Moreover, we have that
	\begin{eqnarray}\label{the real solution is below a shift of obstacle during an initial period of time}
	u(\cdot, t; q) \prec_{(\Upomega(0, r; \nu), 0)} \mathrm{O}_{(\nu, q, m_0 + M_0)}(\cdot, t), && 0 \leq t \leq T + 1.
	\end{eqnarray}
    Since the above $T$ is independent of $r$, we have $r \rightarrow \infty$ if we send $R \rightarrow \infty$. Let us set $A := \left(m_0 + M_0\right)  \left( T(\nu,\sigma) + 1\right) + 1$ and denote $\mathrm{U}_{\infty}$ as follows, which is a supersolution as $\overline{\mathrm{U}}_{a}(x,t)$ is a solution in $\Upomega(0,r;\nu)$: 
    \begin{eqnarray*}
    \mathrm{U}_{\infty}(x,t) := \lim_{R\rightarrow \infty} \overline{\mathrm{U}}_{a}(x - \xi_{A},t) = \inf_{R > 0} \overline{\mathrm{U}}_{a}(x - \xi_{A},t).
    \end{eqnarray*}
    By the above choice of $T$ and the comparison principle, we get that
    \begin{eqnarray*}
    u(\cdot, t; q) \prec_{(\Upomega(0,\infty; \nu),0)} \mathrm{U}_{\infty}(\cdot, t), && t \geq T.
    \end{eqnarray*}
    Recalling (\ref{the real solution is below a shift of obstacle during an initial period of time}), we conclude that
    \begin{eqnarray*}
    	u(\cdot, t; q) \prec_{(\Upomega(0,\infty; \nu),0)} \mathrm{U}_{\infty}(\cdot, t), && 0 \leq t < \infty.
    \end{eqnarray*}
    Then the desired result is valid due to the following inequality.
    \begin{eqnarray*}
    \mathrm{U}_{\infty}(x, t) \leq \overline{\mathrm{U}}_{a}(x - \xi_{A}, t), && (x,t) \in \Upomega(0, R; \nu) \times [0,\infty).
    \end{eqnarray*}

\end{proof}

\subsection{A closeness property}
If the obstacle speed is below (resp. above) the head (resp. tail) speed, it is necessary to describe the closeness of the obstacle subsolution (resp. supersolution) to the associated obstacle function.

\subsubsection{Irrational directions}

If $\nu \in \mathbb{S}^{n-1} \diagdown \RR\ZZ^{n}$, the Proposition \ref{the detachment lemma subsolution} (resp. the Proposition \ref{the detachment lemma supersolution}) shows that the detachment is equivalent to sub-strict-detached (super-strict-detached) obstacle speed. Therefore, if the obstacle speed is strictly smaller (resp. larger) than the head speed (resp. tail speed), the obstacle subsolution (resp. supersolution) touches the obstacle very frequently. In addition, the Birkhoff properties indicate repeated pattern of this kind of touching.

\begin{lemma}\label{the meaning of no detachment}
	Let $a := (\nu, R, 0, q, s) \in \mathbb{A}$ and set $e := (\nu, q, s) \in \mathbb{E}$, where $\nu \in \mathbb{S}^{n-1}\diagdown \RR\ZZ^{n}$ and $s := \overline{s}(\nu) - \sigma$ (resp. $s := \underline{s} + \sigma$) with $\sigma > 0$. Then for any $R > r > \frac{\sqrt{n}}{2}$, $\mu \in \RR$ and $T > 0$, there exists $(x,t) \in \Upomega(0, r; \nu) \times (T, \infty)$, such that
	\begin{eqnarray*}
	\overline{\mathrm{U}}_{a}(x, t) = \mathrm{O}_{e}(x, t) = \mu && \left( \text{resp.} \hspace{2mm} \underline{\mathrm{U}}_{a}(x, t) = \mathrm{O}_{e}(x, t) = \mu\right). 
	\end{eqnarray*}
\end{lemma}

\begin{proof}
	Since $s < \overline{s}(\nu)$ (resp. $s > \underline{s}(\nu)$), the Lemma is the negation of detachment.
\end{proof}

\begin{proposition}\label{the periodic distribution of touching points in obstacle}
	Let $a := (\nu, R, 0, q, s) \in \mathbb{A}$, where $\nu \in \mathbb{S}^{n-1}\diagdown \RR\ZZ^{n}$, $R > \frac{\sqrt{n}}{2}$, $0 < s < \overline{s}(\nu)$ (resp. $s > \underline{s}(\nu)$) and set $e = (\nu, q, s) \in \mathbb{E}$. Then for any $\mu\in \RR$ and $h > 0$, there exists $\xi \in [0,1)^{n}$, such that $\overline{\mathrm{U}}_{a}(x_{1},t_{1}) = \mathrm{O}_{e}(x_{1},t_{1}) = \mu$ (resp. $\underline{\mathrm{U}}_{a}(x_{1},t_{1}) = \mathrm{O}_{e}(x_{1},t_{1}) = \mu$) at any $(x_{1},t_{1})$ satisfying
    \begin{eqnarray*}
    x_{1} \in \Upomega(0,R; \nu) \cap \left( \xi + \ZZ^{n}\right), \hspace{2mm} x_{1} \cdot \nu + \frac{\mu}{|q|} \in \left[ 0, h\right] &\text{and}& t_{1} = \frac{x_{1} \cdot \nu}{s}.
    \end{eqnarray*}
\end{proposition}

\begin{proof}
	Because $\nu \in \mathbb{S}^{n-1}\diagdown \RR\ZZ^{n}$ and $0 < s < \overline{s}(\nu)$, the detachment does not happen at the $\mu$ level set. Then by Lemma \ref{the meaning of no detachment}, with the above $R$ and $T := \frac{h}{s}$, there exist $x_{0} \in \Upomega(0, R; \nu)$ and $t_{0} > T$, such that (here $b := (\nu, 3R, 0, q, s) \in \mathbb{A}$)
	\begin{equation*}
	\overline{\mathrm{U}}_{b}(x_{0}, t_{0}) = \mathrm{O}_{e}(x_{0}, t_{0}) = \mu.
	\end{equation*}
    Let us take $\xi \in [0,1)^{n}$ with $x_{0} - \xi \in \ZZ^{n}$. Consider any above $(x_{1}, t_{1})$, then set $\Delta z_{1}$ and $\Delta t_{1}$ as follows. It suffices to prove that $\overline{\mathrm{U}}_{a}(x_{1},t_{1}) = \mathrm{O}_{e}(x_{1},t_{1}) = \mu$. 
    \begin{eqnarray*}
    \Delta z_{1} := x_{0} - x_{1} &\text{and}& \Delta t_{1} := \frac{\Delta z_{1} \cdot \nu}{s}.
    \end{eqnarray*}
   Because $\mathrm{U}(x,t) := \overline{\mathrm{U}}_{b}(x + \Delta z_{1}, t + \Delta t_{1})$, restricted to $\Upomega(0, R; \nu)$, is a subsolution bounded from above by $\mathrm{O}_{e}(x,t)$, the maximality of $\overline{\mathrm{U}}_{a}(x,t)$ implies that
   \begin{eqnarray*}
   \mathrm{U}(x,t) \leq \overline{\mathrm{U}}_{a}(x,t) &\text{with}& x \in \Upomega(0, R; \nu), \hspace{2mm} t \geq 0.
   \end{eqnarray*}
  Finally, the result from the inequality
  \begin{equation*}
  \mu = \overline{\mathrm{U}}_{b}(x_{0},t_{0}) = \mathrm{U}(x_{1},t_{1}) \leq \overline{\mathrm{U}}_{a}(x_{1},t_{1}) \leq \mathrm{O}_{e}(x_{1},t_{1}) = \mathrm{O}_{e}(x_{0},t_{0}) = \mu.
  \end{equation*}
\end{proof}

\begin{proposition}\label{the repeated patten regarding the head and tail of the oscillation of the real solution in a thin strip}
	Fix any $\mu \in \RR$ and assume (i)-(iii) as follows.
	\begin{enumerate}
		\item [(i)] $\nu \in \mathbb{S}^{n-1}\diagdown \RR\ZZ^{n}$ and $q = - |q|\nu \in \RR^{n}\diagdown \left\lbrace 0\right\rbrace $;
		\item [(ii)] $0 < \sigma < \min\left\lbrace \overline{s}(\nu) - \underline{s}(\nu), \underline{s}(\nu)\right\rbrace$;
		\item [(iii)] $u(x,t;q)$ is the unique solution of (\ref{the forced mean curvature flow with linear initial data}).
	\end{enumerate} 
    Then there exists $C := C(\nu, \sigma) > 0$, such that for any $x_{0} \in \RR^{n}$ and $r > \frac{\sqrt{n}}{2}$, we have (a) and (b) as follows:
    \begin{enumerate}
    	\item [(a)] There is a sequence of numbers $\left\lbrace t_{k}\right\rbrace_{k \geq 1}$ (resp. $\left\lbrace \tau_{k}\right\rbrace_{k \geq 1}$), such that
    	\begin{eqnarray*}
    	\lim_{k \rightarrow \infty}t_{k} = \infty &\text{and}& 0 < t_{k+1} - t_{k} \leq \frac{1}{\overline{s}(\nu) - \sigma}\\
    	(\text{resp.} \hspace{2mm} \lim_{k \rightarrow \infty}\tau_{k} = \infty &\text{and}& 0 < \tau_{k+1} - \tau_{k} \leq \frac{1}{\underline{s}(\nu) + \sigma}) ,
    	\end{eqnarray*}
    	\item [(b)] For each $k$, there exists $x_{k} \hspace{1mm} (\text{resp. } y_{k})\in \Upomega(x_{0}, r; \nu)$, such that
    	\begin{eqnarray*}
    	u(x_{k}, t_{k}; q) = \mu &\text{and}& x_{k}\cdot \nu > - \frac{\mu}{|q|} + \left( \overline{s}(\nu) - \sigma \right) t_{k} - C(\nu, \sigma)\\
    	(\text{resp.} \hspace{2mm} u(y_{k}, \tau_{k}; q) = \mu &\text{and}& y_{k}\cdot \nu < - \frac{\mu}{|q|} + \left( \underline{s}(\nu) + \sigma \right) \tau_{k} + C(\nu, \sigma) ).
    	\end{eqnarray*}
    \end{enumerate}
\end{proposition}

\begin{proof}
	We only consider $(x_{0}, \mu) = (0,0)$, the case of general $(x_{0},\mu) \in \RR^{n} \times \RR$ can be argued similarly. Let us set $s_{1} := \underline{s}(\nu) - \sigma$ and $s_{2} := \overline{s}(\nu) - \sigma$. It suffices to prove a finite time version of the statement. i.e., for any $T > 0$, there exist $0 < t_{T,1} < t_{T,2} < \cdots < t_{T,k} < t_{T,k+1} < \cdots < T$, such that $0 < t_{T, k+1} - t_{T,k} \leq \frac{1}{\overline{s}(\nu) - \sigma}$ and (b) holds. Then we take $\left\lbrace t_{k}\right\rbrace_{k=1}^{\infty} := \cup_{\ell = 1}^{\infty}\left\lbrace t_{\ell,i}\right\rbrace$. By Proposition \ref{LCP in the unit scale}, there exist $R$, $\mathscr{R} > 0$, such that for $0 \leq t \leq T$, we have
    \begin{eqnarray*}
    \overline{\mathrm{U}}_{a_{2}}(\cdot + \xi_{0}, t)\prec_{(\Upomega(0,r;\nu), 0)}  \underline{\mathrm{U}}_{a_{1}}(\cdot,t), &\text{where}& a_{i} := (\nu, R, \mathscr{R}, q, s_{i}), \hspace{2mm} \xi_{0} \in \argmin\limits\limits_{\xi \in \ZZ^{n}, \hspace{1mm} \xi \cdot \geq 1} |\xi|.
    \end{eqnarray*}
    Let us take $\hat{R} := R + \mathscr{R}T$ and set $b := (\nu, \hat{R}, 0, q, s_{1})$, then
    \begin{eqnarray*}
    \underline{\mathrm{U}}_{a_{1}}(x,t) \leq \underline{\mathrm{U}}_{b}(x,t), &\text{for}& x \in \Upomega(0, r; \nu) \hspace{2mm} \text{and} \hspace{2mm} 0 \leq t \leq T.
    \end{eqnarray*} 
    According to Lemma \ref{comparison between the real solution and the obstacle solutions}, there exists $A := A(\nu,\sigma) > 0$, such that
    \begin{eqnarray*}
    \underline{\mathrm{U}}_{b}(\cdot + \xi_{A},t) \prec_{(\Upomega(0, r; \nu), 0)} u(\cdot, t; q), &\text{where}& t \geq 0 \hspace{2mm}, \hspace{2mm} \xi_{A} \in \argmin\limits\limits_{\xi \in \ZZ^{n}, \hspace{1mm} \xi \cdot \nu \geq A} |\xi|.
    \end{eqnarray*}
    Finally, we have for $0 \leq t \leq T$ that
    \begin{eqnarray}\label{the ordering relation between obstacle subsolution and the real solution}
    \overline{\mathrm{U}}_{a_{2}}(\cdot + \xi_{0} + \xi_{A}, t) \prec_{(\Upomega(0,r;\nu), 0)} u(\cdot,t;q).
    \end{eqnarray}
    Set $C(\nu,\sigma) := |\xi_{0} + \xi_{A}|$ and the Proposition \ref{the periodic distribution of touching points in obstacle} indicates the existence of $\left\lbrace x_{k}\right\rbrace_{k \geq 1}$ is a set of points in $\Upomega(0, r; \nu)$ that are relative integers to each other. Since $r > \frac{\sqrt{n}}{2}$, we can choose $\left\lbrace x_{k}\right\rbrace _{k \geq 1}$ such that $|(x_{k+1} - x_{k})\cdot\nu| \leq 1$. Then set $t_{T,k} := \frac{|x_{k}\cdot\nu|}{\overline{s}(\nu) - \sigma}$ and so $0 < t_{T,k + 1} - t_{T,k} \leq \frac{1}{\overline{s}(\nu) - \sigma}$. The statement of (b) follows from (\ref{the ordering relation between obstacle subsolution and the real solution}).\\
    The other ordering relation (the `resp.') can be proved similarly.
\end{proof}

\begin{proposition}\label{the long fingure occurs when head tail speed do not coincide in irrational direction}
	Let $\nu \in \mathbb{S}^{n-1}\diagdown \RR\ZZ^{n}$ and $u(x,t;q)$ be the unique solution of (\ref{the forced mean curvature flow with linear initial data}), where $q = - |q| \nu\in \RR^{n}\diagdown \left\lbrace 0\right\rbrace$. Assume $\overline{s}(\nu) > \underline{s}(\nu)$, then for any $0 < \sigma \leq \overline{s}(\nu) - \underline{s}(\nu)$, there exist constant $K := K(\nu, \sigma)$, such that the following statement holds: for any $(z_{0},\mu, r, t) \in \RR^{n} \times \RR \times \left( \frac{\sqrt{n}}{2}, \infty\right) \times \left(\frac{1}{m_0}, \infty \right) $, there exist  $x, y \in \Upomega(z_{0}, r; \nu)$, such that
	\begin{eqnarray*}
	u(x, t;q) = u(y, t;q) = \mu &\text{and}& \begin{cases}
	 x \cdot \nu > \left( \overline{s}(\nu) - \sigma\right) t - \frac{\mu}{|q|} - K,\\
	y \cdot \nu < \left( \underline{s}(\nu) + \sigma\right) t - \frac{\mu}{|q|} + K.
	\end{cases}
	\end{eqnarray*} 
\end{proposition}

\begin{proof}
	Let us consider $(\mu, z_{0}) = (0, 0)$ and the general $(\mu, z_{0})$ can be argued similarly. By Proposition \ref{the repeated patten regarding the head and tail of the oscillation of the real solution in a thin strip}, there exist $C_{1} = C_{1}(\nu,\sigma) > 0$, $t_{i} > 0$ (with $0 < t_{i + 1} - t_{i}\leq \frac{1}{\overline{s}(\nu) - \sigma}$) and $\hat{x}_{i} \in \Upomega(0,r; \nu)$, such that
	\begin{eqnarray*}
	u(\hat{x}_{i},t_{i};q) = 0 &\text{and}& \hat{x}_{i}\cdot\nu > \left(\overline{s}(\nu) - \sigma \right) t_{i} - C_{1}(\nu,\sigma).
	\end{eqnarray*}
    Similarly, there exist $C_{2} := C_{2}(\nu,\sigma) > 0$, $\tau_{j} > 0$ (with $0 < \tau_{j + 1} - \tau_{j} \leq \frac{1}{\underline{s}(\nu) + \sigma}$) and $\hat{y}_{j} \in \Upomega(0, r; \nu)$, such that
    \begin{eqnarray*}
    u(\hat{y}_{j}, \tau_{j}; q) = 0 &\text{and}& \hat{y}_{j}\cdot\nu < \left( \underline{s}(\nu) + \sigma\right)\tau_{j} + C_{2}(\nu,\sigma). 
    \end{eqnarray*}
    Because $u(x,t;q)$ is increasing in time (c.f. Proposition 5.1 \cite{Caffarelli and Monneau ARMA}). Then for any $t > \frac{1}{m_0}$, there exist $x, \hat{x}_{i} \in \Upomega(0, r; \nu)$, $0 \leq t - t_{i} \leq \frac{1}{\overline{s}(\nu) - \sigma}$ with $u(x, t; q) = 0$ and
    \begin{eqnarray*}
    x \cdot \nu \geq \hat{x}_{i}\cdot \nu > \left(\overline{s}(\nu) - \sigma \right) t_{i} - \frac{\mu}{|q|} - C_{1}(\nu,\sigma) \geq \left(\overline{s}(\nu) - \sigma \right) t - \frac{\mu}{|q|} - C_{1}(\nu,\sigma) - 1.
    \end{eqnarray*}
    Similarly, there exist $y, \hat{y}_{j} \in \Upomega(0, r; \nu)$, $0 \leq \tau_{j} - t \leq \frac{1}{\underline{s}(\nu) + \sigma}$ with $u(y, t; q) = 0$ and
    \begin{eqnarray*}
    y\cdot\nu \leq \hat{y}_{j} \cdot \nu < \left( \underline{s}(\nu) + \sigma\right)\tau_{j} - \frac{\mu}{|q|} + C_{2}(\nu,\sigma) \leq \left( \underline{s}(\nu) + \sigma\right)t - \frac{\mu}{|q|} + C_{2}(\nu,\sigma) + 1.
    \end{eqnarray*}
    Thus the desired result follows once we set $K(\nu, \sigma) := C_{1}(\nu, \sigma) + C_{2}(\nu, \sigma) + 2$.
\end{proof}

\subsubsection{Rational directions}

\begin{proposition}\label{long fingure occurs when head tail speed do not coincide in rational directions}
	Let $\vartheta \in \mathbb{S}^{n-1} \cap \RR\ZZ^{n}$, $q = - |q|\vartheta \in \RR^{n}\diagdown \left\lbrace 0\right\rbrace$ and $u(x, t; q)$ be the unique solution of (\ref{the forced mean curvature flow with linear initial data}). Assume $\overline{s}(\nu) > \underline{s}(\nu)$ and fix $0 < \sigma \ll \overline{s}(\vartheta) - \underline{s}(\vartheta)$, then there exists $r = r(\vartheta) > \sqrt{n}$, such that for any $\mu \in \RR$ and $z_{0} \in \RR^{n}$, there exist $x, y \in \Upomega(z_{0}, r; \vartheta)$, such that $u(x, t; q) = u(y, t; q) = \mu$ and
	\begin{eqnarray*}
	u(x, t; q) = u(y, t; q) = \mu &\text{and}& \begin{cases}
		x \cdot \vartheta > \left( \overline{s}(\vartheta) - \sigma\right) t - \frac{\mu}{|q|} - \sqrt{n},\\
		y \cdot \vartheta < \left( \underline{s}(\vartheta) + \sigma\right) t - \frac{\mu}{|q|} + \sqrt{n}.
	\end{cases}
	\end{eqnarray*}
\end{proposition}

\begin{proof}
	Because $\vartheta \in \mathbb{S}^{n-1} \cap \RR\ZZ^{n}$, then there exists $r = r(\vartheta) > \sqrt{n}$, such that there exists $w_{0} \in \ZZ^{n}\diagdown \left\lbrace 0\right\rbrace $, with $w_{0} \cdot \nu = 0$ and $|w_{0}| \leq r(\vartheta) - \sqrt{n}$. By Proposition \ref{the equivalence between domain obstacle semi solutions and global semi solutions in defining head and tail speeds}, $\overline{s}(\vartheta) = \overline{s}^{\infty}$ and $\underline{s} = \underline{s}_{\infty}$. Recall Definition \ref{the global head speed in an rational direction} and denote 
	\begin{eqnarray*}
	a^{\infty} := (\vartheta, \infty, 0, q, \overline{s}^{\infty} - \sigma) &\text{and}& e^{\infty} := (\vartheta, q, \overline{s}^{\infty} - \sigma),\\
	a_{\infty} := (\vartheta, \infty, 0, q, \underline{s}_{\infty} + \sigma) &\text{and}& e_{\infty} := (\vartheta, q, \underline{s}_{\infty} + \sigma).
	\end{eqnarray*}
    Clearly, we have that
    \begin{eqnarray*}
    \overline{\mathrm{U}}_{a^{\infty}}(x,t) \leq u(x, t; q) \leq \underline{\mathrm{U}}_{a_{\infty}}(x,t), && (x,t) \in \RR^{n} \times [0,\infty).
    \end{eqnarray*}
   By the choice of $w_{0}$ and the uniqueness of both $\overline{\mathrm{U}}_{a^{\infty}}(x,t)$ and $\underline{\mathrm{U}}_{a_{\infty}}(x,t)$, then
   \begin{eqnarray*}
   \overline{\mathrm{U}}_{a^{\infty}}(x,t) = \overline{\mathrm{U}}_{a^{\infty}}(x + w_{0},t) &\text{and}& \underline{\mathrm{U}}_{a_{\infty}}(x,t) = \underline{\mathrm{U}}_{a_{\infty}}(x + w_{0},t).
   \end{eqnarray*}
   Since the detachment does not happen to $\overline{\mathrm{U}}_{a^{\infty}}(x,t)$ or $\underline{\mathrm{U}}_{a_{\infty}}(x,t)$, by Lemma \ref{a simpler description of detachment for subsolution}, there are sequences $\left\lbrace \left( t_{i}, x_{i}\right) \right\rbrace_{i = 1}^{\infty}$ and $\left\lbrace \left( \tau_{j}, y_{j}\right) \right\rbrace_{j = 1}^{\infty}$, such that 
   \begin{eqnarray*}
   \lim_{i \rightarrow \infty} t_{i} = \infty, \hspace{2mm} 0 < t_{i + 1} - t_{i} \leq \frac{\sqrt{n}}{\overline{s}(\vartheta) - \sigma}, \hspace{2mm} |x_{i} - (x_{i} \cdot \vartheta) \vartheta| < r(\vartheta) \hspace{2mm}\text{and} \hspace{2mm} \overline{\mathrm{U}}_{a^{\infty}}(x_{i}, t_{i}) = \mu,
   \end{eqnarray*}
    \begin{eqnarray*}
	\lim_{i \rightarrow \infty} \tau_{j} = \infty, \hspace{2mm} 0 < \tau_{j + 1} - \tau_{j} \leq \frac{\sqrt{n}}{\underline{s}(\vartheta) + \sigma}, \hspace{2mm} |y_{j} - (y_{j} \cdot \vartheta) \vartheta| < r(\vartheta) \hspace{2mm}\text{and} \hspace{2mm} \underline{\mathrm{U}}_{a_{\infty}}(y_{j}, \tau_{j}) = \mu.
    \end{eqnarray*}
    Because $u(x,t;q)$ is increasing in time (c.f. Proposition 5.1 \cite{Caffarelli and Monneau ARMA}). Then for any $t > 0$, there exist $x, x_{i} \in \Upomega(0, r; \vartheta)$, $0 \leq t - t_{i} \leq \frac{\sqrt{n}}{\overline{s}(\vartheta) - \sigma}$ with $u(x, t; q) = \mu$ and
    \begin{eqnarray*}
    	x \cdot \vartheta \geq x_{i}\cdot \vartheta > \left(\overline{s}(\vartheta) - \sigma \right) t_{i} - \frac{\mu}{|q|} \geq \left(\overline{s}(\vartheta) - \sigma \right) t - \frac{\mu}{|q|} - \sqrt{n}.
    \end{eqnarray*}
    Similarly, there exist $y, y_{j} \in \Upomega(0, r; \vartheta)$, $0 \leq \tau_{j} - t \leq \frac{\sqrt{n}}{\underline{s}(\vartheta) + \sigma}$ with $u(y, t; q) = 0$ and
    \begin{eqnarray*}
    	y\cdot\vartheta \leq y_{j} \cdot \vartheta < \left( \underline{s}(\vartheta) + \sigma\right)\tau_{j} - \frac{\mu}{|q|} \leq \left( \underline{s}(\vartheta) + \sigma\right)t - \frac{\mu}{|q|} + \sqrt{n}.
    \end{eqnarray*}

\end{proof}

\subsubsection{The description of head/tail speed in macro and micro scales}

\begin{theorem}\label{the macro and micro scale description of head and tail speeds}
	For any $\nu \in \mathbb{S}^{n-1}$, $x_{0}, z_{0} \in \RR^{n}$ and $\mu \in \RR$, let $u(x, t)$ and $u^{\varepsilon}(x,t)$ be the unique solution of (\ref{the scaled forced mean curvature flow}), such that $u^{\varepsilon}(x, 0) = - (x - x_{0}) \cdot \nu$. Then there exists $r = r(\nu) > 0$, such that in micro-scale
    \begin{eqnarray*}
    	\overline{s}(\nu) = \lim_{t \rightarrow 0} \frac{\sup\left\lbrace x \cdot \nu \big| x \in \Upomega(z_{0}, r; \nu), u^{1}(x,t) = \mu\right\rbrace }{t}, && \underline{s}(\nu) = \lim_{t \rightarrow 0} \frac{\inf\left\lbrace x \cdot \nu \big| x \in \Upomega(z_{0}, r; \nu), u^{1}(x,t) = \mu\right\rbrace }{t},
    \end{eqnarray*}
    and in macro-scale (for $\varepsilon = 1$)
	\begin{eqnarray*}
	\limsup_{\varepsilon \rightarrow 0} u^{\varepsilon}(x,t) = - (x - x_{0})\cdot \nu + \overline{s}(\nu) t &\text{and}& \liminf_{\varepsilon \rightarrow 0} u^{\varepsilon}(x,t) = - (x - x_{0})\cdot \nu + \underline{s}(\nu) t.
	\end{eqnarray*}

\end{theorem}

\begin{proof}
	Based on comparison principle, we can consider without loss of generality that $x_{0} = 0$. From Lemma \ref{comparison between the real solution and the obstacle solutions}, Proposition \ref{the long fingure occurs when head tail speed do not coincide in irrational direction} and Proposition \ref{long fingure occurs when head tail speed do not coincide in rational directions}, there exists $r = r(\nu) > 0$, such that for any $z_{0} \in \RR^{n}$ and $0 < \sigma \ll 1$, there exists $K = K(\nu, \sigma)$, such that
	\begin{eqnarray*}
	\left| \sup\left\lbrace x \cdot \nu \big| x \in \Upomega(z_{0}, r; \nu), u^{1}(x, t) = \mu \right\rbrace - \left( \overline{s}(\nu) - \sigma\right)t + \mu\right| &\leq& K, \\
	\left| \inf\left\lbrace x \cdot \nu \big| x \in \Upomega(z_{0}, r; \nu), u^{1}(x, t) = \mu \right\rbrace - \left( \underline{s}(\nu) + \sigma\right)t + \mu\right| &\leq& K.
	\end{eqnarray*}
    Let us divide both side by $t$ and let $t \rightarrow \infty$, consider that $\sigma$ can be arbitrarily small, we get the result for the micro-scale case. Now, we consider the macro-scale case, the uniqueness of the solution for (\ref{the scaled forced mean curvature flow}) indicates that $u^{\varepsilon}(x,t) = \varepsilon u\left( \frac{x}{\varepsilon}, \frac{t}{\varepsilon}\right)$, then
    \begin{eqnarray*}
    && \left| \sup\left\lbrace x \cdot \nu \big| x \in \Upomega(z_{0}, \varepsilon r; \nu), u^{\varepsilon}(x, t) = \mu \right\rbrace - \left( \overline{s}(\nu) - \sigma\right)t + \mu\right|\\
    &=& \left| \sup\left\lbrace \varepsilon \left( \frac{x}{\varepsilon} \cdot \nu\right)  \Big| x \in \Upomega(z_{0}, \varepsilon r; \nu), u\left(\frac{\mu}{\varepsilon}, \frac{t}{\varepsilon}\right) = \frac{\mu}{\varepsilon} \right\rbrace - \varepsilon\left[ \left( \overline{s}(\nu) - \sigma\right)\frac{t}{\varepsilon} - \frac{\mu}{\varepsilon}\right] \right| \leq \varepsilon K 
    \end{eqnarray*}
    and
    \begin{eqnarray*}
     && \left| \inf\left\lbrace x \cdot \nu \big| x \in \Upomega(z_{0}, \varepsilon r; \nu), u^{\varepsilon}(x, t) = \mu \right\rbrace - \left( \underline{s}(\nu) + \sigma\right)t + \mu\right|\\
     &=& \left| \sup\left\lbrace \varepsilon \left( \frac{x}{\varepsilon} \cdot \nu\right)  \Big| x \in \Upomega(z_{0}, \varepsilon r; \nu), u\left(\frac{\mu}{\varepsilon}, \frac{t}{\varepsilon}\right) = \frac{\mu}{\varepsilon} \right\rbrace - \varepsilon\left[ \left( \underline{s}(\nu) + \sigma\right)\frac{t}{\varepsilon} - \frac{\mu}{\varepsilon}\right] \right| \leq \varepsilon K. 
     \end{eqnarray*}   
    Taking $\varepsilon \rightarrow 0$ and then send $\sigma \rightarrow 0$, we have that
    \begin{eqnarray*}
    \left\lbrace x \in \RR^{n} \big| \limsup_{\varepsilon \rightarrow 0} u^{\varepsilon}(x,t) = \mu\right\rbrace = \overline{s}(\nu) t - \mu &\text{and}& \left\lbrace x \in \RR^{n} \big| \limsup_{\varepsilon \rightarrow 0} u^{\varepsilon}(x,t) = \mu\right\rbrace = \underline{s}(\nu) t - \mu,
    \end{eqnarray*}
    which gives the result in macro-scale case.
\end{proof}

\begin{corollary}\label{the head and tail of front with general initial profile}
	Let $u^{\e}$ be the solution of \eqref{the scaled forced mean curvature flow} with $u^\e(x,0) = u_{0}(x)$, a uniformly continuous function in $\RR^{n}$. Let $\mathcal{A} := \left\lbrace (x_{i}, \nu_{i})\right\rbrace \subseteq \RR^{n} \times \mathbb{S}^{n-1}$ be a collection of points and directions, and define the associated convex sets
	\begin{eqnarray*}
	\mathcal{E}(t) := \inf_{(x_{i}, \nu_{i}) \in \mathcal{A}}\left\lbrace (x - x_{i})\cdot\nu_{i} \leq \overline{s}(\nu_{i})t\right\rbrace.
	\end{eqnarray*} 
    Then (c.f. Definition \ref{half relaxed limits}) if initially $\left\lbrace u_{0}(x) \geq 0 \right\rbrace \subseteq \mathcal{E}(0)$, then $\left\lbrace u^{\star}(\cdot, t) \geq 0 \right\rbrace \subseteq \mathcal{E}(t)$. 	
\end{corollary}

\begin{proof}
	It follows from the comparison principle and the Theorem \ref{the macro and micro scale description of head and tail speeds}.
\end{proof}

\section{Laminar forcing term}\label{the section of laminar forcing term}

In the laminar case, i.e., $g(x) = g(x^{\prime}, 0)$ with $x = (x^{\prime}, x_{n})$. Throughout this section, let us abuse the notation and denote the forcing term by $g(y)$ with $y = x^{\prime} \in \RR^{n-1}$, where $n \geq 3$. The zero level set is now a graph, i.e., $\left\lbrace x_{n} = u(y, t)\right\rbrace$, where $u(y,t)$, with initial graph $u_{0}(y)$, solves the equation as follows.
\begin{equation}\label{the fmcf equation in laminar case}
\begin{cases}
u_{t} = \sqrt{|Du|^{2} + 1} \DIV \left( \frac{Du}{\sqrt{|Du|^{2} + 1}}\right) + g\sqrt{|Du|^{2} + 1}, & (y,t) \in \RR^{d} \times (0, \infty),\\
u(y, 0) = u_{0}(y), & y \in \RR^{n-1}.
\end{cases}
\end{equation}

\subsection{Travelling wave sub and super solutions with head and tail speeds}
In this subsection, we assume that the homogenization associated to (\ref{the fmcf equation in laminar case}) fails, i.e., $\overline{s}(e_{n}) > \underline{s}(e_{n})$. Fix any $s \in [m_0, M_0]$ and denote $a^{\infty} = a_{\infty} := (e_{n}, \infty, 0, - e_{n}, s) \in \mathbb{A}$. Let us consider $\left\lbrace x \in \RR^{n} \big| \overline{\mathrm{U}}_{a^{\infty}}(x,t) = 0\right\rbrace$ (resp. $\left\lbrace x \in \RR^{n} \big| \underline{\mathrm{U}}_{a_{\infty}}(x,t) = 0\right\rbrace$), which is also a graph $\left\lbrace x\in \RR^{n} \big | x_{n} = \overline{U}^{s}(x^{\prime},t)\right\rbrace$ (resp. $\left\lbrace x\in \RR^{n} \big | x_{n} = \underline{U}_{s}(x^{\prime},t)\right\rbrace$). Clearly, $\overline{U}^{s}(y,t)$ (resp. $\underline{U}_{s}(y,t)$) is a subsolution (resp. supersolution) of (\ref{the fmcf equation in laminar case}) with $\overline{U}^{s}(y,0) \equiv 0$ (resp. $\underline{U}_{s}(y,0) \equiv 0$) and $\overline{U}^{s}(y, t) \leq st$ (resp. $\underline{U}_{s}(y, t) \geq st$). The uniqueness also implies that $\overline{U}^{s}(\cdot, t)$ (resp. $\underline{U}_{s}(\cdot, t)$) is $\ZZ^{n-1}$-periodic. In the following discussion, we shall denote without confusion that $\overline{s} = \overline{s}(e_{n})$ (resp. $\underline{s} = \underline{s}(e_{n})$).

\begin{definition}
	For any $s \in \left[ m_0, M_0\right]$, denote by $T^{*}(s)$ (resp. $T_{*}(s)$) the time after which $\overline{U}^{s}(y,t)$ (resp. $\underline{U}_{s}(y,t)$) detaches from the obstacle $\left\lbrace x \in \RR^{n} \big| x_{n} = st\right\rbrace $ totally by 1. i.e.,
	\begin{eqnarray*}
	T^{*}(s) &:=& \inf \left\lbrace t \geq 0 \big| \overline{U}^{s}(y, t) < st - 1, \hspace{2mm} y \in \RR^{n-1}\right\rbrace, \\
	T_{*}(s) &:=& \inf \left\lbrace t \geq 0 \big| \underline{U}_{s}(y, t) > st + 1, \hspace{2mm} y \in \RR^{n-1}\right\rbrace.
	\end{eqnarray*}
\end{definition}
\begin{lemma}
	$T^{*}(s) : (\overline{s}, M_0]$ is a right continuous function, such that $T^{*}(s) \geq \frac{1}{s - \overline{s}}$. Similarly, $T_{*}(s) : [m_0, \underline{s})$ is a left continuous function, such that $T_{*}(s) \geq \frac{1}{\underline{s} - s}$.
\end{lemma}

\begin{proof}
	\underline{Step 1.} By the definition of $\overline{s}(e_{n})$ and that $\overline{U}^{s}(\cdot, t)$ is $\ZZ^{n-1}$-periodic, the function $T^{*}(s)$ is well-defined. Fix any sequence $s_{\ell} \rightarrow s^{+} > \overline{s}$, then $\overline{U}^{s_{\ell}}(y,t) \geq \overline{U}^{s}(y,t)$, and therefore $\lim\limits_{\ell \rightarrow \infty}\inf_{*} \overline{U}^{s_{\ell}}(y,t) \geq \overline{U}^{s}(y,t)$. On the other hand, $\lim\limits_{\ell \rightarrow \infty}\sup^{*} \overline{U}^{s_{\ell}}(y, t)$, as a subsolution bounded from above by $x_{n} = st$, should be bounded by $\overline{U}^{s}(y,t)$. Therefore, $\lim\limits_{\ell\rightarrow \infty}\overline{U}^{s_{\ell}}(y,t) = \overline{U}^{s}(y,t)$ uniformly. Thus $\lim\limits_{\ell\rightarrow \infty} T(s_{\ell}) = T(s)$.\\
	\underline{Step 2.} Let us now choose $\hat{s}_{j} \rightarrow \overline{s}^{-}$, then $\overline{U}^{\overline{s}}(y,t) \geq \lim\limits_{j \rightarrow \infty}\sup^{*}\overline{U}^{\hat{s}_{j}}(y,t)$, therefore, $\max\limits_{y\in\RR^{n-1}}\left\lbrace \overline{U}^{\overline{s}}(y,t) - \overline{s}t\right\rbrace = 0$, for any $t \geq 0$. Then for the sequence $s_{\ell}$, we have that 
	\begin{equation*}
	\max_{y \in \RR^{n-1}}\left\lbrace \overline{U}^{s_{\ell}}(y,t) - s_{\ell}t\right\rbrace \leq \max_{y \in \RR^{n-1}}\left\lbrace \overline{U}^{\overline{s}}(y,t) - s_{\ell}t\right\rbrace \leq (s_{\ell} - \overline{s})t.
	\end{equation*}
    Hence, $T^{*}(s) \geq \frac{1}{s - \overline{s}}$.\\
    \underline{Step 3.} Similarly, we can prove the results associated to $T_{*}(s)$.
\end{proof}

\begin{proposition}[Proposition 4.4 of \cite{Cesaroni and Novaga}]\label{the regularity of front in laminar case}
	For all $t > 0$, assume the function $w(\cdot,t)$ satisfies in the viscosity sense that 
	\begin{eqnarray*}
	\lambda \leq -\DIV\left( \frac{Dw(x,t)}{\sqrt{|Dw(x,t)|^{2} + 1}}\right) \leq \Lambda, && t \in I \subseteq \RR,
	\end{eqnarray*}
    for two fixed numbers $\lambda, \Lambda$. Then $w(\cdot, t)$ are of class $C^{1,\alpha}$, for all $\alpha \in (0,1)$, uniformly in $t \in I$.
\end{proposition}

\begin{corollary}\label{the regularity of obstacle solutions}
	Fix any $\tau > 0$ and $s \in \left[m_0, M_0\right]$, then $\overline{U}^{s}(\cdot, t)$ and $\underline{U}_{s}(\cdot, t)$ are of $C^{1,\alpha}$ for all $\alpha \in (0,1)$, uniformly in $[\tau, \infty)$. 
\end{corollary}

\begin{proof}
	The standard comparison principle yields
	\begin{eqnarray*}
		\overline{U}^{s}(y, t) - s\Delta t \leq \overline{U}^{s}(y, t - \Delta t) &\text{and}& \overline{U}^{s}(y, t) + m_0\Delta t \leq \overline{U}^{s}(y, t + \Delta t),
	\end{eqnarray*}
	and thus
	\begin{eqnarray}\label{the boundedness of time derivative of the graph in laminar case}
	m_0 \leq \partial_{t}\overline{U}^{s}(y,t) \leq s \leq M_0, && (y,t) \in \RR^{n-1} \times (0, \infty).
	\end{eqnarray}
	Recall Proposition \ref{a semi obstacle solution becomes a solution when detachment happens}, $\overline{U}^{s}(x,t)$ is a solution of the equation \eqref{the fmcf equation in laminar case} away the obstacle $\overline{s}t$, therefore we have on this set that
	\begin{eqnarray*}
		m_0 - M_0 \leq -\DIV\left( \frac{D\overline{U}^{s}(x,t)}{\sqrt{|D\overline{U}^{s}(x,t)|^{2} + 1}}\right) = - \frac{\partial_{t}\overline{U}^{s}}{\sqrt{|D\overline{U}^{s}|^{2} + 1}} + g \leq 2M_0.
	\end{eqnarray*}	
	On the obstacle, since $\overline{U}^{s}(\cdot, t)$ is touched from above by a hyperplane, we have in viscosity sense that  
    \begin{eqnarray*}
	0 \leq -\DIV\left( \frac{D\overline{U}^{s}(x,t)}{\sqrt{|D\overline{U}^{s}(x,t)|^{2} + 1}}\right) \leq 2 M_0.
    \end{eqnarray*}
    Therefore, the Proposition {the regularity of front in laminar case} applies. Similarly, we also have the regularity for $\underline{U}_{s}(\cdot, t)$.
\end{proof}

\begin{theorem}\label{the existence of a travelling wave subsolution and its properties}
	If $\overline{s}(e_{n}) > \underline{s}(e_{n})$, then there is an open, nonempty set $E^{\infty} \subset \mathbb{T}^{n-1}$ and functions $U^{\infty}(y): E^{\infty} \rightarrow (-\infty, 0]$, such that the following are true:
	\begin{enumerate}
		\item [(a)] The function $U^{\infty}(y) + \overline{s}t$ is a viscosity subsolution of (\ref{the fmcf equation in laminar case});
 		\item [(b)] $U^{\infty}(y) \rightarrow -\infty$ as $y \rightarrow \partial E^{\infty}$;
 		\item [(c)] The function $U^{\infty}(y) + \overline{s}t$ is a solution of (\ref{the fmcf equation in laminar case}) away from $x_{n} = \overline{s}t$;
		\item [(d)] The set $\partial E^{\infty} \times (-\infty, \infty)$ is a stationary solution of (\ref{the scaled forced mean curvature flow}) with $\varepsilon = 1$.
	\end{enumerate}
\end{theorem}

\begin{proof}
	Fix $K > 0$, let $s_{\ell} := \overline{s} + \frac{1}{\ell^{2}}$ and $t_{\ell} := \ell$, then $T(s_{\ell}) \geq \ell^{2} \geq t_{\ell}$. Next, we define the following function, which is spatially $\ZZ^{n-1}$-periodic.
	\begin{eqnarray}
	\tilde{U}^{\ell}(y, t) := \overline{U}^{s_{\ell}}(y, t + t_{\ell}) - s_{\ell} t_{\ell}, && (y, t) \in \RR^{n-1} \times [-K, 0].
	\end{eqnarray}
   Hence the highest point of $\tilde{U}^{\ell}$ is bounded by $-1$ and $-M_0K$.  By a comparison between $\overline{U}^{s_{\ell}}$ and $\underline{U}_{\underline{s}}$, the lowest point of $\tilde{U}^{\ell}$ is bounded from above by $- \left( s_{\ell} - \underline{s}\right)t_{\ell}$. Furthermore, based on Corollary \ref{the regularity of obstacle solutions}, one can show that the hypersurface $\tilde{U}^{\ell}(y,t)$ is spatially $C^{1,\alpha}$ hypersurface in $\mathbb{T}^{n-1}$, uniformly for $t > 0$. Let us define the set $E_{\ell, K}$, whose measure is neither 0 nor 1, due to $\overline{s}(e_{n}) > \underline{s}(e_{n})$.
   \begin{eqnarray*}
   E_{\ell, K} := \left\lbrace y \in \mathbb{T}^{n-1} \big| \tilde{U}^{\ell}(y,0) > - 2 M_0 K\right\rbrace &\text{and}& E_{K} := \liminf_{\ell\rightarrow\infty} E_{\ell, K}.
   \end{eqnarray*}
   The regularity of the hypersurface along with the fact that $\tilde{U}^{\ell}(y,0)$ have uniformly bounded maximum over $\mathbb{T}^{n-1}$ implies that $E_{M}$ contains a unit neighborhood of some point in $\mathbb{T}^{n-1}$. Let us now define
   \begin{eqnarray*}
   U^{\infty}(y,t) := \lim_{\ell \rightarrow \infty}\sup^{*} \tilde{U}^{\ell}(y,t), && (y, t) \in E_{M} \times [-M, 0].
   \end{eqnarray*}
   The limit is uniform due to Arzela- Ascoli theorem. Now let us define
   \begin{equation*}
   E^{\infty} := \cup_{M > 0} E_{M}.
   \end{equation*}
   By the Birkhoff property (c.f. Proposition \ref{the birkhoff property for subsolution in an expanding domain} and Remark \ref{a remark on the birkhoff property in laminar case}), then
   \begin{eqnarray*}
   \overline{U}^{s_{\ell}}(y, t + k) \leq \overline{U}^{s_{\ell}}(y, t) + s_{\ell}k, &\text{for any}& (y,t) \in \RR^{n-1} \times (0,\infty), \hspace{2mm} k > 0.
   \end{eqnarray*}
   And then
   \begin{eqnarray*}
	\tilde{U}^{\ell}(y, t + k) \leq \tilde{U}^{\ell}(y, t) + s_{\ell}k, &\text{for any}& (y,t) \in \RR^{n-1} \times (0,\infty), \hspace{2mm} k > 0.
   \end{eqnarray*}
   Hence
   \begin{eqnarray*}
   U^{\infty}(y, t + k) \leq U^{\infty}(y, t) + s_{\ell} k, &\text{for any}& (y,t) \in E^{\infty} \times (-\infty, 0), \hspace{2mm} k > 0.
   \end{eqnarray*}
   On the other hand, by the choice of $(s_{\ell}, t_{\ell})$, we have that $\lim_{\ell\rightarrow \infty }(s_{\ell + 1} - s_{\ell})t_{\ell + 1} = 0$. By the ordering relation $\overline{U}^{s_{\ell + 1}} \leq \overline{U}^{s_{\ell}}$ and the Birkhoff property, we have that
   \begin{eqnarray*}
   \overline{U}^{s_{\ell + 1}}(y, t + t_{\ell + 1}) \leq \overline{U}^{s_{\ell}}(y, t + t_{\ell + 1}) \leq \overline{U}^{s_{\ell}}(y, t + t_{\ell} + k) + s_{\ell} (t_{\ell + 1} - t_{\ell}) - s_{\ell}k.
   \end{eqnarray*}
   Or, equivalently,
   \begin{equation*}
   \tilde{U}^{\ell + 1}(y, t) + (s_{\ell + 1} - s_{\ell})t_{\ell + 1} + s_{\ell}k \leq \tilde{U}^{\ell}(y, t + k).
   \end{equation*}
   Sending $\ell \rightarrow \infty$, we get the other inequality
   \begin{eqnarray*}
   U^{\infty}(y,t) + s_{\ell}k \leq U^{\infty}(y, t + k), &\text{for any}& (y, t) \in E^{\infty} \times (-\infty, 0), \hspace{2mm} k > 0.
   \end{eqnarray*}
   Thus we can define $U^{\infty}(y) := U^{\infty}(y,0)$ and have that $U^{\infty}(y,t) = U^{\infty}(y) + \overline{s}t$ is a travelling wave viscosity subsolution over $E^{\infty}$. Shift $U^{\infty}(y)$ by its maximum value if necessary, we have $\max_{y} U^{\infty}(y,0) = 0$, and thus (a) is established. The part (b) follows from the definition of $U^{\infty}(y,t)$. The part (c) is basically a restatment of Proposition \ref{a semi obstacle solution becomes a solution when detachment happens}. Let us now consider the part (d). Fix any $y^{*} \in \partial E^{\infty}$ and let $E^{\infty} \ni y_{i} \rightarrow y^{*}$, then $U^{\infty}(y_{i}) < 0$ as $i \rightarrow \infty$. Plugging $U^{\infty}(y) + \overline{s}t$ into (\ref{the fmcf equation in laminar case}), we have that
   \begin{eqnarray*}
   \frac{\overline{s}}{\sqrt{|DU^{\infty}(y_{i})|^{2} + 1}} - \DIV\left(\frac{DU^{\infty}(y_{i})}{\sqrt{|DU^{\infty}(y_{i})|^{2} + 1}} \right) - g(y) = 0, && (y,t) \in E^{\infty} \times (-\infty, \infty).
   \end{eqnarray*}
   Because $\lim_{i\rightarrow\infty}|DU^{\infty}(y_{i})| = \infty$, let us send $i \rightarrow \infty$ and get (in viscosity sense) $g(y^{*}) = \kappa(y^{*})$, which is the curvature of $\partial E^{\infty}$ at $y^{*}$.
\end{proof}

\begin{proposition}\label{the existence of a travelling wave supersolution and its properties}
	If $\overline{s}(e_{n}) > \underline{s}(e_{n})$, then there is an open, nonempty set $E_{\infty} \subset \mathbb{T}^{n-1}$ and functions $U_{\infty}(y): E_{\infty} \rightarrow [0, \infty)$, such that the following are true:
	\begin{enumerate}
		\item [(a)] The function $U_{\infty}(y) + \underline{s}t$ is a viscosity supersolution of (\ref{the fmcf equation in laminar case});
		\item [(b)] $U_{\infty}(y) \rightarrow +\infty$ as $y \rightarrow \partial E_{\infty}$;
		\item [(c)] The function $U_{\infty}(y) + \underline{s}t$ is a solution of (\ref{the fmcf equation in laminar case}) away from $x_{n} = \underline{s}t$;
		\item [(d)] The set $\partial E_{\infty} \times (-\infty, \infty)$ is a stationary solution of (\ref{the scaled forced mean curvature flow}) with $\varepsilon = 1$.
	\end{enumerate}
\end{proposition}

\begin{proof}
	It is parallel to that of Proposition \ref{the existence of a travelling wave supersolution and its properties}, we omit it here.
\end{proof}

\subsection{More discussion of travelling wave sub/super solution}\label{discussion of construction of travelling waves}

In this part, we investigate the properties of the laminar forcing term $g(y)$ that could induce the failing of the homogenization. i.e., the existence of travelling wave subsolution with head speed $\overline{s}$ and travelling wave supersolution with tail speed $\underline{s}$, such that $0 < \underline{s} < \overline{s} < \infty$. The idea is partially motivated by the example by \cite{Caffarelli and Monneau ARMA}. 

\medskip

Let us consider $0<r_1 < r_2 < \frac{1}{2}$ so that $B(y_1, r_1)$ and $\mathbb{T}^{n-1}\diagdown B(y_2, r_2)$ are two disjoint sets in $\mathbb{T}^{n-1}$ for some $y_1, y_2$.  Define a decreasing function $\zeta: (0,r_1) \to \mathbb{R}$  such that $\zeta(r_1^-) =-\infty$, and an increasing function  $\eta:(r_2, \infty)\to \mathbb{R}$ such that $\eta(r_2^+) = -\infty$ and $\eta(r) \equiv 0$ if $r\geq R$ for some $r_2 < R < \frac{1}{2}$. Let us next choose two sets $E_1 := B(y_1, r_1)$ and $E_2 := \mathbb{T}^{n-1} \setminus B(y_2, r_2)$, and construct travelling wave sub and supersolutions $U_1:E_1\to (-\infty, 0)$ and $U_2: E_2 \to (0, \infty)$ by 
\begin{eqnarray*}
U_{1}(y) := \int_{0}^{r}\zeta(\tau)d\tau, \hspace{2mm} r := |y - y_1| &\text{and}&
U_{2}(y) := \int_{r}^{\infty}\eta(\tau)d\tau, \hspace{2mm} r := |y - y_2|.
\end{eqnarray*}
Then $U_{1}(y) + \overline{s} t$ is a subsolution of (\ref{the fmcf equation in laminar case}) if 
\begin{equation}\label{sub}
g(y) \geq \frac{\overline{s}}{\sqrt{\zeta^{2}(r) + 1}} - \left( \left( \frac{\zeta(r)}{\sqrt{\zeta^{2}(r) + 1}}\right)^{\prime}  + \frac{n-2}{r}\cdot \frac{\zeta(r)}{\sqrt{\zeta^{2}(r) + 1}}\right) \hbox{ with }
	r = |y - y_1| \in [0, r_1).
\end{equation}
Similarly, $U_{2}(y) + \underline{s} t$ is a supersolution of (\ref{the fmcf equation in laminar case}) if
\begin{equation}\label{super}
g(y) \leq \frac{\underline{s}}{\sqrt{\eta^{2}(r) + 1}} - \left( \left( \frac{\eta(r)}{\sqrt{\eta^{2}(r) + 1}}\right)^{\prime}  + \frac{n-2}{r}\cdot \frac{\eta(r)}{\sqrt{\eta^{2}(r) + 1}}\right) \hbox{ with } r = |y - y_2| \in (r_2, R].
\end{equation}
\medskip
Let us choose $\zeta(r) := \frac{r}{r - r_1}$ and $\eta(r) := \min[\frac{r - R}{r - r_2},0]$. Then \eqref{sub} is written as
\begin{eqnarray*}
g(y) \geq \frac{\overline{s}(r_1-r)}{\sqrt{r^{2} + (r_1 - r)^{2}}} + \frac{r_1(r_1 - r)}{\left[ r^{2} + (r_1 - r)^{2}\right]^{\frac{3}{2}}} + \frac{n-2}{\sqrt{r^{2} + (r_1 - r)^{2}}}, && y \in E_{1},
\end{eqnarray*}
which is satisfied if we define $\overline{s}$ as 
\begin{equation}\label{upper}
  \overline{s} := \min_{y\in E_1} g(y)  - \frac{\sqrt{2}n}{r_1}.
\end{equation}
Next, \eqref{super} is written as
$$
g(y) \leq J := \frac{\underline{s}(r - r_2)}{\sqrt{(R - r)^{2} + (r - r_2)^{2}}} - \frac{(R - r_{2})(r - r_2)}{\left[ (R - r)^{2} + (r - r_2)^{2}\right]^{\frac{3}{2}}} + \frac{(n-2)(R - r)}{r\sqrt{(R - r)^{2} + (r - r_2)^{2}}}.
$$
We will show that
\begin{equation}\label{lower}
\underline{s}:= \frac{2}{R-r_2} + \sigma \hbox{ with } \sigma>0
\end{equation} satisfies \eqref{super} if  $\max_{E_2} g(y) < \min\left\lbrace \sigma, n-2\right\rbrace$.
This is because
\begin{equation*}
   J \geq \frac{r - r_2}{R-r_2}\sigma + \frac{R-r}{R - r_{2}}(n-2) > \min\left\lbrace \sigma, n-2\right\rbrace \hbox{ for } r_2<r<R.
\end{equation*}
We have shown the corollary  
\begin{corollary}
Homogenization fails if, for $0 < r_1 < r_2 < R < \frac{1}{2}$, $E_1$ and $E_2$ are disjoint, and there exists  $\sigma>0$ such that 
\begin{equation}
0 < \sigma < \min_{y\in E_1} g(y)  - \left(\frac{\sqrt{2}n}{r_1}+ \frac{2}{R-r_2} \right)\quad \hbox{ and }\quad  \max_{y \in E_2} g(y) < \min\left\lbrace \sigma, n-2\right\rbrace.
\end{equation} 
\end{corollary}

\begin{proof}
It remains to observe that if first condition holds, then $\bar{s}$ and $\underline{s}$ given in \eqref{upper} and \eqref{lower} satisfy $0<\underline{s} < \bar{s}$.
\end{proof}

\section{Appendix}\label{The appendix}
\subsection{Some calculations}
In this subsection, we carry out calculations regarding two functions such that
\begin{eqnarray*}
	\tilde{d}(x) = d(\underbrace{x + r\varphi(x)e_{1}}_{y}) &\text{ for }& x \in U,
\end{eqnarray*}
where $r$ is a constant and $\varphi(x)$ is a positive smooth functions defined in some region $U\subseteq\RR^{n}$. Let us choose $\left\lbrace e_{1},e_{2},\cdots, e_{n} \right\rbrace $ as the orthonormal coordinate system for $\RR^{n}$, fix $x_{0}\in U$ and denote $y_{0} = x_{0} + r\varphi(x_{0})e_{1}$. Moreover, we assume that $D\varphi(x_{0}) = \alpha e_{1} + \beta e_{2}$, where $\alpha, \beta$ are two fixed real numbers. Furthermore, let us also assume the following:
  \begin{enumerate}
  	\item [(i)] $|Dd(y)| = 1$ in a neighborhood of $y_{0}$;
  	\item [(ii)] $\frac{\partial d}{\partial y_{1}}(y_{0}) = -1$ and $\frac{\partial d}{\partial y_{k}}(y_{0}) = 0$, $k \neq 1$;
  	\item [(iii)] $\frac{\partial^{2} d}{\partial y_{1}\partial y_{k}}(y_{0}) = 0$, $k = 1, 2, \cdots, n$.
  \end{enumerate}
Our goal is to compute the term $\Delta_{\infty}\tilde{d} (x_{0})$, i.e., the second derivative of $\tilde{d}$ in the gradient direction of $\tilde{d}$ at the point $x_{0}$. First, we have that
\begin{equation*}
	D\tilde{d}(x_{0}) = \left( -1 - r\alpha\right)e_{1} + \left( -r\beta\right)e_{2}.  
\end{equation*}
Then the normal derivative operator of $\tilde{d}$ at $x_{0}$ writes
\begin{equation*}
	\frac{\partial}{\partial n} = \frac{-1 - r\alpha}{\sqrt{(1+r\alpha)^{2} + \left(r \beta\right) ^{2}}} \frac{\partial}{\partial x_{1}} + \frac{-r\beta}{\sqrt{(1+r\alpha)^{2} + \left( r\beta\right) ^{2}}} \frac{\partial}{\partial x_{2}}.
\end{equation*}
The $1^{\text{st}}$ order derivative in the normal direction is
\begin{eqnarray*}
	\frac{\partial\tilde{d}}{\partial n}
	&=& \frac{-1 - r\alpha}{\sqrt{(1+r\alpha)^{2} + \left(r \beta\right) ^{2}}} \left( \frac{\partial d}{\partial y_{1}} + r\frac{\partial d}{\partial y_{1}} \frac{\partial \varphi}{\partial x_{1}}\right) + \frac{-r\beta}{\sqrt{(1+r\alpha)^{2} + \left( r\beta\right) ^{2}}}\left( \frac{\partial d}{\partial y_{2}} + r\frac{\partial d}{\partial y_{1}} \frac{\partial \varphi}{\partial x_{2}}\right).
\end{eqnarray*}
Then the $2^{\text{nd}}$ order directional derivative becomes
\begin{eqnarray*}
	\frac{\partial}{\partial n}\left(  \frac{\partial\tilde{d}}{\partial n}\right) &=& \frac{-1 - r\alpha}{\sqrt{(1+r\alpha)^{2} + \left(r \beta\right) ^{2}}}\underbrace{\left[ \frac{\partial}{\partial n}\left( \frac{\partial d}{\partial y_{1}} + r\frac{\partial d}{\partial y_{1}} \frac{\partial \varphi}{\partial x_{1}}\right)\right]}_{A} + \frac{-r\beta}{\sqrt{(1+r\alpha)^{2} + \left( r\beta\right) ^{2}}}\underbrace{\left[ \frac{\partial}{\partial n}\left( \frac{\partial d}{\partial y_{2}} + r\frac{\partial d}{\partial y_{1}} \frac{\partial \varphi}{\partial x_{2}}\right) \right]}_{B}.
\end{eqnarray*}

\subsubsection{The term $A$}

The term $A$ is the folllowing.
\begin{eqnarray*}
	A &=& \frac{-1 - r\alpha}{\sqrt{(1+r\alpha)^{2} + \left(r \beta\right) ^{2}}}\underbrace{\left[ \frac{\partial}{\partial x_{1}}\left( \frac{\partial d}{\partial y_{1}} + r\frac{\partial d}{\partial y_{1}} \frac{\partial \varphi}{\partial x_{1}}\right)\right]}_{A_{1}} + \frac{- r\beta}{\sqrt{(1+r\alpha)^{2} + \left(r \beta\right) ^{2}}}\underbrace{\left[ \frac{\partial}{\partial x_{2}}\left( \frac{\partial d}{\partial y_{1}} + r\frac{\partial d}{\partial y_{1}} \frac{\partial \varphi}{\partial x_{1}}\right)\right]}_{A_{2}}, 
\end{eqnarray*}
where 
\begin{eqnarray*}
	A_{1} &=& \frac{\partial}{\partial x_{1}}\left( \frac{\partial d}{\partial y_{1}}\right) + r\frac{\partial}{\partial x_{1}}\left( \frac{\partial d}{\partial y_{1}} \frac{\partial \varphi}{\partial x_{1}}\right)\\
	&=& \frac{\partial^{2} d}{\partial y_{1}^{2}}\left( 1 + r\frac{\partial \varphi}{\partial x_{1}}\right) + r\frac{\partial^{2} d}{\partial y_{1}^{2}}\left( 1 + r\frac{\partial \varphi}{\partial x_{1}}\right)\cdot \frac{\partial \varphi}{\partial x_{1}} + r\frac{\partial d}{\partial y_{1}}\cdot \frac{\partial^{2}\varphi}{\partial x_{1}^{2}}\\
	&=& \frac{\partial^{2} d}{\partial y_{1}^{2}}\left( 1 + r\frac{\partial \varphi}{\partial x_{1}}\right)^{2} + r\frac{\partial d}{\partial y_{1}}\cdot \frac{\partial^{2}\varphi}{\partial x_{1}^{2}} = -r\frac{\partial^{2}\varphi}{\partial x_{1}^{2}}(x_{0})
\end{eqnarray*}
and
\begin{eqnarray*}
	A_{2} &=& \frac{\partial}{\partial x_{2}}\left( \frac{\partial d}{\partial y_{1}}\right)  + r\frac{\partial}{\partial x_{2}}\left( \frac{\partial d}{\partial y_{1}} \frac{\partial \varphi}{\partial x_{1}}\right)\\
	&=& \left( \frac{\partial^{2}d}{\partial y_{2}\partial y_{1}} + \frac{\partial^{2}d}{\partial y_{1}^{2}}\cdot r\frac{\partial\varphi}{\partial x_{2}}\right) + r\left( \frac{\partial^{2}d}{\partial y_{2}\partial y_{1}} + \frac{\partial^{2}d}{\partial y_{1}^{2}}\cdot r\frac{\partial\varphi}{\partial x_{2}}\right)\frac{\partial \varphi}{\partial x_{1}} + r\frac{\partial d}{\partial y_{1}}\frac{\partial^{2}\varphi}{\partial x_{2}\partial x_{1}}\\ 
	&=& \left( \frac{\partial^{2}d}{\partial y_{2}\partial y_{1}} + r\frac{\partial^{2}d}{\partial y_{1}^{2}}\cdot \frac{\partial \varphi}{\partial x_{2}}\right)\left( 1 + r\frac{\partial\varphi}{\partial x_{1}}\right)  +  r\frac{\partial d}{\partial y_{1}}\frac{\partial^{2}\varphi}{\partial x_{2}\partial x_{1}} = -r\frac{\partial^{2}\varphi}{\partial x_{2}\partial x_{1}}(x_{0}).
\end{eqnarray*}
So
\begin{eqnarray*}
	A &=&\frac{\left(1 + r\alpha\right) r}{\sqrt{(1+r\alpha)^{2} + \left(r \beta\right) ^{2}}} \cdot\frac{\partial^{2}\varphi}{\partial x_{1}^{2}}(x_{0}) + \frac{r^{2}\beta}{\sqrt{(1+r\alpha)^{2} + \left(r \beta\right)^{2}}}\cdot\frac{\partial^{2}\varphi}{\partial x_{2}\partial x_{1}}(x_{0}).
\end{eqnarray*}
\subsubsection{The term $B$}
The term $B$ is the folllowing:
\begin{eqnarray*}
	B &=& \frac{-1 - r\alpha}{\sqrt{(1+r\alpha)^{2} + \left(r \beta\right) ^{2}}}\underbrace{\left[ \frac{\partial}{\partial x_{1}}\left( \frac{\partial d}{\partial y_{2}} + r\frac{\partial d}{\partial y_{1}} \frac{\partial \varphi}{\partial x_{2}}\right)\right]}_{B_{1}} + \frac{- r\beta}{\sqrt{(1+r\alpha)^{2} + \left(r \beta\right) ^{2}}}\underbrace{\left[ \frac{\partial}{\partial x_{2}}\left( \frac{\partial d}{\partial y_{2}} + r\frac{\partial d}{\partial y_{1}} \frac{\partial \varphi}{\partial x_{2}}\right)\right]}_{B_{2}}, 
\end{eqnarray*}
where
\begin{eqnarray*}
	B_{1} &=& \frac{\partial}{\partial x_{1}}\left( \frac{\partial d}{\partial y_{2}}\right)  + r\frac{\partial}{\partial x_{1}}\left( \frac{\partial d}{\partial y_{1}} \frac{\partial \varphi}{\partial x_{2}}\right)\\
	&=& \frac{\partial^{2}d}{\partial y_{1}\partial y_{2}}\left( 1 + r\frac{\partial \varphi}{\partial x_{1}}\right) + r\frac{\partial^{2}d}{\partial y_{1}^{2}}\left( 1 + r\frac{\partial \varphi}{\partial x_{1}}\right)\frac{\partial \varphi}{\partial x_{2}} + r\frac{\partial d}{\partial y_{1}}\frac{\partial^{2}\varphi}{\partial x_{1}\partial x_{2}}\\
	&=& \left( \frac{\partial^{2}d}{\partial y_{1}\partial y_{2}} + r\frac{\partial^{2}d}{\partial y_{1}^{2}}\cdot\frac{\partial \varphi}{\partial x_{2}} \right) \left( 1 + r\frac{\partial \varphi}{\partial x_{1}}\right) + r\frac{\partial d}{\partial y_{1}}\frac{\partial^{2}\varphi}{\partial x_{1}\partial x_{2}} = -r\frac{\partial^{2}\varphi}{\partial x_{1}\partial x_{2}}(x_{0})
\end{eqnarray*}
and
\begin{eqnarray*}
	B_{2} &=& \frac{\partial}{\partial x_{2}}\left( \frac{\partial d}{\partial y_{2}}\right)  + r\frac{\partial}{\partial x_{2}}\left( \frac{\partial d}{\partial y_{1}} \cdot\frac{\partial \varphi}{\partial x_{2}}\right)\\
	&=& \left( \frac{\partial^{2}d}{\partial y_{1}\partial y_{2}}\cdot r\frac{\partial \varphi}{\partial x_{2}} + \frac{\partial^{2}d}{\partial y_{2}^{2}}\right) + r\left( \frac{\partial^{2}d}{\partial y_{1}^{2}}\cdot r\frac{\partial \varphi}{\partial x_{2}} + \frac{\partial^{2}d}{\partial y_{2}\partial y_{1}}\right)\cdot\frac{\partial \varphi}{\partial x_{2}} + r\frac{\partial d}{\partial y_{1}}\cdot\frac{\partial^{2}\varphi}{\partial x_{2}^{2}} \\
	&=& \left( \frac{\partial^{2}d}{\partial y_{2}^{2}} + 2r\frac{\partial^{2} d}{\partial y_{1}\partial y_{2}}\cdot \frac{\partial \varphi}{\partial x_{2}} + r^{2}\frac{\partial^{2}d}{\partial y_{1}^{2}}\left(\frac{\partial \varphi}{\partial x_{2}} \right)^{2} \right) + r\frac{\partial d}{\partial y_{1}} \cdot \frac{\partial^{2}\varphi}{\partial x_{2}^{2}}\\
	&=& \frac{\partial^{2}d}{\partial y_{2}^{2}}(y_{0}) - r\frac{\partial^{2}\varphi}{\partial x_{2}^{2}}(x_{0}).
\end{eqnarray*}
So
\begin{eqnarray*}
	B &=& \frac{\left( 1 + r\alpha\right) r}{\sqrt{(1+r\alpha)^{2} + \left(r \beta\right) ^{2}}}\cdot \frac{\partial^{2}\varphi}{\partial x_{1}\partial x_{2}}(x_{0}) + \frac{- r\beta}{\sqrt{(1+r\alpha)^{2} + \left(r \beta\right) ^{2}}}\cdot\left( \frac{\partial^{2}d}{\partial y_{2}^{2}}(y_{0}) - r\frac{\partial^{2}\varphi}{\partial x_{2}^{2}}(x_{0})\right).
\end{eqnarray*}
Finally, we have the conclusion as follows.
\begin{flalign}\label{the fomula of the infinity laplacian term}
\begin{aligned}
\frac{\partial}{\partial n}\left(  \frac{\partial\tilde{d}}{\partial n}\right)(x_{0})
&=& \frac{(r\beta)^{2}}{(1 + r\alpha)^{2} + (r\beta)^{2}}\cdot\frac{\partial^{2}d}{\partial y_{2}^{2}}(y_{0}) - \frac{(1 + r\alpha)^{2}r}{(1 + r\alpha)^{2} + \left( r\beta\right)^{2}}\cdot \frac{\partial^{2}\varphi}{\partial x_{1}^{2}}\\
&-& \frac{2(1 + r\alpha)\left( r\beta\right)r}{(1 + r\alpha)^{2} + \left( r\beta\right)^{2}}\cdot  \frac{\partial^{2}\varphi}{\partial x_{1}\partial x_{2}} - \frac{\left( r\beta\right)^{2}r}{(1 + r\alpha)^{2} + \left( r\beta\right) ^{2} }\cdot \frac{\partial^{2}\varphi}{\partial x_{2}^{2}}.
\end{aligned}
\end{flalign}

\subsection{A comparison principle}
In this subsection, we prove a variant version of the comparison principle regarding a pseudo viscosity subsolution and a pseudo viscosity supersolution. The idea here is partially motivated by \cite{Caffarelli and Monneau ARMA} and \cite{Kim and Kwon}.
\begin{proposition}\label{the comparison principle regarding pseudo viscosity solutions}
	Fix $\nu \in \mathbb{S}^{n-1}$, $x_{0} \in \RR^{n}$, $R > 0$, $0 \leq \alpha < \beta < \infty$. Let $U(x,t), V(x,t): \Upomega(x_{0},R;\nu) \times [\alpha, \beta] \rightarrow \RR$ satisfy the following (i), (ii) and (iii):
	\begin{enumerate}
		\item [(i)] $U(x,t)$ is a pseudo viscosity subsolution (c.f. Definition \ref{pseudo viscosity subsolution});
		\item [(ii)] $V(x,t)$ is a pseudo viscosity supersolution (c.f. Definition \ref{pseudo viscosity supersolution});
		\item [(iii)] $U(x, t) < V(x, t)$, if $(x,t) \in \left( \Upomega(x_{0},R; \nu) \times \left\lbrace \alpha\right\rbrace\right) \cup \left( \partial\Upomega(x_{0},R; \nu) \times [\alpha, \beta]\right)$,
	\end{enumerate}
	then
	\begin{eqnarray*}
		U(x,t) < V(x,t), && (x,t) \in \Upomega(x_{0},R;\nu) \times (\alpha, \beta).
	\end{eqnarray*}
\end{proposition}

\begin{proof}
	Let us assume on the contrary that
	\begin{equation}\label{the contrary assumption in proving comparison principle in the pseudo sense}
	\sup_{(x,t)\in\Upomega\left( x_{0}, R; \nu\right) \times (\alpha, \beta)}\left( U(x,t) - V(x,t)\right) \geq 0,
	\end{equation}
	without loss of generality, we can assume the existence of the $\tau_{0} \in (\alpha, \beta)$, such that $U(x_{0},\tau_{0}) = \mu \geq V(x_{0},\tau_{0})$. Then we consider two characteristic functions.
	\begin{eqnarray*}
		\mathrm{Z}(x,t) := \begin{cases}
			1, & x \in \Upomega\left( x_{0}, R; \nu\right), U(x,t) \geq \mu,\\
			0, & x \in \Upomega\left( x_{0}, R; \nu\right), U(x,t) < \mu,
		\end{cases}
		&&
		\mathrm{W}(x,t) := \begin{cases}
			1, & x \in \Upomega\left( x_{0}, R; \nu\right), V(x,t) > \mu,\\
			0, & x \in \Upomega\left( x_{0}, R; \nu\right), V(x,t) \leq \mu.
		\end{cases}
	\end{eqnarray*}
	Then $\mathrm{Z}(x_{0},\tau_{0}) = 1 > 0 = \mathrm{W}(x_{0}, \tau_{0})$. Let us choose
	\begin{equation*}
	t_{0} := \min\left\lbrace  \alpha \leq t \leq \beta \big| \mathrm{Z}(x, t) = 1 > 0 = \mathrm{W}(x,t) \text{ for some } x \in \Upomega(0, R; \nu) \right\rbrace.
	\end{equation*}
     Then by the assumptions, we have $t_{0} \in (\alpha, \beta)$ and $\mathrm{Z}(x,t_{0}) = 1 > 0 = \mathrm{W}(x,t_{0})$ for some $x \notin \partial\Upomega(0, R; \nu)$. It is well-known that $\mathrm{Z}(x,t)$ (resp. $\mathrm{W}(x,t)$) is a pseudo viscosity subsolution (resp. pseudo viscosity supersolution) of the forced mean curvature flow. Let us also set
	\begin{eqnarray*}
		z(x,t) := e^{-t}\mathrm{Z}(x,t),\hspace{2mm} w(x,t) := e^{-t}\mathrm{W}(x,t), \hspace{2mm} \tilde{\mathscr{F}}(X,p,x,t) := e^{-t} \mathscr{F}\left(e^{t}X, e^{t}p, x \right). 
	\end{eqnarray*}
	Then we have inequalities in the pseudo viscosity sense:
	\begin{eqnarray*}
		z_{t} + z \leq \tilde{\mathscr{F}}\left(D^{2}z, Dz, x, t \right), && (x,t) \in \Upomega\left( x_{0}, R; \nu\right) \times (\alpha, \beta),\\
		w_{t} + w \geq \tilde{\mathscr{F}}\left(D^{2}w, Dw, x, t \right),&& (x,t) \in \Upomega\left( x_{0}, R; \nu\right) \times (\alpha, \beta).
	\end{eqnarray*}
	The assumption (\ref{the contrary assumption in proving comparison principle in the pseudo sense}) indicates that
	\begin{equation*}
	\sup_{(x,t) \in \Upomega\left( x_{0}, R; \nu\right)  \times [\alpha, t_{0}]} \left( z(x,t) - w(x,t)\right) \geq e^{-t_{0}} > 0.
	\end{equation*}
	Next, for any $\varepsilon, \eta > 0$, let us consider the function.
	\begin{eqnarray*}
		\Phi_{\varepsilon,\eta}(x,y,t) := z(x,t) - w(y,t) - \frac{|x - y|^{4}}{4\varepsilon} - \frac{\eta}{t_{0} - t},
	\end{eqnarray*}   
	where $(x, y, t) \in \Upomega\left( x_{0}, R; \nu\right)  \times  \Upomega\left( x_{0}, R; \nu\right)  \times \left(\alpha, t_{0}\right)$. Since $\Phi_{\varepsilon,\eta}(x,y,t)$ is bounded by 1 from above and upper semicontinuous, there exists $(x_{\varepsilon,\eta}, y_{\varepsilon,\eta}, t_{\varepsilon,\eta})$, at which $\Phi_{\varepsilon,\eta}(\cdot,\cdot,\cdot)$ achieves its finite maximum. Clearly, 
	\begin{eqnarray*}
		0 < t_{\varepsilon,\eta} < t_{0} - O\left( \eta\right)  &\text{and}& |x_{\varepsilon,\eta} - y_{\varepsilon,\eta}| = O\left(\varepsilon^{\frac{1}{4}}\right).
	\end{eqnarray*}
	Even though $\Upomega\left( x_{0}, R; \nu\right) $ is unbounded, we can actually find above $x_{\varepsilon, \eta}$ and $y_{\varepsilon, \eta}$ in the bounded domain $\left\lbrace x\in\Upomega\left( x_{0}, R; \nu\right)  \big| 0 \leq (x - x_{0}) \cdot \nu \leq M_0T\right\rbrace$. Up to extracting a subsequence, we have as $(\varepsilon, \eta) \rightarrow (0, 0)$
	\begin{equation*}
	\Phi_{\varepsilon,\eta}(x_{\varepsilon,\eta}, y_{\varepsilon,\eta}, t_{\varepsilon,\eta}) \rightarrow e^{-t_{0}} > 0.
	\end{equation*}
	Then the fact
	\begin{eqnarray*}
		\Phi_{2\varepsilon,\frac{\eta}{2}}(x_{2\varepsilon, \frac{\eta}{2}}, y_{2\varepsilon, \frac{\eta}{2}}, t_{2\varepsilon, \frac{\eta}{2}}) &\geq& \Phi_{2\varepsilon, \frac{\eta}{2}}(x_{\varepsilon,\eta}, y_{\varepsilon, \eta}, t_{\varepsilon,\eta})\\
		&=& \Phi_{\varepsilon, \eta}(x_{\varepsilon,\eta}, y_{\varepsilon, \eta}, t_{\varepsilon,\eta}) + \frac{|x_{\varepsilon, \eta} - y_{\varepsilon, \eta}|^{4}}{8\varepsilon} + \frac{\eta}{2\left( t_{0} - t_{\varepsilon, \eta}\right) }
	\end{eqnarray*}
	implies that (up to a subsequence)
	\begin{eqnarray*}
		\lim_{(\varepsilon, \eta) \rightarrow (0,0)} \frac{|x_{\varepsilon,\eta} - y_{\varepsilon,\eta}|^{4}}{\varepsilon} = 0, && \lim_{(\varepsilon,\eta) \rightarrow (0,0)}\frac{\eta}{t_{0} - t_{\varepsilon,\eta}} = 0.
	\end{eqnarray*}
	Then we have that
	\begin{eqnarray*}
		x_{\varepsilon, \eta} \neq y_{\varepsilon,\eta}, \hspace{2mm} z(x_{\varepsilon, \eta}, t_{\varepsilon,\eta}) = e^{-t_{\varepsilon,\eta}} &\text{and}& w(y_{\varepsilon, \eta}, t_{\varepsilon,\eta}) = 0.
	\end{eqnarray*} 
	If for a subsequence of $(\varepsilon, \eta) \rightarrow (0,0)$, $(x_{\varepsilon,\eta}, y_{\varepsilon,\eta}) \in \left( \partial\Upomega\left( x_{0}, R; \nu\right)  \times \Upomega\left( x_{0}, R; \nu\right) \right) \cup \left( \Upomega\left( x_{0}, R; \nu\right)  \times \partial\Upomega\left( x_{0}, R; \nu\right) \right) $, then we have a contradiction as follows:
	\begin{eqnarray*}
		e^{-t_{0}} &\leq& \sup_{(x,t) \in \Upomega\left( x_{0}, R; \nu\right)  \times [ \alpha, t_{0}]}\left( z(x,t) - w(x,t)\right) \\
		&\leq& \lim_{\varepsilon\rightarrow 0}\sup_{(x,t) \in \Upomega\left( x_{0}, R; \nu\right)  \times [\alpha, t_{0}]}\left( z(x,t) - w(y,t) - \frac{|x - y|^{4}}{4\varepsilon}\right)\\
		&\leq& \lim_{\varepsilon\rightarrow 0}\lim_{\eta\rightarrow 0}\Phi_{\varepsilon, \eta}(x_{\varepsilon,\eta}, y_{\varepsilon, \eta}, t_{\varepsilon,\eta})\\
		&\leq& \sup_{(x,t) \in \left( \Upomega\left( x_{0}, R; \nu\right)  \times \left\lbrace \alpha\right\rbrace\right) \cup \partial\Upomega\left( x_{0}, R; \nu\right)  \times [\alpha, t_{0}]}\left( z(x,t) - w(x,t)\right) \leq 0. 
	\end{eqnarray*}
	If both of $x_{\varepsilon,\eta}$ and $y_{\varepsilon,\eta}$ are interior points of $\Upomega\left( x_{0}, R; \nu\right) $, we can derive two viscosity inequalities associated to $z(x,t)$ and $w(x,t)$, respectively. Let us denote
	\begin{equation*}
	\Psi(x, y, t) := \frac{|x - y|^{4}}{4\varepsilon} + \frac{\eta}{t_{0} - t}.
	\end{equation*}
	By the Theorem 8.3 in \cite{Crandall Ishii Lions BAMS}, we deduce that for any $\lambda > 0$, there exists
	\begin{equation*}
	\left( b_{1}, D_{x}\Psi(x_{\varepsilon,\eta}, y_{\varepsilon, \eta}, t_{\varepsilon, \eta}), X\right) \in \overline{\mathscr{P}}^{2,+} z(x_{\varepsilon,\eta}, t_{\varepsilon,\eta}),
	\end{equation*}
	and
	\begin{equation*}
	\left( b_{2}, -D_{y}\Psi(x_{\varepsilon,\eta}, y_{\varepsilon, \eta}, t_{\varepsilon, \eta}), Y\right) \in \overline{\mathscr{P}}^{2,-} w(y_{\varepsilon,\eta}, t_{\varepsilon,\eta}),
	\end{equation*}  
	such that
	\begin{equation*}
	b_{1} - b_{2} = \Psi_{t}(x_{\varepsilon,\eta}, y_{\varepsilon, \eta}, t_{\varepsilon, \eta}),
	\end{equation*}  
	and
	\begin{equation}\label{matrix inequality}
	-\left( \frac{1}{\lambda} + \lVert A\rVert\right) \begin{pmatrix}
	I & 0\\
	0 & I
	\end{pmatrix} \leq \begin{pmatrix}
	X & 0 \\
	0 & -Y
	\end{pmatrix} \leq A + \lambda A^{2}.
	\end{equation}
	where $A = D^{2}\Psi(x_{\varepsilon, \eta}, y_{\varepsilon, \eta}, t_{\varepsilon, \eta})$ and $\lVert A \rVert := \sup_{|\xi| = 1}\left\langle A\xi, \xi\right\rangle $. Because $x_{\varepsilon,\eta} \neq y_{\varepsilon,\eta}$, we have that $|D_{x}\Psi(x_{\varepsilon,\eta}, y_{\varepsilon,\eta}| > 0$ and $|D_{y}\Psi(x_{\varepsilon,\eta}, y_{\varepsilon,\eta}| > 0$. Since $z(x,t)$ is a pseudo viscosity subsolution and $w(x,t)$ is a pseudo viscosity supersolution, then the following viscosity inequalities appear.
	\begin{eqnarray}\label{a viscosity inequality to subsolution in proving pseudo comparison principle}
	b_{1} + z(x_{\varepsilon,\eta}, t_{\varepsilon,\eta}) &\leq& \tilde{\mathscr{F}}\left( X, D_{x}\Psi(x_{\varepsilon,\eta}, y_{\varepsilon,\eta}, t_{\varepsilon,\eta}), x_{\varepsilon,\eta}, t_{\varepsilon,\eta} \right),
	\end{eqnarray}
	\begin{eqnarray}\label{a viscosity inequality to supersolution in proving pseudo comparison principle}
	b_{2} + w(x_{\varepsilon,\eta}, t_{\varepsilon,\eta}) &\geq& \tilde{\mathscr{F}}\left( Y, -D_{y}\Psi(x_{\varepsilon,\eta}, y_{\varepsilon,\eta}, t_{\varepsilon,\eta}), y_{\varepsilon,\eta}, t_{\varepsilon,\eta} \right).  
	\end{eqnarray}
	Note that in order to apply Theorem 8.3 in \cite{Crandall Ishii Lions BAMS}, we need to have the boundedness assumption of $b_{1} \leq C$ and $b_{2} \geq -C$, for general $(b_{1}, p_{1}, X) \in \overline{\mathscr{P}}^{2,+}z(x,t)$ and $(b_{2}, p_{2}, Y) \in \overline{\mathscr{P}}^{2,-}w(y,t)$ for $(x,y,t)$ close to $(x_{\varepsilon,\eta}, y_{\varepsilon,\eta}, t_{\varepsilon,\eta})$, and bounded $p_{1}$, $p_{2}$, $X$, $Y$, $z(x,t)$, $w(y,t)$. This follows naturally from viscosity inequalities similar to (\ref{a viscosity inequality to subsolution in proving pseudo comparison principle}) and (\ref{a viscosity inequality to supersolution in proving pseudo comparison principle}). Then when $\varepsilon$, $\eta$ are sufficiently small, taking the difference of the above two inequalities gives
	\begin{eqnarray*}
		0 < \frac{e^{-t_{0}}}{2} \leq \frac{\eta}{(t_{0} - t_{\varepsilon,\eta})^{2}} + z(x_{\varepsilon,\eta}, t_{\varepsilon,\eta}) - w(y_{\varepsilon,\eta}, t_{\varepsilon,\eta})\leq \tilde{\mathscr{F}}\left( X, p, x_{\varepsilon,\eta}, t_{\varepsilon,\eta} \right) - \tilde{\mathscr{F}}\left( Y, p, y_{\varepsilon,\eta}, t_{\varepsilon,\eta} \right)  
	\end{eqnarray*}
	where
	\begin{eqnarray*}
		p = \delta\left( x_{\varepsilon,\eta} - y_{\varepsilon,\eta}\right), && \delta = \frac{|x_{\varepsilon,\eta} - y_{\varepsilon,\eta}|^{2}}{\varepsilon} 
	\end{eqnarray*}  
	Since we always have $x_{\varepsilon,\eta} \neq y_{\varepsilon,\eta}$, by setting $\hat{p} := \frac{p}{|p|}$, we have that (Since $E^{2} = 2E$, we have that $A^{2} \leq 18\delta^{2} E$)
	\begin{equation*}
	0 \leq A = \delta \begin{pmatrix}
	I + 2\hat{p}\otimes\hat{p} & -I - 2\hat{p}\otimes\hat{p}\\
	-I - 2\hat{p}\otimes\hat{p} & I + 2\hat{p}\otimes\hat{p}
	\end{pmatrix} \leq 3\delta E \hspace{2mm}\text{with}\hspace{2mm} E = \begin{pmatrix}
	I & -I\\
	-I & I
	\end{pmatrix}.
	\end{equation*}
	Because $\lVert A \rVert = 6\delta$, by setting $\lambda = \frac{1}{3\delta}$ in (\ref{matrix inequality}), we get that
	\begin{equation*}
	-9\delta \begin{pmatrix}
	I & 0\\
	0 & I
	\end{pmatrix} \leq \begin{pmatrix}
	X & 0\\
	0 & -Y
	\end{pmatrix} \leq 9\delta \begin{pmatrix}
	I & -I\\
	-I & I
	\end{pmatrix}.
	\end{equation*}  
	For any $\xi \in \RR^{n}$, let us multiply the above inequalities on the left by $(\xi, \xi)$ and on the right by $(\xi, \xi)^{T}$, then
	\begin{eqnarray*}
	\xi (X - Y)\xi^{T} \leq 0, &\text{i.e.,}& X \leq Y.
	\end{eqnarray*}
	Finally, we get the following contradiction, where $L_0$ is from (\ref{the hypothesis}).
	\begin{eqnarray*}
		0 < \frac{e^{-t_{0}}}{2} &\leq& \tilde{\mathscr{F}}\left( X, p, x_{\varepsilon,\eta}, t_{\varepsilon,\eta} \right) - \tilde{\mathscr{F}}\left( Y, p, y_{\varepsilon,\eta}, t_{\varepsilon,\eta} \right) \\
		&=& e^{-t_{\varepsilon,\eta}}\mathscr{F}\left( e^{t_{\varepsilon,\eta}}X, e^{t_{\varepsilon,\eta}}p, x_{\varepsilon,\eta}\right) -  e^{-t_{\varepsilon,\eta}}\mathscr{F}\left( e^{t_{\varepsilon,\eta}}Y, e^{t_{\varepsilon,\eta}}p, y_{\varepsilon,\eta}\right)\\
		&=& \tr \left\lbrace \left( X - Y\right) \left( I - \hat{p}\otimes\hat{p}\right)\right\rbrace + \left( g(x_{\varepsilon,\eta}) - g(y_{\varepsilon,\eta})\right)|p| \\
		&\leq& 0 + L_0|x_{\varepsilon,\eta} - y_{\varepsilon,\eta}| \rightarrow 0, \hspace{1cm} \text{as} \hspace{4mm} (\varepsilon,\eta) \rightarrow (0, 0).
	\end{eqnarray*}
	Therefore, we must have $U(x,t) \leq V(x,t)$ for any $(x,t) \in \Upomega\left( x_{0}, R; \nu\right)  \times (\alpha, \beta)$.

\end{proof}

\bibliographystyle{amsplain}

\end{document}